\documentclass[11pt]{amsart}

\usepackage{amssymb,amsfonts,theoremref,mathrsfs}
\usepackage{hyperref}
\usepackage{mdwlist}

\usepackage{amssymb,textcomp}
\usepackage{enumerate}

\usepackage[utf8]{inputenc}
\usepackage[T1]{fontenc}

\usepackage[mathscr]{eucal}
\usepackage{xcolor}

\usepackage{tikz}
\tikzstyle{startstop} = [rectangle, rounded corners, minimum width=3cm, minimum height=1cm,text centered, draw=black, fill=red!30]

\usepackage{hyperref}
 \hypersetup{
     colorlinks=true,
     linktocpage=true,
     linkcolor=red,
     filecolor=blue,
     citecolor = blue,
     urlcolor=cyan,
     }

\usepackage[a4paper, twoside=false, vmargin={2cm,3cm}, includehead]{geometry}

\usepackage{comment}
\newtheorem{lemma}{Lemma}[section]
\newtheorem{theorem}[lemma]{Theorem}
\newtheorem*{theorem*}{Theorem}
\newtheorem{corollary}[lemma]{Corollary}

\newtheorem{proposition}[lemma]{Proposition}
\newtheorem*{proposition*}{Proposition}

\newtheorem*{problem*}{Problem}

\theoremstyle{definition}
\newtheorem*{claim*}{Claim}

\newtheorem{definition}[lemma]{Definition}

\newtheorem*{remark}{Remark}

\DeclareMathOperator*{\E}{\mathbb{E}}

\newcommand{\C}{{\mathbb C}}

\newcommand{\F}{{\mathbb F}}
\newcommand{\N}{{\mathbb N}}

\newcommand{\R}{{\mathbb R}}

\newcommand{\Z}{{\mathbb Z}}

\newcommand{\CC}{{\mathcal C}}
\newcommand{\CD}{{\mathcal D}}

\newcommand{\CV}{{\mathcal V}}

\newcommand{\be}{{\mathbf{e}}}

\newcommand{\bh}{{\mathbf{h}}}

\newcommand{\bu}{{\textbf{u}}}
\newcommand{\bv}{{\textbf{v}}}
\newcommand{\x}{{\textbf{x}}}
\newcommand{\bx}{{\textbf{x}}}

\newcommand{\bbeta}{{\boldsymbol{\beta}}}

\newcommand{\uh}{{\underline{h}}}

\newcommand{\ubh}{{\underline{\textbf{h}}}}

\newcommand{\eps}{\varepsilon}
\newcommand{\ueps}{{\underline{\epsilon}}}

\newcommand{\norm}[1]{\left\Vert #1\right\Vert}

\DeclareMathOperator{\poly}{poly}

\newcommand{\veps}{\varepsilon}

\newcommand{\model}{{\mathrm{Model}}}

\DeclareMathOperator{\corners}{Corners}

\newcommand{\abs}[1]{\mathopen{}\left| #1\mathclose{}\right|}

\newcommand{\Bigabs}[1]{\Bigl| #1 \Bigr|}
\newcommand{\brac}[1]{\mathopen{}\left( #1 \mathclose{}\right)}

\newcommand{\Bigbrac}[1]{\Bigl( #1 \Bigr)}
\newcommand{\floor}[1]{\left \lfloor #1 \right \rfloor}
\newcommand{\ceil}[1]{\left \lceil #1 \right \rceil}

\title[]{Corners with polynomial side length}
\author{Noah Kravitz, Borys Kuca, James Leng}
\date{}

\address[Noah Kravitz]{Department of Mathematics, Princeton University, Princeton, NJ 08540, USA}
\email{nkravitz@princeton.edu}

\address[Borys Kuca]{Faculty of Mathematics and Computer Science, Jagiellonian University, 30-348 Krak\'ow, Poland}
\email{borys.kuca@uj.edu.pl}

\address[James Leng]{Department of Mathematics, UCLA, Los Angeles, CA 90095, USA}
\email{jamesleng@math.ucla.edu}

\begin{document}

\maketitle

\begin{abstract}
A \emph{$P$-polynomial corner}, for $P \in \mathbb{Z}[z]$ a polynomial, is a triple of points $(x,y),\; (x+P(z),y),\; (x,y+P(z))$ for $x,y,z \in \mathbb{Z}$.  In the case where $P$ has an integer root of multiplicity $1$, we show that if $A \subseteq [N]^2$ does not contain any nontrivial $P$-polynomial corners, then $$|A| \ll_P \frac{N^2}{(\log\log\log N)^c}$$ for some absolute constant $c>0$.  This simultaneously generalizes a result of Shkredov about corner-free sets and a recent result of Peluse, Sah, and Sawhney about sets without $3$-term arithmetic progressions of common difference $z^2-1$.  The main ingredients in our proof are a multidimensional quantitative concatenation result from our companion paper \cite{KKL24a} and a novel degree-lowering argument for box norms.
\end{abstract}

\section{Introduction}

In recent years, there has been a surge of interest in obtaining ``reasonable'' upper bounds for the Polynomial Szemer\'edi Theorem of Bergelson and Leibman \cite{BL96}. 
This theorem asserts that if $\mathbf{P}_1, \ldots, \mathbf{P}_\ell \in \Z^D[z]$ are polynomials satisfying some mild local conditions, then every subset of $\Z^D$ with upper positive Banach density must contain a (nontrivial) occurrence of the polynomial progression 
$$\bx,\; \bx+\mathbf{P}_1(z),\; \ldots,\; \bx+\mathbf{P}_\ell(z).$$
This result provided a unified proof for several earlier results in arithmetic Ramsey theory over $\Z$, including Roth's pioneering theorem \cite{R53} on sets avoiding $3$-term arithmetic progressions, Szemer\'edi's later generalization \cite{Sz75} to sets avoiding $\ell$-term arithmetic progressions for arbitrary fixed $\ell \geq 3$, and work of S\'ark\"ozy \cite{Sa78a,Sa78b} and Furstenberg \cite{Fu77} (independently) on sets avoiding length-$2$ polynomial progressions.

The Polynomial Szemer\'edi Theorem, despite its useful generality, comes with a drawback: Because the proof uses methods from ergodic theory, the result is ineffective and provides no  bounds on sets avoiding polynomial progressions.  Roth and S\'ark\"ozy proved their theorems with fairly good bounds.  Szemer\'edi provided bounds of tower-type which, although strictly-speaking ``effective'', are not particularly ``reasonable''.  The work of Gowers on obtaining improved bounds for Szemer\'edi's Theorem \cite{Go98a, Go01} gave birth to the field of higher-order Fourier analysis, and since then there has been substantial work on obtaining quantitative versions of various instances of the Polynomial Szemer\'edi Theorem.

Of particular interest is an early result due to Shkredov \cite{Sh06b, Sh06a}, who obtained reasonable bounds for subsets of $[N]^2$ avoiding the ``corner'' configuration
$$(x,y),\; (x+z,y),\; (x,y+z).$$
His argument involved a suitable notion of ``pseudorandom sets'' with which one can run a modification of Roth's original density-increment strategy.  Shkredov's approach was later simplified by Green \cite{G05b, G05} for corner-free subsets of $\mathbb{F}_p^D \times \mathbb{F}_p^D$; see also the work of Lacey and McClain \cite{LM07} for an improvement in $\mathbb{F}_2^D \times \mathbb{F}_2^D$.  Peluse's recent work on subsets of $\F_p^D\times\F_p^D$ avoiding the ``$L$-shaped configuration'' $(x, y),\; (x + z, y),\; (x, y + z),\; (x, y + 2z)$ is an extension of Shkredov's approach.

Until recently, Shkredov's work on corners was the only instance of a ``reasonable bound'' for a multidimensional configuration over the integers. Our main result is a polynomial extension of Shkredov's theorem.
\begin{theorem}\label{maintheorem}
Let $P \in \mathbb{Z}[z]$ be a polynomial with an integer root of multiplicity $1$. If $A \subseteq [N]^2$ does not contain a configuration of the form $(x, y),\; (x + P(z), y),\; (x, y + P(z))$ with $P(z) \neq 0$, then
$$|A| \ll_P \frac{N^2}{(\log\log\log N)^c}$$
for some absolute constant $c > 0$.
\end{theorem}

Theorem \ref{maintheorem} and the concurrently released work of Peluse, Prendiville and Shao \cite{PPS24} on the configuration $(x, y),\; (x + z,\; y),\; (x,\; y + z^2)$ provide the first examples of reasonable bounds for multidimensional polynomial (i.e., non-linear) progressions over the integers.

Using a calculation of Shkredov \cite{Sh06b}, we can take any $c<1/73$. The reader may notice that our statement has an artificial-looking constraint on the polynomial $P$.  This constraint, which is slightly stronger than intersectivity, arises from a technical point related to the use of the $W$-trick, as discussed later.

A key analytic step in our argument is the following inverse theorem for polynomial corners, which may be of independent interest. 
\begin{theorem}[{Inverse theorem for polynomial corners}]\label{thm:non-w-tricked-degree lowering-lambda}
Let $d\in\N$, let $P\in\Z[z]$ be a degree-$d$ polynomial with leading coefficient $\beta_d$, let $\delta\in(0, 1/10)$, and let $f_0, f_1, f_2: \Z^2 \to \mathbb{C}$ be $1$-bounded functions supported on $[N]^2$ for some integer $N \geq \exp(\log(1/\delta)^{\Omega_P(1)})$. If
$$\abs{\sum_{x,y}\E_{z\in[(N/\beta_d)^{1/d}]}f_0(x,y)f_1(x+P(z), y) f_2(x, y + P(z))} \geq \delta N^2,$$
then
\begin{align*}
    &\sum_{x,y}f_0(x,y)b_1(x)b_2(y) \geq \exp(-\log(1/\delta)^{O_P(1)}) N^2,\quad \text{for some $1$-bounded $b_1,b_2: \Z \to \C$},\\
    &\sum_{x,y}f_1(x,y)b_1(x+y)b_2(y) \geq \exp(-\log(1/\delta)^{O_P(1)}) N^2\quad \text{for some $1$-bounded $b_1,b_2: \Z \to \C$},\\
    &\sum_{x,y}f_2(x,y)b_1(x)b_2(x+y) \geq \exp(-\log(1/\delta)^{O_P(1)}) N^2\quad \text{for some $1$-bounded $b_1,b_2: \Z \to \C$}.
\end{align*}
\end{theorem}
In the proof of Theorem \ref{maintheorem}, we will end up using a slightly different version of this inverse theorem that is better suited to $W$-tricks (Theorem~\ref{thm:degree lowering-lambda}); the proof simplifies in several ways for the $W$-trick-free version stated here.  Theorem~\ref{thm:degree lowering-lambda} leads to a ``comparison theorem'' with a ``model'' counting operator (Theorem \ref{P: Lambda*}), which in Section~\ref{S: finishing} we will use to deduce the following improvement on the result of Peluse, Sah, and Sawhney \cite{PSS23}.

\begin{corollary}\label{cor: PSS improvement}
Let $P \in \mathbb{Z}[z]$ be a polynomial with an integer root of multiplicity $1$. If $A \subseteq [N]$ does not contain a configuration of the form $x,\; x + P(z),\; x + 2P(z)$ with $P(z) \neq 0$, then
$$|A| \ll \frac{N}{\exp((\log\log N)^{c_P})}$$
for some constant $c_P > 0$.
\end{corollary}

Before we say more about our main theorem, let us briefly survey related results.  Previous work has focused on $1$-dimensional polynomial progressions over the integers and both $1$-dimensional and multidimensional polynomial progressions in finite-field settings.  One of the many important differences between the integers and finite fields is that for a fixed polynomial $P \in \Z[z]$ of degree at least $2$, the image of $P$ is an equidistributed set of positive density in (a sequence of) finite fields, while in the integers it is a sparse set with local biases; since there are more ``opportunities'' to find polynomial progressions in finite fields $\mathbb{F}_p$, it is generally easier to prove bounds on the sizes of sets avoiding polynomial progressions in finite-field settings, and much more is known here than over the integers.

Early results for $1$-dimensional polynomial progressions in $\mathbb{F}_p$ ($p$ a prime) are due to Bourgain and Chang \cite{BC17}, Dong, Li, and Sawin \cite{DLS17}, and Peluse \cite{Pel18}.  In a major breakthrough, Peluse \cite{Pel19} proved power-saving upper bounds for subsets of $\mathbb{F}_p$ avoiding the polynomial progression
$$x,\; x+P_1(z),\; \ldots,\; x+P_\ell(z)$$
for any given linearly independent polynomials $P_1, \ldots, P_m \in \Z[z]$ with constant coefficients all zero.  Her proof introduced the now-well-used technique of degree-lowering for Gowers norms.  Peluse's work has been generalized in several ways.  For instance, the second author \cite{Kuc21a} proved reasonable upper bounds for subsets of $\mathbb{F}_p$ avoiding a given progression of the form
$$(x,\; x + z,\; \dots,\; x + (\ell - 1)z,\; x + z^\ell,\;\dots,\; x + z^{\ell + k - 1}).$$
In \cite{Leng22}, the third author used a $U^3$-version of degree-lowering to obtain reasonable upper bounds for a large family of so-called ``complexity-$1$'' polynomial configurations. Lastly, Bergelson and Best \cite{BB23} generalized Peluse's result to general finite commutative rings.

Less is known about multidimensional progressions in finite fields.  Regarding linear configurations, we have already mentioned results on corners and L-shaped configurations.  In the higher-degree case, the second author \cite{Kuc23, Kuc21} proved a general Szemer\'edi-type theorem for patterns of the form 
$$\bx,\; \bx + \bv_1 P_1(z),\; \dots,\; \bx + \bv_\ell P_\ell(z),$$ where $P_1, \dots, P_\ell$ are linearly independent polynomials with constant coefficients all zero and $\bv_1, \ldots, \bv_\ell$ are arbitrary vectors in $\mathbb{F}_p^D$. When $\ell=2$, special cases of this result had been established earlier by Han, Lacey, and Yang \cite{HLY21}.

Let us now turn to polynomial progressions in $\Z$.  In a sequence of papers \cite{Pel19, PP19, PP20}, Peluse and Prendiville adapted Peluse's original degree-lowering argument for progressions of the form $x,\; x+P_1(z),\; \ldots,\; x+P_m(z)$ where the polynomials $P_i$ have distinct degrees.  Prendiville \cite{Pre17} also proved reasonable bounds for sets avoiding arithmetic progressions with common difference of the form $z^d$ for fixed $d$; the homogeneity of the polynomial $z^d$ allowed him to use a density-increment argument.  Peluse, Sah, and Sawhney \cite{PSS23}, adapting a $U^3$-degree-lowering argument developed in \cite{Leng22}, addressed the case of three-term arithmetic progressions with common difference $z^2-1$.  They proved a ``transference'' result relating the number of $3$-term arithmetic progressions to the number of $3$-term arithmetic progressions with common difference $z^2-1$.

By a simple projection argument (see the proof of Corollary~\ref{cor: PSS improvement} below), Theorem \ref{maintheorem} contains the setting studied by Peluse, Sah, and Sawhney as a special case. Hence, our main theorem can be understood as a common generalization of Shkredov's corners theorem and the result of Peluse, Sah, and Sawhney.  To prove Theorem~\ref{maintheorem}, we will establish a transference result for corners, in the style of \cite{PSS23}, and then use Shkredov's work as a black box.  The starting point for our argument is a PET and quantitative concatenation result carried out in our companion paper \cite{KKL24a}.  The bulk of the present paper is devoted to obtaining degree-lowering for the polynomial corners configuration.  Our degree-lowering argument is the first to use box norms instead of Gowers norms, and its success attests to the robustness of degree-lowering as a tool in arithmetic combinatorics.

\subsection{Future directions}
Our results and methods suggest several avenues for future inquiry.  The most natural open problem is extending Theorem~\ref{maintheorem} to a larger class of polynomials $P$.  A particularly appealing example, recently highlighted by Peluse \cite[Problem 3.17]{Pel23}, is $P(z)=z^2$.  (Theorem~\ref{maintheorem} fails to cover $P(z)=z^2$ because $0$, the only integer root, has multiplicity $2>1$.)  One strategy, which we attempted initially before turning to transference arguments, is obtaining $U^1 \times U^1$-control for the $P$-polynomial corner counting operator and then running Shkredov's original density-increment argument for corners.  Our degree-lowering argument goes through with $\widetilde P$ replaced by an arbitrary (non-$W$-tricked) polynomial $P \in \mathbb{Z}[z]$, which gives the norm-control necessary for starting Shkredov's iterative argument.  The pitfall of this approach, however, is that in later stages of the iteration one must deal with subsets of pseudorandom sets rather than subsets of $[N]^2$, and this creates complications in both the arguments of this paper and the PET and concatenation results of our companion paper \cite{KKL24a}.  It seems possible that one could generalize our arguments to the setting of pseudorandom sets as in \cite{Sh06b, Sh06a}, but we have not pursued this.

The polynomial progressions investigated in this paper belong to the class of linear patterns (in our case, corners) with polynomial differences. Other works on obtaining bounds for such configurations have proceeded either by modifying arguments for the corresponding linear patterns (as in, e.g., \cite{Pre17}) or by using transference results to relate the counts of linear and polynomial patterns (as in the present work and \cite{Kuc21a, PSS23}).  Reasonable bounds are known for $L$-shaped configurations and arithmetic progressions of length $\ell\geq 4$, but their polynomial analogues are currently unknown, so it would be interesting to extend our methods to establish transference results for these two particular configurations.  It would also be interesting to use this machinery to obtain more general transference principles, even for cases where there are not currently reasonable bounds for linear configurations.

\subsection{Acknowledgments} NK was supported in part by the NSF Graduate Research Fellowship Program under grant DGE–203965. BK is supported by the NCN Polonez Bis 3 grant No. 2022/47/P/ST1/00854 (H2020 MSCA GA No. 945339). For the initial stages of the project, BK was supported by the ELIDEK grant No: 1684.  JL is suppored by the NSF Graduate Research Fellowship Grant No. DGE-2034835. 

For the purpose of Open Access, the authors have applied a CC-BY public copyright licence to any Author Accepted Manuscript (AAM) version arising from this submission.

We would like to thank Terry Tao, Sarah Peluse, and Mehtaab Sawhney for helpful discussions and for comments on a draft of this paper.

\section{Proof strategy}\label{S: proof strategy}

We now give an overview of the proof of Theorem~\ref{maintheorem}. Our overarching goal is to relate the number of polynomial corners to the number of ordinary corners; once this is accomplished, we can ``transfer'' Shkredov's result (in a standard supersaturation version) from ordinary corners to polynomial corners.

Let $A \subseteq [N]^2$ be a set of density $\delta$, and let $P \in \mathbb{Z}[z]$ be a polynomial of degree $d$ with an integer root $\rho$ of multiplicity $1$. Then $$P(z-\rho)=\beta_dz^d+\cdots+\beta_1 z$$ for some $\beta_1, \ldots,\; \beta_d\in\Z$ with $\beta_1,\; \beta_d \neq 0$; without loss of generality we may assume that $\beta_d>0$.  The number of corners in $A$ is $N\Lambda^{\corners}(1_A,1_A,1_A)$, where $\Lambda^{\corners}$ is the normalized {corner-counting operator} at scale $N$ given by
$$\Lambda^{\corners}(f_0,f_1,f_2):=\sum_{x,y \in \Z} \E_{z \in [N]}f_0(x,y)f_1(x+z,y),f_2(x,y+z).$$
The number of $P$-polynomial corners in $A$ is $$(N/\beta_d)^{1/d}\Lambda(1_A,1_A,1_A),$$
where $\Lambda$ is the {normalized $P$-polynomial corner-counting operator} at scale $N$ given by
$$\Lambda(f_0,f_1,f_2):=\sum_{x,y \in \mathbb{Z}}\E_{z \in [(N/\beta_d)^{1/d}]} f_0(x,y)f_1(x+P(z),y)f_2(x,y+P(z)).$$
In order to show that $|\Lambda^{\corners}-\Lambda|$ is small, we must make both an Archimedean and a non-Archimedean modification.

First, since the set $\{P(z): z \in \N\}$ becomes sparser at larger scales, we introduce a similarly-decaying weight
$$\nu(z):=d^{-1} 1_{[N]}(z) (N/(z+1))^{(d-1)/d}.$$
into $\Lambda^{\corners}$ to obtain the {model corner-counting operator}
$$\Lambda^{\model}(f_0,f_1,f_2):=\sum_{x,y\in \mathbb{Z}}  \E_{z \in [N]} f_0(x,y)f_1(x+z,y)f_2(x,y+z)\nu(z).$$
Second, since the set $\{P(z): z \in \N\}$ may be biased with respect to small moduli, we ``pre-sieve'' the set $A$ before attempting to compare the counts of corners and of $P$-polynomial corners.  More precisely, for $W$ a very slowly-growing function of $N$, we restrict our attention to the elements of $A$ lying in a single coset of $(\beta_1^2 W\Z)^2$; by the pigeonhole principle, there is some such coset $(b_1+\beta_1^2 W\Z) \times (b_2+\beta_1^2 W\Z)$ in which $A$ has relative density at least $\delta$ (at scale $N$).  Let
$$A':=\{(x',y') \in \Z^2: (\beta_1^2 W x'+b_1,\; \beta_1^2 W y'+b_2) \in A\} \subseteq [N/(\beta_1^2 W)]^2$$
be this ``slice'' of $A$.  We claim that in order to find a $P$-polynomial corner in $A$, it suffices to find a $\widetilde P$-polynomial corner in $A'$, where $\widetilde P$ is the \emph{$W$-tricked} auxiliary polynomial
\begin{align}\label{E: tilde P}
    \widetilde P(z):=\frac{P(\beta_1 Wz - \rho)}{\beta_1^2 W}=\beta_d \beta_1^{d-2}W^{d-1}z^d+\beta_{d-1} \beta_1^{d-3} W^{d-2}z^{d-1}+\cdots+z \in \mathbb{Z}[z].
\end{align}
(The linear coefficient $1$ of $\widetilde P$ will be convenient later for technical reasons.)  Indeed, if $$(x',y'),\; (x'+\widetilde P(z'),y'),\;(x',y'+\widetilde P(z'))$$ form a $\widetilde P$-polynomial corner in $A'$, then $A$ contains the $P$-polynomial corner $$(x,y),\; (x+P(z),y),\; (x,y+P(z))$$ with
$$x:=\beta_1^2 Wx'+b_1, \quad y:=\beta_1^2 Wy'+b_2, \quad \text{and} \quad z:=\beta_1 Wz'-\rho.$$
Hence, we replace $\Lambda$ with the \emph{$W$-tricked counting operator}
\begin{align*}
    \Lambda^W(f_0,f_1,f_2):=\sum_{x,y} \E_{z \in [(N/W_d)^{1/d}]} f_0(x,y)f_1(x+\widetilde{P}(z),y)f_2(x,y+\widetilde{P}(z)),
\end{align*}
where $W_d := \beta_d \beta_1^{d-2}W^{d-1}$ is the leading coefficient of $\widetilde P$.  The main work of this paper is showing that, with these modifications made, the counting operators $\Lambda^{\model}$ and $\Lambda^W$ are very close (Proposition \ref{P: Lambda*}).

This comparison theorem boils down to showing that both $\Lambda^{\model}$ and $\Lambda^W$ are ``controlled'' by appropriate $U^1 \times U^1$-box norms of their arguments.  The control for $\Lambda^{\model}$ (Proposition \ref{P: box norm control of model}) is fairly straightforward and follows from an argument similar to the proof of \cite[Lemma 4.1]{PSS23}, so we will focus on the control for $\Lambda^W$.  We wish to show that for all $1$-bounded functions $f_0,f_1,f_2$ supported on $[N]^2$, the largeness of $|\Lambda^W(f_0,f_1,f_2)|$ implies (with quantitative control) the largeness of a box norm of $f_0$ in the directions $\be_1:=(1,0)$ and $\be_2:=(0,1)$.  The control in terms of $f_1$ 
(with the directions $\be_1,\; \be_1-\be_2$) and $f_2$ (with the directions $\be_2,\; \be_2-\be_1$) is completely symmetrical. The precise statement of this result is in Theorem \ref{thm:degree lowering-lambda}, which constitutes a $W$-tricked version of Theorem \ref{thm:non-w-tricked-degree lowering-lambda}.

The starting point of this endeavor is the PET induction scheme and a quantitative concatenation result, as carried out in our companion paper \cite{KKL24a}.  The output of these arguments is the existence of a natural number $t=O_d(1)$ such that the lower bound $|\Lambda^W(f_0,f_1,f_2)| \geq \delta N^2$ implies that
\begin{equation}\label{eq:PET-output-sketch}
\|f_0\|_{(\be_1 W_d\cdot [\pm N/W_d])^t,\; (\be_2 W_d\cdot [\pm N/W_d])^t}^{2^{2t}}\gg_d \delta^{O_d(1)} N^2,
\end{equation}
and likewise for $f_1,f_2$ (cf. Proposition \ref{P: preliminary box control for Lambda^W}).  
For simplicity of exposition, we will temporarily pretend that $W=1$; at the end of this sketch, we will mention how to deal with the additional technical complications arising from the presence of $W$.  For notational simplicity, we will also make the unimportant assumption that $\beta_d=\beta_1=1$ (so that $W_d = 1$ and the $z$-variable ranges over $[N^{1/d}]$ in the definition of $\Lambda^W$).  

In order to replace $t$ with $1$, we will carry out an iterative degree-lowering procedure (see Proposition \ref{prop:degreeloweringdual} below): We will show that if the largeness of $|\Lambda^W(f_0,f_1,f_2)|$ is already known to imply the largeness of $\|f_0\|_{(\be_1\cdot [\pm N])^k,\; (\be_2\cdot [\pm N])^\ell}^{2^{k+\ell}}$ for some $k,\ell$ with $k \geq 2$, then it also implies the largeness of $\|f_0\|_{(\be_1\cdot [\pm N])^{k-1},\; (\be_2\cdot [\pm N])^\ell}^{2^{k+\ell-1}}$ (and likewise with the roles of $k,\ell$ reversed).  We will arrive at our goal after $2t-2$ iterations of this degree-lowering procedure. This is the first time that degree-lowering has been carried out using box norms; all previous degree-lowering arguments in the literature aim at lowering degrees of Gowers norms. Our argument thus testifies to the versatility of the degree-lowering approach pioneered by Peluse \cite{Pel19}.

By the dual--difference interchange (a consequence of the Cauchy--Schwarz inequality; see Lemma \ref{lem:dual-diff}), it suffices to obtain our degree-lowering result when $k=2$ and $\ell=1$.  Suppose that $|\Lambda^W(f_0,f_1,f_2)| \geq \delta N^2$.  We begin by applying the ``stashing''\footnote{The name for this trick is due to Manners \cite{Man21}.} trick: Since $f_0$ is supported on $[N]^2$, using the Cauchy--Schwarz inequality to eliminate $f_0$ from $\Lambda^W$ gives
$$\delta^2 N^4 \leq \abs{\Lambda^W(f_0,f_1,f_2)}^2 \leq N^2 \Lambda^W(\mathcal{D}_0(f_1,f_2),f_1,f_2),$$
where $\mathcal{D}_0(f_1,f_2):\mathbb{Z}^2 \to \mathbb{C}$, the \emph{dual function} of the first argument of $\Lambda^W$, is given by
$$\mathcal{D}_0(f_1,f_2)(x,y):=\E_{z \in [N^{1/d}]} f_1(x+\widetilde P(z),y)f_2(x,y+\widetilde P(z)).$$
The hypothesis of the degree-lowering statement tells us that
$$\norm{\mathcal{D}_0(f_1,f_2)}_{\be_1\cdot [\pm N],\; \be_1\cdot [\pm N],\; \be_2\cdot [\pm N]}^8$$
is large (i.e., at least $c(\delta)N^2$ for some explicit $c(\delta)>0$).  Then the inverse theorem for the $U^2 \times U^1$-norm (see Lemma~\ref{L: U^2 x U^1 inverse} below) implies that $\mathcal{D}_0(f_1,f_2)$ correlates with a ``$U^2 \times U^1$-obstruction'' of the form $g(x)e(a(y)x+b(y))$, where $g: \Z \to \C$ is a $1$-bounded function, $a:\Z \to \R/\Z$ is a phase, and $b:\Z \to \R/\Z$ is a shift.\footnote{To understand where this inverse theorem comes from, notice that differentiating once in the $y$-direction removes the $g(x)$ term and differentiating twice in the $x$-direction removes the $e(a(y)x+b(y))$ term.}  
Unraveling the definition of $\mathcal{D}_0(f_1,f_2)$ and ignoring the terms $g(x),b(y)$ (which are of secondary importance), we find that
$$\Lambda'(a,f_1,f_2):=\sum_{x,y} \E_{z \in [N^{1/d}]}e(a(y)x)f_1(x+\widetilde P(z),y)f_2(x,y+\widetilde P(z))$$
is large in absolute value.

We now aim to show that the phase $a$ is close to some constant $\alpha$ for many values of $y$ (see Proposition \ref{prop:popular-a-value} below); this step ``lowers'' the degree-$2$ obstruction $e(a(y)x)$ to the degree-$1$ obstruction $e(\alpha x)$.  Before sketching the proof of this structural result for $a$, let us see how to use it to complete the degree-lowering procedure.  Substituting $\alpha$ for $a(y)$ and shifting $x \mapsto x-\widetilde P(z)$ in the expression for $\Lambda'(a,f_1,f_2)$, we find that
$$\sum_{x,y} \E_{z \in [N^{1/d}]} e(\alpha (x-\widetilde P(z))f_1(x,y)f_2(x-\widetilde P(z),y+\widetilde P(z))$$
is large.  Grouping the exponential with $f_2$, we recognize this expression as an instance of the counting operator for the $2$-point configuration $(x,y),\; (x-\widetilde P(z),y+\widetilde P(z))$.  Adapting a classical Fourier-analytic argument of S\'ark\"ozy \cite{Sa78a, Sa78b} for $2$-point polynomial configurations, we conclude that $$\norm{f_1}_{(\be_2-\be_1)\cdot[\pm N],\; \be_2\cdot [\pm N]}^4$$
is large.  A further sequence of manipulations in the flavor of stashing let us transfer this ``degree-$2$ control'' from $f_1$ back to $f_0$.  

It remains to analyze the phase function $a$.  It turns out that it suffices to understand the case where $f_2$ is a ``$U^1 \times U^1$-obstruction'' of the form $b_1(y)b_2(x+y)$ for some $1$-bounded functions $b_1,b_2:\Z \to \C$, in which case we are dealing with the expression
\begin{align*}
    \sum_{x,y} \E_{z \in [N^{1/d}]}e(a(y)x)f_1(x+\widetilde P(z),y)b_1(y+\widetilde P(z))b_2(x+y+\widetilde P(z)).
\end{align*}
Shifting $x\mapsto x-\widetilde P(z)$, we reach an expression in which the functions $f_1$ and $b_2$ are independent of $z$. Then, by applying the triangle inequality in $x,y$ and pigeonholing in $x$, we deduce that
\begin{align*}
    \sum_{y \in [N]} \abs{\E_{z \in [N^{1/d}]} e(-a(y)\widetilde P(z)) b_1(y+\widetilde P(z))}
\end{align*}
is large.  A variant of the PET argument from our companion paper \cite{KKL24a} lets us control this expression by a Gowers norm of $b_1$, and the inverse theorem for the Gowers norms produces a nilsequence $\varphi$ of bounded degree, dimension, and complexity such that
\begin{equation}\label{eq:nilsequence-sketch}
\sum_{y \in [N]} \abs{\E_{z \in [N^{1/d}]} e(a(y)\widetilde P(z)) \varphi(y+\widetilde P(z))}
\end{equation}
is large.  To analyze \eqref{eq:nilsequence-sketch}, we require input from the theory of nilsequences.

We use an iterative process, itself a microcosm of degree-lowering, to show that the nilsequence in \eqref{eq:nilsequence-sketch} can be replaced by nilsequences of progressively smaller step (see Proposition \ref{prop:Wtricknilsequenceprop}).  This argument makes use of recent quantitative equidistribution results for polynomial sequences on nilmanifolds due to the third author \cite{Leng23b}.  At the end of this iteration, we have replaced $\varphi(y)$ by a step-$1$ nilsequence, namely, a function of the form $e(p(y))$ for some polynomial $p$.  We then use Weyl's inequality to reduce to the case where $p(y)=\alpha_1 y+\alpha_0$ is linear.  Another application of Weyl's inequality shows that there are many values of $y$ such that $a(y)-\alpha_1$ is major-arc, i.e., close to a rational of small height; pigeonholing in these rational shifts finally gives that $a(y)$ is nearly constant for many $y$'s.

Figure \ref{fig:flowchart} provides a visual overview of the main steps of our argument.

Our degree-lowering arguments must be modified in several ways to account for the parameter $W$.  Recall that the initial norm-control in \eqref{eq:PET-output-sketch} involves differencing with respect to boxes with common difference $W$ (rather than $1$).  In order to manipulate such expressions and use inverse theorems for box norms, we must split all of the parameters in sight (e.g., $x,y,z$) into arithmetic progressions with common difference $W$.  This maneuver, which often requires introducing iterated $W$-tricks, is quite a nuisance but does not change the essential nature of our arguments.

More importantly, the miniature PET induction in the lead-up to \eqref{eq:nilsequence-sketch} yields  Gowers-norm control where the differencing parameters are boxes with side lengths equal to the leading coefficient of $\widetilde P$.  Since this leading coefficient is a constant multiple of a fixed positive integer power of $W$, this step of the argument requires us to replace $W$ with some $W^{O(1)}$, and we must keep track of the accumulation of these exponents in all of the iterations of the degree-lowering.  (We also have to perform $W^{O(1)}$-tricks instead of $W$-tricks.)  

We conclude with a short discussion of the shape of bounds in Theorem \ref{maintheorem}. Of the three logarithms in the bound in Theorem \ref{maintheorem}, two come from the supersaturation version of Shkredov's result \cite{Sh06a, Sh06b}, and the third arises in the transference step of our argument. That we lose only one logarithm in comparing the counts of polynomial and linear corners is reliant on a recent equidistribution result for nilsequences due to the third author \cite{Leng23b} {and the quasipolynomial inverse theorem of Gowers norms \cite{LSS24} (which is itself reliant on \cite{Leng23b})}. If we instead used bounds for equidistribution of nilsequences provided by Tao and Ter\"av\"ainen \cite{TT21} (who quantified earlier equidistribution results of Green and Tao \cite{GT12}), we would end up with a large number of iterated logarithms bounded in terms of the degree of the polynomial $P$.

\subsection{Organization of the paper}

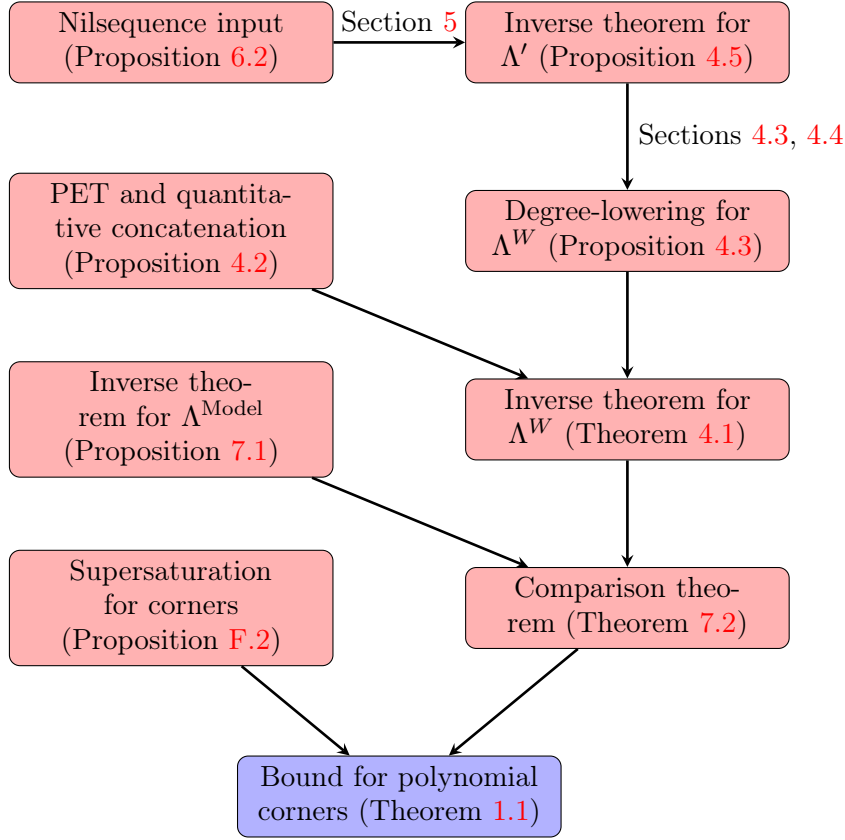
\begin{figure}[h!]
  \centering
\begin{tikzpicture}[mynode/.style={draw, startstop, text width=4cm, align=center}]
    \node[draw, mynode] (A) at (0,0) {Nilsequence input (Proposition \ref{prop:Wtricknilsequenceprop})};
    \node[draw, mynode] (B) at (6,0) {Inverse theorem for $\Lambda'$ (Proposition \ref{prop:iterate-popular-a-value})};
    \node[draw, mynode] (C) at (6,-2.5) {Degree-lowering for $\Lambda^W$ (Proposition \ref{prop:degreeloweringdual})};
    \node[draw, mynode] (D) at (0,-2.5) {PET and quantitative concatenation (Proposition \ref{P: preliminary box control for Lambda^W})};
    \node[draw, mynode] (E) at (6,-5) {Inverse theorem for $\Lambda^W$ (Theorem \ref{thm:degree lowering-lambda})};
    \node[draw, mynode] (F) at (0,-5) {Inverse theorem for $\Lambda^\model$ (Proposition \ref{P: box norm control of model})};
    \node[draw, mynode] (G) at (6,-7.5) {Comparison theorem (Theorem \ref{P: Lambda*})};
    \node[draw, mynode] (H) at (0,-7.5) {Supersaturation for corners (Proposition \ref{prop:supersat})};
    \node[draw, mynode, fill=blue!30] (I) at (3,-10) {Bound for polynomial corners (Theorem \ref{maintheorem})};
    
    \draw[->, >=stealth, line width=1pt] (A) -- node[above] {Section \ref{S: lambda'}} (B);
    \draw[->, >=stealth, line width=1pt] (B) -- node[right] {Sections \ref{S: auiliary}, \ref{S: completing}} (C);
    \draw[->, >=stealth, line width=1pt] (C) -- (E);
    \draw[->, >=stealth, line width=1pt] (D) -- (E);
    \draw[->, >=stealth, line width=1pt] (E) -- (G);
    \draw[->, >=stealth, line width=1pt] (F) -- (G);
    \draw[->, >=stealth, line width=1pt] (G) -- (I);
    \draw[->, >=stealth, line width=1pt] (H) -- (I);
\end{tikzpicture}

    \caption{The main steps and logical implications in the proof of Theorem \ref{maintheorem}.}\label{fig:flowchart}
\end{figure}

After recalling definitions and preliminaries in Section \ref{S: definitions}, we embark on the proof of the main technical result of the paper, Theorem \ref{thm:degree lowering-lambda}, which asserts that the counting operator $\Lambda^W$ is controlled by suitable degree-$2$ box norms. In Section \ref{S: degree lowering}, we show how to use our degree-lowering result, Proposition~\ref{prop:degreeloweringdual}, to deduce Theorem \ref{thm:degree lowering-lambda} from the PET and quantitative concatenation results of the companion paper \cite{KKL24a}. We also reduce Proposition~\ref{prop:degreeloweringdual} to Proposition \ref{prop:popular-a-value}, an inverse theorem for the auxiliary operator $\Lambda'$. The lion's share of the proof of Proposition \ref{prop:popular-a-value} is then carried out in Section \ref{S: lambda'}. The key input in this argument is Proposition \ref{prop:Wtricknilsequenceprop}, a structural result on \eqref{eq:nilsequence-sketch}, which we prove in Section \ref{S: nilsequence} using nilsequence machinery.

The arguments in Sections \ref{S: degree lowering}--\ref{S: nilsequence} make use of several standard results which we collect in the appendices: inverse theorems for various box norms  in Appendix \ref{A: inverse theorems}; a summary of nilsequence theory in Appendix \ref{A: nilsequence theory}; a fraction comparison argument in Appendix \ref{A: fraction comparison}; and a S\'ark\"ozy-style inverse theorem for a $2$-point polynomial progressions (derived with circle-method techniques)  in Appendix \ref{A: Sarkozy}.

We conclude the main body of the paper with Section \ref{S: finishing}, where we finally complete the proof of Theorem \ref{maintheorem}.   Proposition \ref{P: box norm control of model} establishes the desired box-norm control for $\Lambda^\model$, and we compare $\Lambda^W$ and $\Lambda^\model$ in Proposition \ref{P: Lambda*}. The latter proposition requires two more ingredients from the appendices: a Fourier-uniformity result for certain polynomial weights from Appendix \ref{A: Fourier uniformity}; and a supersaturation version of Shkredov's corners theorem from Appendix \ref{A: supersaturation}.

\section{Definitions and preliminaries}\label{S: definitions}

\subsection{Basic notation} 
We write $\N,\Z, \R, \C$ for the sets of positive integers, integers, real numbers, and complex numbers (respectively). For reals $a<b$, we set $[a,b]:= \{\ceil{a}, a+1, \ldots, \floor{b}\}$, and we abbreviate $[0,N-1]$ as $[N]$ and $[-N+1,N-1]$ as $[\pm N]$.  We remark that our somewhat unusual convention for $[N]$ (the norm would be $[N]=[1,N]$) behaves well under passing to arithmetic progressions: For $q|N$, we have the partition $[N] = \bigcup_{j\in[q]}(q\cdot[N/q]+j)$.

We write $\x = (x_1, ..., x_D)$ for an element of $\Z^D$ (for fixed $D\in\N$) and $x$ for an element of $\Z$. When the context is clear, we often write a sum over $\Z^D$ or $\Z$ as $\sum_{\bx}$ or $\sum_x$ instead of $\sum_{\bx \in \Z^D}$ or $\sum_{x \in \Z}$. We also write $\uh=(h_1, \ldots, h_s)$ for a tuple of $s$ integers, and we denote its $\ell^1$-norm by $|\uh|:=|h_1|+\cdots + |h_s|$. We will often combine these two pieces of notation; we use $\uh \bv$ to indicate the tuple $(h_1 \bv, \ldots, h_s\bv)$. 

We denote the indicator function of a set $X$ by $1_X$.  If $X$ is finite and nonempty, we let $\E_{x\in X} := \frac{1}{|X|}\sum_{x\in X}$ be the average over $X$.

We call a function $f:X\to \C$ \textit{$1$-bounded} if $\sup\limits_{x\in X}|f(x)|\leq 1$. 

We use the following standard asymptotic notation. Let $f,g: X\to\C$, with $g$ taking positive real values.  We write $f=O(g)$, $f\ll g$, $g\gg f$, or $g = \Omega(f)$ if there exists $C>0$ such that $|f(x)|\leq C g(x)$ for all $x\in X$, and we write $f\asymp g$ or $f = \Theta(g)$ if $f\ll g$ and $f\gg g$.  If the constant $C$ depends on a parameter, we record this dependence with a subscript, e.g., $f=O_P(g)$, $f \ll_P g$. 

We let $\CC: \C \to \C$ denote the complex conjugation operator $\CC z=\overline{z}$.

\subsection{$W$-trick notation}
Throughout the paper, we let $w=w(N)>10$ be a (fixed) slowly-growing function of $N$; we will later take $w(N)=\exp((\log\log\log\log N)^{C_P})$ for some $C_P>0$.  Let
$$W =W(N):= \prod_{\substack{p < w \\ p \text{ prime}}} p$$
be the product of the primes up to $w$. Standard results on the growth rate of Chebyshev's function give $\log W \asymp w$.

It will often be useful to restrict functions to arithmetic progressions, which we capture in the following definition.
\begin{definition}[$V$-tricked function]\label{d:Vtrickfunction}
For $V \in \mathbb{N}$ and $r \in \mathbb{Z}$, we can parameterize the arithmetic progression $V\mathbb{Z}+r \subseteq \mathbb{Z}$ by $\{Vx+r: x \in \mathbb{Z}\}$.  If $f$ is a function on $\mathbb{Z}$, then its restriction to $V\mathbb{Z}+r$, with this parametrization, is
$$f_{(r,V)}(x):=f(Vx+r).$$
Likewise, for $r_1,r_2 \in \mathbb{Z}$ and $f$ a function on $\mathbb{Z}^2$, the restriction of $f$ to $(V \mathbb{Z} +r_1) \times (V\mathbb{Z}+r_2)$ can be parameterized as
$$f_{(r_1,r_2,V)}(x,y):=f(Vx+r_1,Vy+r_2).$$
We refer to the operation of passing to such a subprogression as a \emph{$V$-trick}, and we say that the resulting functions are \emph{$V$-tricked}.  
\end{definition}
In this paper, we will usually take $V$ to be a small (nonnegative integer) power of $W$ and $r,r_1,r_2$ to be elements of $[V]$. We remark that a composition of two such $V$-tricks is itself a $V$-trick.  To see this symbolically for $f: \mathbb{Z} \to \mathbb{C}$, one can calculate that
\begin{align*}
(f_{(r_1,V_1)})_{(r_2,V_2)}(x) &=f_{(r_1,V_1)}(V_2 x+r_2)\\
 &=f(V_1(V_2 x+r_2)+r_1)\\
 &=f_{(V_1 r_2+r_1,V_1V_2)}(x),
\end{align*}
where $V_1V_2$ is a power of $W$ if $V_1,V_2$ are. The case of functions on $\mathbb{Z}^2$ is analogous.
\begin{definition}[$V$-tricked polynomial]\label{d: Vtrickpolynomial}
For $V$-tricks of polynomials, we will use slightly different notation that will turn out to be helpful for keeping track of divisibility properties of various coefficients.  Let $Q \in \mathbb{Z}[z]$ be a polynomial.  For $V \in \mathbb{N}$ and $r \in\Z$, define the \emph{$V$-tricked} polynomial of $Q$ to be
$$Q_{[r, V]}(z) := \frac{Q(Vz + r) - Q(r)}{V} \in \mathbb{Z}[z].$$
That $Q_{[r,V]}$ has integer coefficients can be easily verified using the Binomial Theorem.    
\end{definition}
As in the case of $V$-tricks for functions, the composition of two $V$-tricks for a polynomial is again a $V$-trick.  In particular, we have
\begin{align*}
(Q_{[r_1, V_1]})_{[r_2, V_2]}(z) &= \frac{Q_{[r_1, V_1]}(V_2z + r_2) - Q_{[r_1, V_1]}(r_2)}{V_2} \\
&= \frac{Q(V_1(V_2z + r_2) + r_1) - Q(r_1) - Q(V_1r_2 + r_1) + Q(r_1)}{V_1V_2} \\
&= Q_{[V_1r_2 + r_1, V_1V_2]}(z).
\end{align*}
Note that iterated $V$-tricks (in both the function setting and the polynomial setting) do not in general commute.

\subsection{Van der Corput Inequalities} 

If $E\subseteq \Z^D$ is finite {and nonempty}, we define the \textit{Fej\'er kernel}
\begin{align*}
    \mu_E(\bh) := \E_{\bh_1, \bh_2\in E}1_{\bh_1 - \bh_2}(\bh). 
\end{align*}
Note that $\sum_\bh \mu_E(\bh) = 1$ and that $0\leq \mu_E(\bh)\leq \frac{1_{E-E}(\bh)}{|E|}$ pointwise. If $E = [\pm H]$, we set
\begin{align*}
    \mu_H(h):=\mu_{[\pm H]}(h) = \E_{h_1, h_2\in[\pm H]} 1_{h_1 - h_2}(h) = \frac{1}{2H-1}\brac{1-\frac{|h|}{2H-1}}_+.
\end{align*}
This definition differs slightly from those in \cite{Pel20, PP19, PP20, Pre17} in that we take the ranges of $h_1, h_2$ to be $[\pm H]$ rather than $[H]$. Making the range of $h_1, h_2$ symmetric around the origin will be convenient for us later.

We will use the following form of the 
standard van der Corput inequality.  For a more thorough discussion, see \cite[Section 3]{KKL24a} (which follows from, e.g., \cite[Lemma 3.2]{Pre17}).

\begin{lemma}[The van der Corput inequality]\label{L: vdC}
    Let $\delta\in(0, 1]$, and let $H, N\in\N$ with $H\leq \frac{\delta^2}{4}N$. If $f:\Z\to\C$ is a $1$-bounded function satisfying 
    \begin{align}\label{E: assumption in vdC}
        \abs{\E_{x\in[N]}f(x)}\geq \delta,
    \end{align}
    then
    \begin{align*}
        \E_{h,h'\in[\pm H]}\E_{x\in[N]}f(x+h)\overline{f(x+h')}\geq \frac{\delta^2}{4}
    \end{align*}
    and
    \begin{align*}
        {\sum_{h}\mu_H(h)\E_{x\in[N]}f(x)\overline{f(x+h)}}\geq \frac{\delta^2}{4}.
    \end{align*}
\end{lemma}

\subsection{Box norms}
We now briefly review the theory of box norms on $\Z^D$; we will later work with the groups $\Z$ and $\Z^2$.
Given $\bh, \bh'\in \Z^D$ and a finitely supported function $f: \Z^D\to\C$, we define
the (multiplicative) discrete derivatives
$\Delta_\bh f(\bx) := f(\bx)\overline{f(\bx+\bh)}$ and $\Delta'_{(\bh,\bh')} f(\bx) := f(\bx+\bh)\overline{f(\bx+\bh')}$.  Similarly, compositions of such discrete derivatives are given by
\begin{align*}
    \Delta_{\bh_1, \ldots, \bh_s}f(\bx) &= \Delta_{\bh_1}\cdots \Delta_{\bh_s}f(\bx) = \prod_{\ueps\in\{0,1\}^s}\CC^{|\ueps|}f(\bx+\ueps\cdot \ubh),\\
    \Delta'_{(\bh_1, \bh_1'), \ldots,\; (\bh_s,\bh_s')}f(\bx) &= \Delta'_{(\bh_1, \bh_1')}\cdots\Delta'_{(\bh_s, \bh_s')}f(\bx) = \prod_{\ueps\in\{0,1\}^s}\CC^{|\ueps|}f(\bx+(\underline{1}-\ueps)\cdot\ubh + \ueps\cdot\ubh'),
\end{align*}
where $\ubh = (\bh_1, \ldots, \bh_s)$ and again $\CC z = \overline{z}$ for $z\in\C$.

For finite sets  $E_1, \ldots, E_s\subseteq \Z^D$, we define the \textit{box norm} of $f$ {along}  $E_1, \ldots, E_s$ to be
\begin{align*}
    \norm{f}_{E_1, \ldots, E_s}:=\brac{\E_{\bh_1, \bh_1'\in E_1}\cdots\E_{\bh_s, \bh_s'\in E_s}\sum_\bx \Delta'_{(\bh_1, \bh_1'), \ldots,\; (\bh_s,\bh_s')}f(\bx)}^{1/2^s}.
\end{align*}
We can equivalently express this as
\begin{align}\label{E: box norms with Fejer kernels}
    \norm{f}_{E_1, \ldots, E_s} = \brac{\sum_{\bx, \bh_1, \ldots, \bh_s} \mu_{E_1}(\bh_1)\cdots \mu_{E_s}(\bh_s) \Delta_{\bh_1, \ldots, \bh_s}f(\bx)}^{1/2^s}
\end{align}
after shifting $\bx\mapsto \bx-\bh'_1-\cdots - \bh_s'$ and removing the extraneous averages over $\bh'_1, \ldots, \bh'_s$.
For the sake of completeness, we also define $\norm{f}_\emptyset:= \sum_\bx f(\bx)$ (but note that this is not a seminorm).  If a set $E$ repeats $t$ times, we often write $E^t$ instead of $E, \ldots, E$.  We also write $\norm{f}_{U^s(E)} := \norm{f}_{E^s}$ for the degree-$s$ \textit{Gowers norm} of $f$ along $E$.
Box norms satisfy a number of well known properties, listed in the following lemma.
\begin{lemma}[{Properties of box norms, \cite[Lemma 3.5]{KKL24a}}]\label{L: properties of box norms}
Let $\delta\in(0, 1)$, let $d,D,N, s\in\N$,
and let $E_1, \ldots, E_s\subseteq \Z^D$ be finite {and nonempty} sets. Let $f:\Z^D\to\C$ be a $1$-bounded function supported on $[N]^D$. Then the following properties hold:
\begin{enumerate}
    \item\label{i: inductive formula} (Inductive formula) For any $i\in\{1, \ldots, s\}$, we have
    \begin{align*}
        \norm{f}_{E_1, \ldots, E_s}^{2^s} = \E_{\bh_1, \bh_1'\in E_1}\cdots \E_{\bh_i, \bh_i'\in E_i}\norm{\Delta'_{(\bh_1, \bh_1'), \ldots,\; (\bh_i, \bh_i')}f}_{E_{i+1}, \ldots, E_s}^{2^{s-i}}.
    \end{align*}
    \item\label{i: permutation invariance} (Permutation invariance) For any permutation $\sigma$ on $\{1, \ldots, s\}$, we have 
    \begin{align*}
        \norm{f}_{E_1, \ldots, E_s} = \norm{f}_{E_{\sigma(1)}, \ldots, E_{\sigma(s)}}.
    \end{align*}
    \item\label{i: monotonicity} (Monotonicity) If $f$ is $1$-bounded and $|[N]^D-E_s|\leq C N^D$ for some $C>0$, then 
    \begin{align}\label{E: monotonicity}
        \norm{f}_{E_1, \ldots, E_{s-1}}^{2^{s-1}}\geq \delta N^D \quad \Longrightarrow \quad \norm{f}_{E_1, \ldots, E_{s}}^{2^{s}}\geq \delta^2 N^D/C.
    \end{align}
    \item\label{i: enlarging} (Enlarging the sets) If $s\geq 2$ and $E'_i\subseteq \Z^D$ is a finite set satisfying $E'_i \supseteq E_i$ for each $i\in\{1,\ldots, s\}$, then 
    \begin{align*}
        \norm{f}_{E_1, \ldots, E_s}\leq \brac{\frac{|E_1'|\cdots|E_s'|}{|E_1|\cdots|E_s|}}^{1/2^{s-1}}\norm{f}_{E_1', \ldots, E_s'}.
    \end{align*}
    \item\label{i: trimming 2} (Trimming the lengths of boxes)  
    Let
    \begin{align*}
        E_i = \bbeta_{i}\cdot[\pm H_{i}]\quad \textrm{and}\quad E_i' = \bbeta_{i}\cdot[\pm \kappa_{i} H_{i}]
    \end{align*}
    for some $\bbeta_{i}\in \Z^D$, $\kappa_{i}>0$, and $H_{i}\in \N$. If $f$ is $1$-bounded, then
    \begin{align*}
        \norm{f}_{E_1, \ldots, E_s}^{2^s}\geq \delta N^D\quad \Longrightarrow\quad        \norm{f}_{E_1', \ldots, E_s'}^{2^s}\gg_{s}\delta^{2^s}\min(1, \delta^{2^s}/\kappa)^{2s} N^D
    \end{align*}
    for $\kappa = \max_{i}\kappa_{i}.$
    \item\label{i: passing to APs} (Passing to sub-arithmetic progressions) Suppose that  $f$ is $1$-bounded and
    \begin{align*}
        E_i = \bbeta_{i}\cdot[\pm H_{i}]\quad \textrm{and}\quad E_i' = q_{i}\bbeta_{i}\cdot[\pm H_{i}/q_{i}]
    \end{align*}
    for some $\bbeta_{i}\in \Z^D$ and $q_{i}, H_{i}\in\N$ satisfying $q_{i}\leq C \delta^{2^s}H_{i}$ for all $i$ and some constant $C>0$.
    Then
    \begin{align*}
        \norm{f}_{E_1, \ldots, E_s}^{2^s}\geq \delta N^D \quad \Longrightarrow\quad        \norm{f}_{E_1', \ldots, E_s'}^{2^s}\gg_{s,C} \delta N^D.
    \end{align*}
\end{enumerate}
\end{lemma}

\subsection{Dual-difference interchange}

On several occasions, we apply a standard dual-difference interchange trick which has been a linchpin in all previous degree-lowering arguments. Its proof rests on the following observation.

\begin{lemma}\label{lem:dual-diff preliminary}
Let $D, t\in\N$, let $C>0$, 
let $E\subseteq\Z^D$ be finite and nonempty,
and for each $i\in\{1, \ldots, t\}$, let $F_i(\bx) = \E_{z_i\in Z_i} f_i(\bx, z_i)$ be a function from $\Z^D$ to $\C$ supported on a set of size at most $CN^D$, where $Z_i$ is a nonempty finite set and $f_i: \Z^D \times Z_i \to \mathbb{C}$ is $1$-bounded. 
If
$$\abs{\sum_{\bx, \bh} \mu_E(\bh) \Delta_{(\bh,0)}\prod_{i=1}^t\E_{z_i \in Z_i}f_i(\bx,z_i)} \geq \delta N^D,$$
then
$${\sum_{\bx,\bh} \mu_E(\bh)\prod_{i=1}^t\E_{z_i \in Z_i}\Delta_{(\bh,0)}f_i(\bx,z_i)} \geq (\delta^2/C) N^D.$$
\end{lemma}
\begin{proof}
Expanding and applying the Cauchy--Schwarz inequality gives
\begin{align*}
\delta^2 N^{2D} &\leq 
 \abs{\sum_\bx \E_{\bh\in E} \prod_{i=1}^t\E_{z_i \in Z_i}f_i(\bx+\bh,z_i)\cdot  \E_{\bh'\in E} \prod_{i=1}^t\E_{z'_i \in Z_i}f_i(\bx+\bh',z'_i)}^2\\
 &\leq C N^D \cdot \sum_\bx \E_{z_1 \in Z_1, \ldots, z_t\in Z_t} \Bigabs{\E_{\bh \in E} \prod_{i=1}^t f_i(\bx+\bh,z_i)}^2\\
 &=CN^D \cdot \sum_{\bx, \bh} \mu_E(\bh) \prod_{i=1}^t \E_{z_i \in Z_i} \Delta_{(\bh,0)}f_i(\bx,z_i),
\end{align*}
as desired.
\end{proof}

Iterating Lemma \ref{lem:dual-diff preliminary}, we obtain the following.
\begin{lemma}[Dual--difference interchange]\label{lem:dual-diff}
    Let $C>0$, let $D, r,s\in\N$ with $1 \leq r \leq s$, 
    let $E_1, \ldots, E_s\subseteq [\pm CN^D]$ be nonempty sets, 
and let $F(\bx) = \E_{z\in Z} f(\bx, z)$ be a function from $\Z^D$ to $\C$,
where $Z$ is a nonempty finite set and $f: \Z^D \times Z \to \mathbb{C}$ is a $1$-bounded function supported on $[\pm CN^D]\times Z$. If
\begin{align*}
    \norm{F}_{E_1, \ldots, E_s}^{2^s}\geq \delta N^D,
\end{align*}
then 
\begin{align*}
    \sum_{\bh_1, \ldots, \bh_r}
    \mu_{E_1}(\bh_1)\cdots \mu_{E_r}(\bh_r)\norm{F^{\bh_1, \ldots \bh_r}}^{2^{s-r}}_{E_{r+1}, \ldots, E_s}\gg_{C, r} \delta^{2^r} N^D,
\end{align*}
where 
\begin{align*}
    F^{\bh_1, \ldots, \bh_r}(\bx) := \E_{z\in Z}\Delta_{(\bh_1,0), \ldots,\; (\bh_r,0)}f(\bx, z).
\end{align*}
\end{lemma}
\begin{proof}
    Expanding the definition of the box norm, we get that
    \begin{align*}
        \sum_{\bh_2, \ldots, \bh_s}
        \mu_{E_2}(\bh_2)\cdots \mu_{E_s}(\bh_s)\cdot \sum_{\bx, \bh_1}\mu_{E_1}(\bh_1) \Delta_{(\bh_1,0), \ldots,\; (\bh_s, 0)}\E_{z\in Z}f(\bx, z)\geq \delta N^D.
    \end{align*}
    We apply Lemma \ref{lem:dual-diff preliminary} for a fixed tuple $(\bh_2, \ldots, \bh_s)\in E_2 \times E_s$ with the role of the product over $i\in \{1, \ldots, t\}$ played by the iterated multiplicative derivative $\Delta_{(\bh_2,0), \ldots,\; (\bh_s, 0)}$.  This lets us move the differencing over $\bh_1$ inside the expectation, which gives 
    \begin{align*}
        \sum_{\bh_1, \ldots, \bh_s}
        \mu_{E_1}(\bh_1)\cdots \mu_{E_s}(\bh_s)\cdot \sum_{\bx}\Delta_{(\bh_2,0), \ldots,\; (\bh_s, 0)}\E_{z\in Z}\Delta_{(\bh_1, 0)}f(\bx, z)\gg_C \delta^2 N^D.
    \end{align*}
    We continue in the same fashion to move the differencings over $\bh_2, \ldots, \bh_r$ inside the expectation; here we are implicitly using the fact that the functions $\Delta_{(\bh_1, 0), \ldots,\; (\bh_j, 0)}f$ for $\bh_i\in E_i$, $1\leq i\leq j\leq r$ are supported on $[\pm C N^D]\times Z$.
\end{proof}

\subsection{Fourier analysis}
We conclude the preliminaries section with a brief summary of Fourier analysis on $\Z$, which we shall use on occasion. For a finitely supported function $f:\Z\to\C$ and $\xi\in\R/\Z$, we define the \emph{Fourier transform} $\hat f: \R/\Z \to \C$ via
\begin{align*}
    \hat{f}(\xi):= \sum_x f(x)e(\xi x).
\end{align*}
The Fourier inversion formula gives $f(\xi) = \int_{\R/\Z} \hat{f}(\xi)e(-\xi x)\, d\xi$, and the Parseval Identity gives $\norm{f}_2 = \norm{\hat{f}}_2$.  Here the $L^p$-norm of a function $f:\Z\to\C$ is defined by $\norm{f}_p:= \Bigbrac{\sum_x |f(x)|^p}^{1/p}$ for $1\leq p<\infty$ and $\norm{f}_\infty :=\sup_{x\in\Z}f(x)$, the $L^p$-norm of a (Lebesgue-measurable) function $F:\R/\Z\to\C$ is defined by $\norm{F}_p:= \Bigbrac{\int_{\R/\Z}|F(\xi)|^p\, d\xi}^{1/p}$ for $1\leq p<\infty$ and $\norm{F}_\infty :=\sup\{r\geq 0:\; \mathrm{Leb}(|F|^{-1}(r,\infty))>0\}$.

\section{Degree-lowering}\label{S: degree lowering}
For the remainder of the paper, we fix the integer polynomial $P$ together with its degree $d$ and the auxiliary $W$-tricked polynomial $\widetilde{P}$ defined in \eqref{E: tilde P}.  Recall that for $1$-bounded functions $f_0,f_1,f_2: \Z^2 \to \mathbb{C}$ supported on $[N]^2$, we have the counting operator
$$\Lambda^W(f_0,f_1,f_2)=\sum_{x,y \in \mathbb{Z}} \E_{z \in [(N/W_d)^{1/d}]}f_0(x,y)f_1(x+\widetilde{P}(z),y)f_2(x,y+\widetilde{P}(z)),$$
where $W_d$ is the leading coefficient of $\widetilde P$.  Recall also that $\be_1:=(1,0)$ and $\be_2:=(0,1)$ denote the usual basis vectors in $\mathbb{Z}^2$. As described in the proof outline above, the main work of this paper is showing that $\Lambda^W$ is ``controlled'' by appropriate degree-$2$ box norms of its arguments.  
\begin{theorem}[{Inverse theorem for $\Lambda^W$}]\label{thm:degree lowering-lambda}
There is a constant $C=C_P>0$ such that the following holds.  Let $\delta\in(0, 1/10)$, and let $f_0, f_1, f_2: \Z^2 \to \mathbb{C}$ be $1$-bounded functions supported on $[N]^2$ for some integer $N  \geq C W^{C}\exp(\log(1/\delta)^{C})$.  If
$$|\Lambda^W(f_0,f_1,f_2)| \geq \delta N^2,$$
then 
there exists some positive integer power $V = W^{O_P(1)}$ such that
\begin{align*}
    &\|f_0\|_{\be_1V \cdot [N/V],\; \be_2V \cdot [N/V]}^4 \geq \exp(-\log(1/\delta)^{O_P(1)}) N^2,\\ 
    &\|f_1\|_{\be_1V \cdot [N/V],\; (\be_2-\be_1)V \cdot [N/V]}^4 \geq \exp(-\log(1/\delta)^{O_P(1)}) N^2,\\
    \quad \text{and} \quad &\|f_2\|_{\be_2V \cdot [N/V],\; (\be_2-\be_1)V \cdot [N/V]}^4 \geq \exp(-\log(1/\delta)^{O_P(1)}) N^2.
\end{align*}
\end{theorem}
Theorem \ref{thm:degree lowering-lambda} is a $W$-tricked version of Theorem \ref{thm:non-w-tricked-degree lowering-lambda} from the introduction. While we prove only the former, the latter follows from a simplified version of the proof of Theorem \ref{thm:degree lowering-lambda} in which we take $W=V=1$ in all places and replace $\widetilde{P}$ with $P$. We remark that while Theorem \ref{thm:degree lowering-lambda} has the same restriction on $P$ as Theorem \ref{maintheorem}, no such restriction is present in the statement of Theorem \ref{thm:non-w-tricked-degree lowering-lambda}. This is because the fraction-comparison argument Lemma \ref{lem:composing-polynomials-2}, whose proof in the $W$-tricked setting uses the special structure of ${P}$, holds trivially for all non-$W$-tricked polynomials.

 To prove Theorem \ref{thm:degree lowering-lambda}, we will use a degree-lowering argument.  The starting point for the argument is the following special case of \cite[Theorem 1.2]{KKL24a}. 
\begin{proposition}\label{P: preliminary box control for Lambda^W}
There exist a constant $C=C_d>0$ and a positive integer $t = O_d(1)$ such that for all $\delta\in(0, 1/10)$ and $1$-bounded functions $f_0,f_1,f_2: \Z^2 \to \mathbb{C}$ supported on $[N]^2$ {for some integer $N\geq C W_d\delta^{-C}$,} the bound
    $$|\Lambda^W(f_0,f_1,f_2)| \geq \delta N^{2}$$
implies that
\begin{align*}
&\|f_0\|_{(\be_1 W_d\cdot [\pm N/W_d])^t,\; (\be_2 W_d\cdot [\pm N/W_d])^t}^{2^{2t}} \gg_d \delta^{O_d(1)} N^2\\
&\|f_1\|_{(\be_1 W_d\cdot [\pm N/W_d])^t,\; ((\be_2-\be_1) W_d\cdot [\pm N/W_d])^t}^{2^{2t}} \gg_d \delta^{O_d(1)} N^2\\
\quad \text{and} \quad &\|f_2\|_{(\be_2 W_d\cdot [\pm N/W_d])^t,\; ((\be_2-\be_1) W_d\cdot [\pm N/W_d])^t}^{2^{2t}} \gg_d \delta^{O_d(1)} N^2.
\end{align*}
\end{proposition}

Our goal in the next few sections is to reduce $t$ to $2$ in the box norms appearing in Proposition \ref{P: preliminary box control for Lambda^W}. The following proposition captures the main iterative step of this process.
\begin{proposition}[Degree-lowering]\label{prop:degreeloweringdual}
Let $k,\ell \in \mathbb{N}$ with $\max(k,\ell) \geq 2$. For every $C_1>0$ there exists $C_2 = C_2(C_1, P, k, \ell)>0$ such that the following holds. Let $V \leq W^{C_1}$ be a nonnegative integer power of $W$, and suppose that for all $\delta\in(0, 1/10)$, all positive integers $N \geq C_2 W^{C_2}\exp(\log(1/\delta)^{C_2})$, and all $1$-bounded functions $f_0, f_1, f_2: \Z^2 \to \mathbb{C}$ supported on $[N]^2$, the lower bound
$$|\Lambda^W(f_0, f_1, f_2)| \ge \delta N^2$$
implies that
$$\|f_0\|_{(\be_1V \cdot [\pm N/V])^k,\; (\be_2 V \cdot [\pm N/V])^\ell}^{2^{k+\ell}} \geq \exp(-\log(1/\delta)^{C_1}) N^2.$$
Then 
there exists some positive integer power $\tilde{V} = V^{O_{C_1,P,k,\ell}(1)}$ such that
$$\|f_0\|_{(\be_1\tilde{V} \cdot [N/\tilde{V}])^{k - 1},\; (\be_2 \tilde{V} \cdot [\pm N/\tilde{V}])^\ell}^{2^{k+\ell-1}} \ge \exp(-\log(1/\delta)^{O_{C_1,P,k,\ell}(1)})N^2$$
if $k \ge 2$ and
$$\|f_0\|_{(\be_1\tilde{V} \cdot [N/\tilde{V}])^{k},\; (\be_2 \tilde{V} \cdot [\pm N/\tilde{V}])^{\ell - 1}}^{2^{k+\ell-1}} \ge \exp(-\log(1/\delta)^{O_{C_1,P,k,\ell}(1)})N^2$$
if $\ell \ge 2$.
\end{proposition}

Before proving Proposition \ref{prop:degreeloweringdual}, we show how it implies Theorem \ref{thm:degree lowering-lambda}.

\begin{proof}[Proof of Theorem~\ref{thm:degree lowering-lambda} assuming Propositions \ref{P: preliminary box control for Lambda^W} and \ref{prop:degreeloweringdual}.]
By Proposition \ref{P: preliminary box control for Lambda^W}, there exists a positive integer $t=O_d(1)$ such that 
$$\|f_0\|_{(\be_1W_d \cdot [\pm N/W_d])^t,\; (\be_2 W_d \cdot [\pm N/W_d])^t}^{2^{2t}} \gg_{d}\delta^{O_d(1)} N^2.$$
Using Lemma \ref{L: properties of box norms}\eqref{i: enlarging}, we can replace $W_d$ by $W^{d-1}$ at the cost of a factor of $O_P(1)$ in the lower bound.
We now apply Proposition~\ref{prop:degreeloweringdual} $2t-2$ times.  Since $N$ is chosen sufficiently large, we end up with a positive integer power $V = W^{O_d(1)}$ such that
$$\|f_0\|_{\be_1V \cdot [\pm N/V],\; \be_2 V \cdot [\pm N/V]}^{4} \geq \exp(-\log(1/\delta)^{O_P(1)}) N^2.$$
This gives the claimed control in terms of $f_0$.

The control in terms of $f_1,f_2$ follows by symmetry. To see this, notice that the affine linear map $(x, y) \mapsto (2N-1-(x + y), y)$ sends $[N]^2$ into $[2N]^2$ and preserves the set of corners with each given side length, with the caveat that the two points on the horizontal side of each corner have been ``exchanged''.  More precisely, for each $i$, define the function $\widetilde f_i: [2N]^2 \to \C$ by $\widetilde f_i(x,y):=f_i(2N-1-(x + y), y)$.  Then the change of variables $x \mapsto 2N-1-x-y-\widetilde P(z)$ gives the identity
$$\Lambda^W(f_0, f_1, f_2) = \Lambda^W(\widetilde{f_1}, \widetilde{f_0}, \widetilde{f_2}).$$
The previous paragraph (with $N,\delta$ replaced by $2N, \delta/4$) tells us that
$$\|\widetilde f_1\|_{\be_1V \cdot [\pm 2N/V],\; \be_2 V \cdot [\pm2 N/V]}^{4} \geq \exp(-\log(1/\delta)^{O_P(1)}) N^2.$$
Using Lemma \ref{L: properties of box norms}\eqref{i: trimming 2} to replace $[\pm2 N/V]$ by $[\pm N/V]$ and unwinding the definition of $\widetilde f_1$, we conclude that $$\|f_1\|_{\be_1V \cdot [N/V],\; (\be_2-\be_1)V \cdot [N/V]}^4 \geq \exp(-\log(1/\delta)^{O_P(1)}) N^2.\\$$
The control for $f_2$ follows from an analogous change of variables.
\end{proof}

\subsection{Proof of Proposition \ref{prop:degreeloweringdual}: preliminary maneuvers}

We now begin the proof of Proposition \ref{prop:degreeloweringdual}. All implicit constants are allowed to depend on $P,k,l$. 
The initial assumption
\begin{align*}
    |\Lambda^W(f_0, f_1, f_2)| \ge \delta N^2
\end{align*}
implies that
\begin{align}\label{E: initial f_0 control}
    \|f_0\|_{(\be_1V \cdot [\pm N/V])^k,\; (\be_2 V \cdot [\pm N/V])^\ell}^{2^{k+\ell}} \geq \exp(-\log(1/\delta)^{O(1)}) N^2
\end{align}
for all $1$-bounded functions $f_0, f_1, f_2:\Z^2\to\C$ supported on $[N]^2$. Our proof strategy is as follows. Applying several standard maneuvers such as stashing, the dual-difference interchange and an inverse theorem for the $U^2\times U^1$ box norm, we reduce the problem to the study of the auxiliary operator
$$\Lambda'(a,\widetilde{f_1},\widetilde{f_2}):=\sum_{x,y} \E_{z \in [(N/W_d)^{1/d}V^{-1}]} e(a(y)x) \widetilde{f_1}(x+\widetilde P_{[r, V]}(z),y) \widetilde{f_2}(x,y+\widetilde P_{[r, V]}(z)).$$
The phase $e(a(y)x)$ is the ``main term'' that arises from the $U^2\times U^1$ inverse theorem.
Our main technical advancement is a structural result on $\Lambda'$: We show that if $\Lambda'(a,\widetilde{f_1},\widetilde{f_2})$ is large, then $a$ is ``approximately constant'' in the sense that there are many arithmetic progressions of difference $V'=V^{O(1)}$ on each of which $a$ is constant a positive proportion of the time. Morally speaking, this means that $\Lambda'(a,\widetilde{f_1},\widetilde{f_2})$ can be large only when $e(a(y)x)$ is actually a $U^1\times U^1$ obstruction. We use this information to obtain the auxiliary norm control
\begin{align*}
    \|f_1\|^4_{(\be_1 VV' \cdot [\pm N/\Tilde{V}])^{k - 1},\; (\be_2 - \be_1)\Tilde{V} \cdot [\pm N/\Tilde{V}],\; (\be_2 \Tilde{V} \cdot [\pm N/\Tilde{V}])^{\ell-1}} \geq \exp(-\log(1/\delta)^{O(1)}) N^2,
\end{align*}
where $\Tilde{V}=VV'$.  We then transfer this norm-control from $f_1$ to $f_0$.

Let us fill in some more details of this strategy. In doing so, we will also highlight several techniques used in various guises throughout the paper.

 We define the \emph{dual function}
\begin{align*}
    \CD_0(f_1,f_2)(x,y) := \E_{z \in [(N/W_d)^{1/d}]} f_1(x+\widetilde{P}(z),y)f_2(x,y+\widetilde{P}(z))
\end{align*}
with respect to the first argument of $\Lambda^W$.  The Cauchy--Schwarz inequality gives 
\begin{align*}
    |\Lambda^W(f_0,f_1,f_2)|^2 &= \abs{\sum_{x,y}f_0(x,y)\CD_0(f_1,f_2)(x,y)}^2\\
    &\leq N^2\cdot\sum_{x,y}|\CD_0(f_1,f_2)(x,y)|^2\\
    &= N^2\cdot \Lambda^W(\CD_0(f_1, f_2),\overline{f_1},\overline{f_2}),
\end{align*}
for all $1$-bounded functions $f_0$ supported on $[N]^2$.  It follows that if $\Lambda^W(f_0,f_1,f_2) \geq \delta N^2$, then we also have
\begin{align*}
    \Lambda^W(\CD_0(f_1, f_2),\overline{f_1},\overline{f_2}) \geq \delta^2 N^2.
\end{align*}
This trick whereby we replace an arbitrary function $f_0$ by the structured function $\CD_0(f_1, f_2)$ is often called \textit{stashing}, and it will be the bread and butter of our arguments.
In particular, once we know that $\Lambda^W$ is controlled by some norm $\norm{\cdot}$ of its first argument, we see that the largeness of $\Lambda^W(f_0,f_1,f_2)$ implies not only the largeness of $\norm{f_0}$ but also the largeness of $\norm{\CD_0(f_1,f_2)}$.
In this particular case, stashing and \eqref{E: initial f_0 control} together give
\begin{align}\label{E: D_0 control}
        \|\CD_0(f_1, f_2)\|_{(\be_1V \cdot [\pm N/V])^k,\; (\be_2 V \cdot [\pm N/V])^\ell}^{2^{k+\ell}} \geq \exp(-\log(1/\delta)^{O(1)}) N^2.
\end{align}

The next step is to apply the dual--difference interchange to \eqref{E: D_0 control} in order to arrive at an average of $U^2\times U^1$-box norms to which we can apply an inverse theorem. To begin, applying Lemma \ref{lem:dual-diff} to \eqref{E: D_0 control} gives
$$\sum_{\substack{\uh_1 \in \mathbb{Z}^{k-2},\\ \uh_2 \in \mathbb{Z}^{\ell-1}}} \mu_{N/V}(\uh_1,\uh_2)\|\mathcal{D}_0^{\uh_1,\uh_2}(f_1,f_2)\|_{(\be_1V\cdot[\pm N/V])^2,\; \be_2V\cdot[\pm N/V]}^8 \geq \exp(-\log(1/\delta)^{O(1)})N^2,$$
where
$$\mathcal{D}_0^{\uh_1,\uh_2}(f_1,f_2)(x,y):=\E_{z \in [(N/W_d)^{1/d}]} \Delta_{\uh_1 V\be_1,\uh_2 V\be_2}(f_1(x+\widetilde P(z),y)f_2(x,y+\widetilde P(z)))$$
for $\uh_1 \in \mathbb{Z}^{k-2}$ and $\uh_2 \in \mathbb{Z}^{\ell-1}$. (We recall here that $\uh_1 V\be_1 = (h_{11}V\be_1, \ldots, h_{1(k-2)}V\be_1)$, and similarly for $\uh_2 V\be_2$.)
Since $\mu_{N/V}(\uh_1,\uh_2) \ll (N/V)^{-k-\ell+3}$ pointwise, the popularity principle provides at least $\exp(-\log(1/\delta)^{O(1)}) (N/V)^{k+\ell-3}$ choices of $(\uh_1,\uh_2)$ such that
\begin{equation} \label{eq:dual-before-change-of-var}
\|\mathcal{D}_0^{\uh_1,\uh_2}(f_1,f_2)\|_{(\be_1V\cdot[\pm N/V])^2,\; \be_2V\cdot[\pm N/V]}^8 \geq \exp(-\log(1/\delta)^{O(1)})N^2.
\end{equation}
For the next several steps, fix such a choice of $\uh_1,\uh_2$.  (The reader can now see why all of the main ideas appear in the case $k=2$, $\ell=1$ described in the proof sketch.)

The box norm in \eqref{eq:dual-before-change-of-var} is not yet ripe for an application of an inverse theorem for the $U^2\times U^1$ box norms because of the multiplicative factor $V$ in the boxes $\be_1V\cdot[\pm N/V]$ and $\be_2V\cdot[\pm N/V]$. To bring the box norm into a form amenable to an application of the inverse theorem, we split the ranges of $x,y$ into arithmetic progressions of common difference $V$ and observe the identity 
\begin{align}\label{E: identity on dilated box norms}
    \norm{g}_{(\be_1V\cdot[\pm N/V])^2,\; \be_2V\cdot[\pm N/V]}^8 = \sum_{j_1, j_2\in[V]} \norm{g_{(j_1, j_2, V)}}_{(\be_1\cdot[\pm N/V])^2,\; \be_2\cdot[\pm N/V]}^8.
\end{align}
for any compactly supported $g:\Z\to\C$; we recall the $V$-tricked function $g_{(j_1, j_2, V)}(x,y)= g(Vx + j_1, Vy + j_2)$ from Definition~\ref{d:Vtrickfunction}. Notice that after this change of variables, the boxes in the box norm have side length $[N/V]$ instead of $V \cdot [N/V]$.

We will of course apply the identity \eqref{E: identity on dilated box norms} to the functions $\mathcal{D}^{\uh_1,\uh_2}_0(f_1,f_2)$.
This maneuver requires us to adapt the range of $z$ accordingly: Recalling the definition of the $V$-trick for polynomials, we replace $\widetilde{P}(z)$ by $\widetilde{P}_{[r, V]}(z) = \frac{\widetilde P(Vz + r)-\widetilde P(r)}{V}$. 
On changing variables 
$z \mapsto Vz+r$, we can express the dual function as 
\begin{align*}
        &\mathcal{D}_0^{\uh_1,\uh_2}(f_1,f_2)(Vx+j_1,Vy+j_2)\\
    &=\E_{z \in [(N/W_d)^{1/d}]} \Delta_{V\uh_1 \be_1,V\uh_2 \be_2}(f_1(Vx+j_1 + \widetilde P(z),Vy+j_2)f_2(Vx+j_1,Vy+j_2+\widetilde P(z)))\\
 &=\E_{r \in [V]} \E_{z \in [(N/W_d)^{1/d}V^{-1}]} \Delta_{V\uh_1 \be_1,V\uh_2 \be_2} (f_1(V(x+\widetilde P_{[r, V]}(z))+j_1+\widetilde P(r),Vy+j_2) \\
  &\qquad\qquad\qquad\qquad\qquad\qquad\qquad\qquad\times f_2(Vx+j_1,V(y+\widetilde P_{[r, V]}(z))+j_2+\widetilde P(r)))\\
  &=\E_{r \in [V]} \E_{z \in [(N/W_d)^{1/d}V^{-1}]} \Delta_{\uh_1 \be_1,\uh_2 \be_2} (f_{1,j_1+\widetilde P(r),j_2}(x+\widetilde P_{[r, V]}(z),y) f_{2,j_1, j_2+\widetilde P(r)}(x,y+\widetilde{P}_{[r, V]}(z))),
\end{align*}
where we also set $$f_{l, j_1, j_2}(x,y) := (f_l)_{(j_1, j_2, V)}(x,y) = f_l(Vx + j_1, Vy + j_2)$$ for brevity.  We observe that these $V$-tricked functions are supported on an interval of length $N/V$. {In the formula above (and in similar expressions in the future), we assume without loss of generality that $(N/W_d)^{1/d}V^{-1}\in\N$ (as we may since $N$ is large compared to $V$), so that splitting the range of $z$ into arithmetic progressions incurs no error.} 

Invoking \eqref{E: identity on dilated box norms} with $g = \mathcal{D}_0^{\uh_1,\uh_2}(f_1,f_2)$ and applying the triangle inequality to \eqref{eq:dual-before-change-of-var} gives
\begin{multline*}
\E_{j_1,j_2,r \in [V]}\left|\!\left|\E_{z \in [(N/W_d)^{1/d}V^{-1}]} \Delta_{\uh_1 \be_1,\uh_2 \be_2} f_{1,j_1+\widetilde P(r),j_2}(\cdot+\widetilde P_{[r, V]}(z),\cdot)\right.\right.\\
\left.\left. \Delta_{\uh_1 \be_1,\uh_2 \be_2} f_{2,j_1, j_2+\widetilde P(r)}(\cdot, \cdot+\widetilde P_{[r, V]}(z))\vphantom{\E_{z \in [(N/W_d)^{1/d}V^{-1}]}}\right|\!\right|_{(\be_1\cdot[\pm N/V])^2,\; \be_2\cdot[\pm N/V]}^8\\
\geq \exp(-\log(1/\delta)^{O(1)}) (N/V)^2.
\end{multline*}
  Since each summand contributes at most $(N/V)^2$, the popularity principle provides at least $\exp(-\log(1/\delta)^{O(1)}) V^3$ choices of $(j_1,j_2,r)\in[V]^3$ such that
  \begin{multline*}
      \left|\!\left|\E_{z \in [(N/W_d)^{1/d}V^{-1}]} \Delta_{\uh_1 \be_1,\uh_2 \be_2} f_{1,j_1+\widetilde P(r),j_2}(\cdot+\widetilde P_{[r, V]}(z),\cdot)\right.\right.\\
      \left.\left. \Delta_{\uh_1 \be_1,\uh_2 \be_2} f_{2,j_1, j_2+\widetilde P(r)}(\cdot, \cdot+\widetilde P_{[r, V]}(z))\vphantom{\E_{z \in [(N/W_d)^{1/d}V^{-1}]}}\right|\!\right|_{(\be_1\cdot[\pm N/V])^2,\; \be_2\cdot[\pm N/V]}^8\\
      \geq \exp(-\log(1/\delta)^{O(1)})(N/V)^2.
  \end{multline*}

For the next several steps, fix such a choice of $j_1,j_2,r$.
For each ``good'' choice of $j_1,j_2,r$ from the previous step, the $U^2 \times U^1$-inverse theorem (Lemma \ref{L: U^2 x U^1 inverse}) tells us that there exist a $1$-bounded function $g: \mathbb{Z} \to \mathbb{C}$ supported on $[N]$ and phase functions $a,b: \Z \to \mathbb{R}/\mathbb{Z}$ (all depending on $\uh_1,\uh_2, j_1,j_2,r$) such that
\begin{multline*}
\sum_{x,y}\E_{z \in [(N/W_d)^{1/d}V^{-1}]} \Delta_{\uh_1 \be_1,\uh_2 \be_2} f_{1,j_1+\widetilde P(r),j_2}(x+\widetilde P_{[r, V]}(z),y)\\
\Delta_{\uh_1 \be_1,\uh_2 \be_2} f_{2,j_1, j_2+\widetilde P(r)}(x,y+\widetilde P_{[r, V]}(z)) 
g(x)e(a(y)x+b(y))\geq \exp(-\log(1/\delta)^{O(1)}) (N/V)^2.
\end{multline*}
Hiding the dependence on $\uh_1, \uh_2, j_1,j_2,r$, we define
\begin{align*}
    &\widetilde f_1(x,y):=e(b(y))\Delta_{\uh_1 \be_1,\uh_2 \be_2} f_{1,j_1+\widetilde P(r),j_2}(x,y)\\
    \quad \text{and} \quad &\widetilde f_2(x,y):=g(x)\Delta_{\uh_1 \be_1,\uh_2 \be_2} f_{2,j_1, j_2+\widetilde P(r)}(x,y);
\end{align*}
then the previous inequality can be rewritten as
\begin{multline*}
    \sum_{x,y}\E_{z \in [(N/W_d)^{1/d}V^{-1}]} e(a(y)x) \widetilde f_1(x+\widetilde P_{[r, V]}(z),y) \widetilde f_2(x,y+\widetilde P_{[r, V]}(z))\\
    \geq \exp(-\log(1/\delta)^{O(1)}) (N/V)^2.
\end{multline*}

\subsection{Examining the structure of $a$}

The main task is showing that the phase function $a$ possesses some extra structure; specifically, we will show that $a(y)$ is approximately constant much of the time.  For future reference, define the counting operator
$$\Lambda'(a,\widetilde{f_1},\widetilde{f_2}):=\sum_{x,y} \E_{z \in [(N/W_d)^{1/d}V^{-1}]} e(a(y)x) \widetilde{f_1}(x+\widetilde P_{[r, V]}(z),y) \widetilde{f_2}(x,y+\widetilde P_{[r, V]}(z)).$$
The key structural result for the function $a$ is the following proposition.

\begin{proposition}\label{prop:popular-a-value}
 For every $C_1>0$ there exists $C_2=C_2(C_1,P)>0$ such that the following holds. Let $\delta\in(0, 1/10)$, let $V\leq W^{C_1}$ be a nonnegative integer power of $W$, and let $r\in[V]$.  Let $\widetilde{f_1},\widetilde{f_2}: \Z^2 \to \mathbb{C}$ be $1$-bounded functions supported on $[N/V]^2$, for some integer $N\geq C_2 V^{C_2}\exp(\log(1/\delta)^{C_2})$ , and let $a: \Z \to \mathbb{R}/\mathbb{Z}$ be a phase function.  If
$$|\Lambda'(a,\widetilde{f_1}, \widetilde{f_2})| \geq \delta (N/V)^{2},$$
then there exist a positive integer power $V' = V^{O_{C_1,P}(1)}$ 
and phases $\alpha_0, \ldots, \alpha_{V'-1}\in\R/\Z$ 
such that for at least $\exp(-\log(1/\delta)^{O_{C_1,P}(1)})V'$ values of $j \in[V']$, the bound
$$\|V'a_{(j, V')}(y) -\alpha_j\|_{\mathbb{R}/\mathbb{Z}} \leq \exp(\log(1/\delta)^{O_{C_1,P}(1)}) VV'/N$$
holds for at least $\exp(-\log(1/\delta)^{O_{C_1,P}(1)}) N/(VV')$  choices of $y \in [N/(VV')]$.
\end{proposition}

Unfortunately, this set of $y$'s where the phase can be well-approximated by constants may contribute little to the value of $\Lambda'(a,\widetilde{f_1}, \widetilde{f_2})$. To obtain a consequence better suited to our purposes, we bootstrap Proposition \ref{prop:popular-a-value} and produce some such set of $y$'s that does make a significant contribution to $\Lambda'(a, \widetilde{f_1}, \widetilde{f_2})$. The exact statement is as follows.

\begin{proposition}\label{prop:iterate-popular-a-value}
For every $C_1>0$ there exists $C_2=C_2(C_1,P)>0$ such that the following holds. Let $\delta\in(0, 1/10)$, let $V\leq W^{C_1}$ be a nonnegative integer power of $W$, and let $r\in[V]$. Let $\widetilde{f_1},\widetilde{f_2}: \Z^2 \to \mathbb{C}$ be $1$-bounded functions supported on $[N/V]^2$, for some integer $N\geq C_2 V^{C_2}\exp(\log(1/\delta)^{C_2})$, and let $a: \Z \to \mathbb{R}/\mathbb{Z}$ be any function.  If
$$|\Lambda'(a,\widetilde{f_1}, \widetilde{f_2})| \geq \delta (N/V)^{2},$$
then 
there exist a positive integer power $V' = V^{O_{C_1, P}(1)}$, an integer $$\exp(-\log(1/\delta)^{O_{C_1, P}(1)}))N\leq N'\leq N,$$ and phases $\alpha_{0}, \ldots, \alpha_{V'-1}\in \R/\Z$
such that
\begin{multline*}
    \sum_{j\in[V']}\sum_{i\in [N/N']}\sum_y\left|\sum_{x\in [N'/(VV')]} \E_{z \in [(N/W_d)^{1/d}V^{-1}]} e(\alpha_{y\!\!\!\!\pmod{V'}} x)\right.\\ 
     \widetilde{f_1}(V'x + j + \frac{N'}{V}i+\widetilde P_{[r, V]}(z),y)
     \left. \widetilde{f_2}(V'x + j + \frac{N'}{V}i,y+\widetilde P_{[r, V]}(z))\vphantom{\sum_x} \right|\\
     \geq \exp(-\log(1/\delta)^{O_{C_1, P}(1)})(N/V)^2.
\end{multline*}
\end{proposition}
Proposition \ref{prop:iterate-popular-a-value} essentially says that we can replace $a(y)$ by a phase function $y\mapsto \alpha_{y\pmod{V'}}$ that is constant on arithmetic progressions of difference $V'$. The summation over $i\in[N/N']$ is a technical annoyance, necessitated by the proof, that will be easily removed in our later application of Proposition \ref{prop:iterate-popular-a-value}.
\begin{proof}
Fix $C_1>0$. Proposition~\ref{prop:popular-a-value} applied with $\delta/2$ in place of $\delta$ tells us that the inequality 
 \begin{align}\label{E: Lambda' with delta/2}
    |\Lambda'({a},\widetilde{f_1}, \widetilde{f_2})| \geq (\delta/2) (N/V)^{2}    
 \end{align}
(for $a, \widetilde{f_1}, \widetilde{f_2}$ satisfying the conditions of Proposition~\ref{prop:popular-a-value}) implies that there exist $C:=C_2 = C_2(C_1, P)>0$ and $V' = V^{O_P(1)}$ such that, on setting $\tilde{V} = VV'$, we obtain a set $\CV\subset[V']$ of size $|\CV|\geq \exp(-\log(1/\delta)^{C})V'$, and for each $j'\in\CV$ we get a phase $\alpha_{j'} \in \mathbb{R}/\mathbb{Z}$ and a set $S_{j'} \subseteq [N/\tilde{V}]$ of size $|S_{j'}| \geq \exp(-\log(1/\delta)^{C}) N/\tilde{V}$ satisfying
\begin{align}\label{E: approximation of a}
    \|V' a_{({j'}, V')}(y')-\alpha_{j'}\|_{\R/\Z} \leq \exp(\log(1/\delta)^{C}) \tilde{V}/N
\end{align}
for all $y' \in S_{j'}$.  (Note that if this conclusion holds for $V'$, then it also holds with $V'$ replaced by a larger power of $V$, so we may assume that $V'$ is independent of the particular choice of $a, \widetilde f_1,\widetilde f_2$.)  For any constant $C'>C$, we can set $N'\sim (C')^{-1}\exp(-2\log(1/\delta)^{2C'})N$ and split the $x$-variable into progressions as
$$V'x + j + \frac{N'}{V}i\quad \textrm{for}\quad x\in [N'/\tilde{V}],\;\; j\in[V'],\;\; i\in [N/N'].$$ 
Then for each $j'\in\CV$ and $y'\in S_{j'}$, we can approximate
\begin{multline}\label{E: exponential sum approximation}
    \abs{e\brac{a_{({j'}, V')}(y')\brac{V'x + j+\frac{N'}{V}i}}-e\brac{\alpha_{j'} x + a_{({j'}, V')}(y')\brac{j+\frac{N'}{V}i}}}\\
    \leq \exp(\log(1/\delta)^{C})\frac{N'}{N}\leq  \exp(-\log(1/\delta)^{2C'})/2.
\end{multline}
Define the set
\begin{align*}
    \tilde{S} := \{y\in[N/V]:\; y = V'y'+j',\; j'\in\CV,\; y' \in S_{j'}\},
\end{align*}
which has size $|\tilde{S}| \geq \exp(-2\log(1/\delta)^{C}) N/V$.   Now the approximation \eqref{E: exponential sum approximation} gives
\begin{align*}
    &\left|\sum_{x} \E_{z \in [(N/W_d)^{1/d}V^{-1}]} e(a(y)x)\widetilde{f_1}(x+\widetilde P_{[r, V]}(z),y)\widetilde{f_2}(x,y+\widetilde P_{[r, V]}(z))\right.\\
     &\qquad- \sum_{j\in[V']}\sum_{i\in [N/N']}\sum_{x\in [N'/\tilde{V}]} \E_{z \in [(N/W_d)^{1/d}V^{-1}]} e\brac{\alpha_{y\!\!\!\!\pmod{V'}} x + a(y)\brac{j+\frac{N'}{V}i}}\\ 
     &\qquad\qquad\qquad\left. \widetilde{f_1}(V'x + j + \frac{N'}{V}i+\widetilde P_{[r, V]}(z),y)\widetilde{f_2}(V'x + j + \frac{N'}{V}i,y+\widetilde P_{[r, V]}(z))\vphantom{\sum_x} \right|\\
     &\qquad\qquad\qquad\qquad\qquad\leq (\exp(-\log(1/\delta)^{2C'})/2) N/V
\end{align*}
for each $y\in\tilde{S}$, 
and hence
\begin{align*}
    \sum_{y\in \tilde{S}}&\left|\sum_{x} \E_{z \in [(N/W_d)^{1/d}V^{-1}]} e(a(y)x)\widetilde{f_1}(x+\widetilde P_{[r, V]}(z),y)\widetilde{f_2}(x,y+\widetilde P_{[r, V]}(z))\right.\\
     &\qquad- \sum_{j\in[V']}\sum_{i\in [N/N']}\sum_{x\in [N'/\tilde{V}]} \E_{z \in [(N/W_d)^{1/d}V^{-1}]} e\brac{\alpha_{y\!\!\!\!\pmod{V'}} x + a(y)\brac{j+\frac{N'}{V}i}}\\ 
     &\qquad\qquad\qquad\left. \widetilde{f_1}(V'x + j + \frac{N'}{V}i+\widetilde P_{[r, V]}(z),y)\widetilde{f_2}(V'x + j + \frac{N'}{V}i,y+\widetilde P_{[r, V]}(z))\vphantom{\sum_x} \right|\\
     &\qquad\qquad\qquad\qquad\qquad\leq (\exp(-\log(1/\delta)^{2C'})/2)(N/V)^2.
\end{align*}
The upshot is that if $\Lambda'(a, \widetilde f_1, \widetilde f_2)$ is large, then we can use the approximation \eqref{E: approximation of a} provided by Proposition \ref{prop:popular-a-value} to approximate
\begin{align*}
    \Lambda'_{\tilde{S}}(a, \widetilde f_1, \widetilde f_2):=\sum_{y\in \tilde{S}}\abs{\sum_{x} \E_{z \in [(N/W_d)^{1/d}V^{-1}]} e(a(y)x)\widetilde{f_1}(x+\widetilde P_{[r, V]}(z),y)\widetilde{f_2}(x,y+\widetilde P_{[r, V]}(z))}
\end{align*}
by a sum closer to what we are seeking. Indeed, if 
\begin{align}\label{E: large Lambda'_S}
    \Lambda'_{\tilde{S}}(a, \widetilde f_1, \widetilde f_2)\geq \exp(-\log(1/\delta)^{2C'})(N/V)^2
\end{align}
for $C'>C$ as above, then the triangle inequality and the long inequality above imply that
\begin{multline*}
    \sum_{y}\sum_{j\in[V']}\sum_{i\in [N/N']}\left|\sum_{x\in [N'/\tilde{V}]} \E_{z \in [(N/W_d)^{1/d}V^{-1}]} e(\alpha_{y\!\!\!\!\pmod{V'}} x)\widetilde{f_1}(V'x + j + \frac{N'}{V}i+\widetilde P_{[r, V]}(z),y)\right.\\
    \left.\widetilde{f_2}(V'x + j + \frac{N'}{V}i,y+\widetilde P_{[r, V]}(z))\vphantom{\sum_x} \right|\geq (\exp(-\log(1/\delta)^{2C'})/2)(N/V)^2,
\end{multline*}
which establishes the conclusion of the proposition.

The lower bound \eqref{E: large Lambda'_S} does not always hold, however.  We will show that in this case, we can ``throw out'' the $y$-values in $\widetilde S$ and repeat the argument from the previous paragraph to obtain a new set $\widetilde S$ which is mostly disjoint from the set that we threw out; we can then iterate this procedure, which is guaranteed to terminate after a bounded number of steps because of the disjointness of the sets $\widetilde S$.  The details are as follows.

We first note that if $S \subseteq [N/\widetilde V]$ and $a':S \to \mathbb{R}/\mathbb{Z}$ is a function such that $V'a'(y)$ assumes each integer multiple of $1/|S|$ exactly once as $y$ ranges over $S$, then, for any interval $I \subseteq \mathbb{R}/\mathbb{Z}$, we have the bound
\begin{equation}\label{eq:pseudorandom}
|\{y \in S: V'a'(y) \in I\}| \leq 1+|S| \cdot |I| \leq 1+(N/\widetilde V)|I|.
\end{equation}
In this case, we will say that $a'$ is \emph{uniformly distributed} on  $S$.

We also note that if $S \subseteq [N/V]$ and $a'':[N/V] \to \mathbb{R}/\mathbb{Z}$ is any function that agrees with $a$ on $[N/V] \setminus S$, then
\begin{align}\label{E: equivalence of lambdas'}
    \Lambda'(a'', 1_{[N/V]\setminus S} \widetilde f_1, \widetilde f_2)=\Lambda'(a, 1_{[N/V]\setminus S} \widetilde f_1, \widetilde f_2)
\end{align}
(here we write $(1_{[N/V]\setminus S} \widetilde f_1)(x,y)$ for $1_{[N/V]\setminus S}(y) \widetilde f_1(x,y)$).  Indeed, there is no contribution from the values of $y$ where $a(y) \neq a''(y)$ because for such $y$ we have $(1_{[N/V]\setminus S} \widetilde f_1)(x+\widetilde P_{[r,V]}(z),y)=0$.

To begin our iterative procedure, set $\widetilde S_0=\emptyset$ and $a_0=a$.  For $\ell=1,2, \ldots$, we will construct phases $\alpha_{\ell,j'}$ (for $j' \in [V']$), a set $\widetilde S_\ell$, and a phase function $a_\ell$ that agrees with $a$ on the complement of $\widetilde S_{\leq \ell}:=\cup_{\ell'\leq \ell} \widetilde S_{\ell'}$ and is uniformly distributed on the intersection of $\widetilde S_{\leq \ell}$ with each arithmetic progression of common difference $V'$.  The $\ell$-th stage of the iteration begins with the inequality
$$\Lambda'(a, 1_{[N/V]\setminus \tilde{S}_{\leq \ell-1}}\widetilde f_1, \widetilde f_2)\geq (\delta/2)(N/V)^2$$
(which certainly holds for $\ell=1$).  Then, by \eqref{E: equivalence of lambdas'}, we also have
$$\Lambda'(a_{\ell-1}, 1_{[N/V]\setminus \tilde{S}_{\leq \ell-1}}\widetilde f_1, \widetilde f_2)\geq (\delta/2)(N/V)^2.$$
Arguing as in the first paragraph of the proof, with $a_{\ell-1}$ instead of $a$ and $1_{[N/V]\setminus \tilde{S}_{\leq \ell-1}}\widetilde f_1$ instead of $\widetilde f_1$, we obtain phases $\alpha_{\ell,j'}$ (for $j' \in [V']$) and a set $\widetilde S_\ell$.  If \eqref{E: large Lambda'_S} holds, then we obtain the conclusion of the proposition and our iterative procedure terminates.  If not, then, by induction on $\ell$ and the triangle inequality, we have
\begin{equation}\label{eq:iterative-loss-lambda'}
\Lambda'(a, 1_{[N/V]\setminus \widetilde S_{\leq \ell}}\widetilde f_1, \widetilde f_2) \geq (\delta - \ell\exp(-\log(1/\delta)^{C'}))(N/V)^2 > (\delta/2) (N/V)^2
\end{equation}
as long as $\ell<(\delta/2)\exp(\log(1/\delta)^{C'})$.  Now let $a_\ell$ be a function that agrees with $a$ on $[N/V] \setminus \widetilde S_{\leq \ell}$ and that is uniformly distributed on the intersection of $\widetilde S_{\leq \ell}$ with each arithmetic progression of common difference $V'$.  This completes the $\ell$-th step of the iteration and puts us in a position to perform the $(\ell+1)$-th step.

It remains to analyze the sizes of the sets $\widetilde S_{\leq \ell}$ and argue that the procedure terminates after a bounded number of steps.  By \eqref{eq:pseudorandom} applied to each of the sets $S_{\ell,j'}$ (for $j' \in [V']$), we have
$$|\widetilde S_\ell \cap \widetilde S_{\leq \ell-1}| \leq V'(1+2(N/\widetilde V)\exp(\log(1/\delta)^C)\widetilde V/N)<3V'\exp(\log(1/\delta)^C).$$
Hence
\begin{align*}
|\widetilde S_{\leq \ell}| &>|\widetilde S_{\leq \ell-1}|+|\widetilde S_\ell|-3V'\exp(\log(1/\delta)^C)\\
 &\geq |\widetilde S_{\leq \ell-1}|+\exp(-2\log(1/\delta)^C)N/V-3V'\exp(\log(1/\delta)^C)\\
 & \geq |\widetilde S_{\leq \ell-1}|+(1/2)\exp(-2\log(1/\delta)^C)N/V
\end{align*}
(using the assumption that $N$ is large), and by induction we obtain
$$|\widetilde S_{\leq \ell}|>(\ell/2)\exp(-2\log(1/\delta)^C)N/V.$$
Since $\widetilde S_{\leq \ell}$ is a subset of $[N/V]$, we conclude that the iterative procedure must terminate in at most $2\exp(-2\log(1/\delta)^C)$ steps, and, in particular, a suitable choice of $C'$ guarantees that the estimate in \eqref{eq:iterative-loss-lambda'} holds for all $\ell$ in the iterative procedure.
\end{proof}

\subsection{Auxiliary norm control on $f_1$.}\label{S: auiliary}
Recall from before that we have the lower bound
\begin{multline*}
\sum_{x,y}\E_{z \in [(N/W_d)^{1/d}V^{-1}]} \Delta_{\uh_1 \be_1,\uh_2 \be_2} f_{1,j_1+\widetilde P(r),j_2}(x+\widetilde P_{[r, V]}(z),y)\\
\Delta_{\uh_1 \be_1,\uh_2 \be_2} f_{2,j_1, j_2+\widetilde P(r)}(x,y+\widetilde P_{[r, V]}(z)) 
g(x)e(a(y)x+b(y))\geq \exp(-\log(1/\delta)^{O(1)}) (N/V)^2 
\end{multline*}
for at least $$\exp(-\log(1/\delta)^{O(1)})(N/V)^{k+\ell-3}$$ tuples $(\uh_1, \uh_2)\in\Z^{k+\ell-3}$ and at least $\exp(-\log(1/\delta)^{O(1)})V^3$ values $(j_1, j_2, r) \in [V]^3$, where $g,a,b$ depend on $\uh_1, \uh_2, j_1,j_2, r$.
Proposition~\ref{prop:iterate-popular-a-value} gives an integer $$\exp(-\log(1/\delta)^{O(1)})N\leq N'\leq N,$$ a positive integer power $V'=V^{O(1)}$, and $1$-bounded functions $b_{1,j,j',i}, b_{2,j,j',i}:\Z\to\C$ (the latter dependent on $\uh_1, \uh_2, j_1,j_2, r$) such that
\begin{multline*}
    \sum_{i\in [N/N']}\sum_{j,j'\in[V']}\sum_y\sum_{x\in [N'/(VV')]} \E_{z \in [(N/W_d)^{1/d}V^{-1}]} b_{1,j,j',i}(x)b_{2,j,j',i}(y)\\ 
     \Delta_{\uh_1 \be_1,\uh_2 \be_2}f_{1, j_1 + \widetilde P(r), j_2}(V'x + j + \frac{N'}{V}i+\widetilde P_{[r, V]}(z),V'y+j')\\
     \Delta_{\uh_1 \be_1,\uh_2 \be_2}f_{1, j_1, j_2 + \widetilde P(r)}(V'x + j + \frac{N'}{V}i,V'y+j'+\widetilde P_{[r, V]}(z))\\
     \geq \exp(-\log(1/\delta)^{O(1)})(N/V)^2
\end{multline*}
for at least $\exp(-\log(1/\delta)^{O(1)}) V^3$ values $(j_1, j_2, r) \in [V]^3$. By the pigeonhole principle, we can choose a single $i\in[N/N']$ for which the sum over $j,j',y,x,z$ satisfies the lower bound, at the cost of losing an extra factor of $\exp(-\log(1/\delta)^{O(1)})$, modifying $b_{1,j,j',i}$ by a phase if necessary, and dropping the index $i$ from $b_{1,j,j',i}$, $b_{2,j,j',i}$. We then $V'$-trick the polynomial $\widetilde P_{[r,V]}$, simplify the notation by setting
\begin{equation}\label{E: f^l}
\begin{split}
    f^l_{j_1, j_2, j, j'}(x,y) &:= \Delta_{\uh_1 \be_1,\uh_2 \be_2}f_{l, j_1, j_2}(V'x + j + \frac{N'}{V}i,V'y+j')\\
    &=\Delta_{V\uh_1 \be_1,V\uh_2 \be_2}f_l(VV'x + Vj + j_1+ N'i,VV'y+ Vj'+j_2)
\end{split}
\end{equation}
for $l=1,2$, and shift $y\mapsto y - \widetilde P_{[Vr'+r, VV']}(z)$, obtaining
\begin{multline*}
    \sum_{j,j',r'\in[V']}\sum_y\sum_{x\in [N'/(VV')]} \E_{z \in [(N/W_d)^{1/d}(VV')^{-1}]} b_{1,j,j'}(x)b_{2,j,j'}(y- \widetilde P_{[Vr'+r, VV']}(z))\\ 
     f^1_{j_1 + \widetilde P(r), j_2, j + \widetilde P_{[r, V]}(r'), j'}(x+ \widetilde P_{[Vr'+r, VV']}(z),y- \widetilde P_{[Vr'+r, VV']}(z))\\
     f^2_{j_1, j_2 + \widetilde P(r), j, j'+ \widetilde P_{[r, V]}(r')}(x,y)
     \geq \exp(-\log(1/\delta)^{O(1)})(N/V)^2
\end{multline*}
for the aforementioned values $(j_1, j_2, r) \in [V]^3$. Importantly, this new expression is amenable to an application of a S\'ark\"ozy-type estimate. By Corollary~\ref{C: Sarkozy} applied with $$b_{1,j,j'}(x)1_{[N'/(VV')]}(x) f^2_{j_1, j_2 + \widetilde P(r), j, j'+ \widetilde P_{[r, V]}(r')}(x,y),
\;\; b_{2,j,j'}(y) f^1_{j_1 + \widetilde P(r), j_2, j + \widetilde P_{[r, V]}(r'), j'}(x,y)$$ in place of $f_0, f_1$ and with $\bv = \be_1-\be_2$, we obtain the lower bound
\begin{multline*}
    \sum_{j,j',r'\in[V']}\sum_{x,y} \abs{\sum_{z\in[\pm N''/(VV')]} b_{2,j,j'}(y-qz) f^1_{j_1 + \widetilde P(r), j_2, j + \widetilde P_{[r, V]}(r'), j'}(x  + qz, y - qz)}\\ \geq \exp(-\log(1/\delta)^{O(1)}) (N/V)^3
\end{multline*}
for some positive integers $q\leq \exp(\log(1/\delta)^{O(1)})$ and $\exp(-\log(1/\delta)^{O(1)})N\leq N''\leq N$. We apply the Cauchy--Schwarz inequality to remove the absolute value and then change variables $(x,y)\mapsto (x-qz, y + qz)$ and $z\mapsto -z$, so that
\begin{multline*}
    \sum_{j,j',r'\in[V']}\sum_{x, y, z}\mu_{N''/(VV')}(z) \Delta_{qz} b_{2,j,j'}(y) \Delta_{qz(\be_2 - \be_1)} f^1_{j_1 + \widetilde P(r), j_2, j + \widetilde P_{[r, V]}(r'), j'}(x, y)\\ \geq \exp(-\log(1/\delta)^{O(1)})V'(N/V)^2.
\end{multline*}
In order to get rid of $b_{2,j,j'}$, we introduce an additional averaging over $[\pm N/(VV')]$ in the $x$-direction and then apply the Cauchy--Schwarz inequality to double this new variable. Using Lemma \ref{L: properties of box norms}\eqref{i: enlarging} to remove $q$, and summing over the admissible values of $j_1, j_2, r\in[V]$, we thus obtain
\begin{multline*}
    \sum_{\substack{j_1, j_2, r\in [V],\\ j,j',r'\in[V']}}\|f^1_{j_1 + \widetilde P(r), j_2, j + \widetilde P_{[r, V]}(r'), j'}\|_{(\be_2 - \be_1)\cdot[\pm N/(VV')],\; \be_1\cdot[\pm N/(VV')]}^4\\
    \geq \exp(-\log(1/\delta)^{O(1)}) VV' N^2.
\end{multline*}

We now want to remove the shifts by $\widetilde P(r)$ and $\widetilde P_{[r, V]}(r')$ from the indices $j_1, j'$. 
Notice first that the function $f^1_{j_1,j_2, j+s_3V', j'+s_4V'}$ is a translate of the function $f^1_{j_1,j_2, j, j'}$ and hence $\norm{f^1_{j_1,j_2, j,j'}}_{(\be_2 - \be_1)\cdot[\pm N/(VV')],\; \be_1\cdot[\pm N/(VV')]}$ depends only on the residue classes of $j,j'$ modulo $V'$. We can therefore replace the summation over $j,j'\in[V']$ by a summation over $j,j'\in\Z/V'\Z$ and then shift $j \mapsto j - \widetilde P_{[r, V]}(r')$ to obtain
\begin{multline*}
    \sum_{\substack{j_1, j_2, r\in [V]}}\sum_{j,j'\in\Z/V'\Z}\|f^1_{j_1 + \widetilde P(r), j_2, j, j'}\|_{(\be_2 - \be_1)\cdot[\pm N/(VV')],\; \be_1\cdot[\pm N/(VV')]}^4\\
    \geq \exp(-\log(1/\delta)^{O(1)}) V N^2.
\end{multline*}
Similarly, the identity
\begin{align*}
    f^1_{j_1+s_1V,j_2+s_2V, j, j'} = f^1_{j_1,j_2, j+s_1, j'+s_2},
\end{align*}
makes $\sum_{j,j'\in\Z/V'\Z}\|f^1_{j_1 + \widetilde P(r), j_2, j, j'}\|_{(\be_2 - \be_1)\cdot[\pm N/(VV')],\; \be_1\cdot[\pm N/(VV')]}^4$ depend only on the residue class of $j_1, j_2$ modulo $V$, and so we can replace the sum over $j_1, j_2\in[V]$ by the sum over $j_1, j_2\in\Z/V\Z$. Shifting $j_1\mapsto j_1 - \tilde{P}(r)$, we then get
\begin{multline*}
    \sum_{\substack{j_1, j_2\in \Z/V\Z}}\sum_{j,j'\in\Z/V'\Z}\|f^1_{j_1, j_2, j, j'}\|_{(\be_2 - \be_1)\cdot[\pm N/(VV')],\; \be_1\cdot[\pm N/(VV')]}^4\\
    \geq \exp(-\log(1/\delta)^{O(1)}) N^2,
\end{multline*}
and so
$$\|\Delta_{V\uh_1 \be_1,V\uh_2 \be_2}f_1\|_{(\be_2 - \be_1)VV' \cdot [\pm N/(VV')],\; \be_1VV' \cdot [\pm N/(VV')]}^4 \geq \exp(-\log(1/\delta)^{O(1)}) N^2.$$
Since this holds for at least $\exp(-\log(1/\delta)^{O(1)})(N/V)^{k+\ell-3}$ choices of $(\uh_1, \uh_2)\in\Z^{k+\ell-3}$, we get from properties \eqref{i: inductive formula} and \eqref{i: passing to APs} of Lemma \ref{L: properties of box norms} that
\begin{multline}\label{E: auxiliary bound}
    \|f_1\|^{2^{k+\ell-1}}_{(\be_1 VV' \cdot [\pm N/(VV')])^{k - 1},\; (\be_2 - \be_1)VV' \cdot [\pm N/(VV')],\; (\be_2 VV' \cdot [\pm N/(VV')])^{\ell-1}}\\ \geq \exp(-\log(1/\delta)^{O(1)}) N^2.
\end{multline}

\subsection{Completing the proof of Proposition~\ref{prop:degreeloweringdual}}\label{S: completing}
So far, we have proved that if $$|\Lambda^W(f_0, f_1, f_2)| \ge \delta N^2$$
implies that
$$\|f_0\|_{(\be_1V \cdot [\pm N/V])^k,\; (\be_2 V \cdot [\pm N/V])^\ell}^{2^{k+\ell}} \geq \exp(-\log(1/\delta)^{O(1)}) N^2,$$
then the norm-control \eqref{E: auxiliary bound} also follows. This lower bound is not particularly useful in itself, but we will use it as an intermediate step towards the desired norm-control on $f_0$. To this end, we first use stashing to replace $f_1$ in $\Lambda^W$ with the dual function
\begin{align*}
    \CD_1(f_0,f_2)(x,y):=\E_{z \in [(N/W_d)^{1/d}} f_0(x-\widetilde P(z),y)f_2(x-\widetilde P(z),y+\widetilde P(z)),
\end{align*}
so that \eqref{E: auxiliary bound} holds with $\CD_1(f_0,f_2)$ in place of $f_1$. We then apply the dual--difference interchange (Lemma~\ref{lem:dual-diff}) to this expression, so that
\begin{multline*}
    \sum_{\substack{\uh_1 \in \mathbb{Z}^{k-1},\\ \uh_2 \in \mathbb{Z}^{\ell-1}}} \mu_{N/(VV')}(\uh_1,\uh_2)\cdot \|\mathcal{D}_1^{\uh_1,\uh_2}(f_0,f_2)\|_{(\be_2-\be_1)VV' \cdot [\pm N/(VV')]}^2\\
    \geq \exp(-\log(1/\delta)^{O(1)}) N^2,
\end{multline*}
where
$$\mathcal{D}_1^{\uh_1,\uh_2}(f_0,f_2)(x,y):=\E_{z \in [(N/W_d)^{1/d}} \Delta_{VV'\uh_1\be_1,VV'\uh_2 \be_2}(f_0(x-\widetilde P(z),y)f_2(x-\widetilde P(z),y+\widetilde P(z))).$$
Thus for at least $\exp(-\log(1/\delta)^{O(1)})(N/(VV'))^{k+\ell-2}$ values of $(\uh_1, \uh_2)\in\Z^{k+\ell-2}$, we have the lower bound
$$\|\mathcal{D}_1^{\uh_1,\uh_2}(f_0,f_2)\|_{(\be_2-\be_1)VV' \cdot [\pm N/(VV')]}^2 \geq \exp(-\log(1/\delta)^{O(1)}) N^2.$$
We temporarily fix such a choice of $(\uh_1, \uh_2)\in\Z^{k+\ell-2}$ and repeat the above procedure of splitting the domains of $x,y$ into arithmetic progressions of difference $VV'$. This yields
\begin{align*}
    &\mathcal{D}_1^{\uh_1,\uh_2}(f_0,f_2)(VV'x+j_1,VV'y+j_2)\\
    &\qquad\qquad=\E_{z \in [(N/W_d)^{1/d}]} \Delta_{VV'\uh_1\be_1,VV'\uh_2 \be_2}(f_0(VV'x+j_1-\widetilde P(z),VV'y+j_2)\\
    &\qquad\qquad\qquad\qquad\qquad\qquad\qquad\times f_2(VV'x+j_1-\widetilde P(z),VV'y+j_2+\widetilde P(z)))\\
    &\qquad\qquad=\E_{r\in[VV']}\E_{z \in [(N/W_d)^{1/d}(VV')^{-1}]} \Delta_{\uh_1\be_1,\uh_2 \be_2}(f_{0, j_1 - \widetilde P(r), j_2}(x-\widetilde P_{[r,VV']}(z),y)\\
    &\qquad\qquad\qquad\qquad\qquad\qquad\qquad\times f_{2, j_1 - \widetilde P(r), j_2 + \widetilde P(r)}(x-\widetilde P_{[r,VV']}(z),y+\widetilde P_{[r,VV']}(z))),
\end{align*}
where $f_{l, j_1, j_2}(x,y) = (f_l)_{(j_1, j_2, VV')}(x,y) = f_l(VV'x + j_1, VV'y+j_2)$ for $l=0,2$.
Hence
\begin{multline*}
    \E_{j_1, j_2, r\in[VV']} \left|\!\left|\E_{z \in [(N/W_d)^{1/d}(VV')^{-1}]} \Delta_{\uh_1\be_1,\uh_2 \be_2}f_{0, j_1 - \widetilde P(r), j_2}(\cdot-\widetilde P_{[r,VV']}(z),\cdot)\right.\right.\\
    \left.\left.\Delta_{\uh_1\be_1,\uh_2 \be_2} f_{2, j_1 - \widetilde P(r), j_2 + \widetilde P(r)}(\cdot-\widetilde P_{[r,VV']}(z),\cdot+\widetilde P_{[r,VV']}(z)))\vphantom{\E_{z \in [(N/W_d)^{1/d}(VV')^{-1}]}}\right|\!\right|_{(\be_2-\be_1) \cdot [\pm N/(VV')]}^2\\
    \geq \exp(-\log(1/\delta)^{O(1)}) (N/(VV'))^2.
\end{multline*}
Applying the inverse theorem for the $U^1$ norm (Lemma \ref{L: U^1 inverse})\footnote{Technically, to apply this result, the length of the boxes should be larger by a factor $\exp(\log(1/\delta)^{O(1)})$ than the support of the function whose norm we are evaluating. This can be ensured by replacing the box $(\be_2-\be_1)VV' \cdot [\pm N/(VV')]$ by $(\be_2-\be_1)VV' \cdot [\pm N'/(VV')]$ (with $N' \asymp \exp(\log(1/\delta)^{C})N$ for an appropriate $C>0$) before performing the Dual-Difference Interchange using Lemma \ref{L: properties of box norms}\eqref{i: enlarging}.} for the values $(j_1, j_2, r)\in[VV']^3$ where the above norm is at least $\exp(-\log(1/\delta)^{O(1)}) (N/(VV'))^2$, we obtain a $1$-bounded function $b: \mathbb{Z} \to \mathbb{C}$ (depending on $\uh_1, \uh_2, j_1, j_2, r$) such that
\begin{multline*}
    \sum_{x,y}\E_{z \in [(N/W_d)^{1/d}(VV')^{-1}]} b(x+y) \Delta_{\uh_1 \be_1,\uh_2 \be_2}f_{0, j_1 - \widetilde P(r), j_2}(x-\widetilde P_{[r,VV']}(z),y)\\
    \Delta_{\uh_1 \be_1,\uh_2 \be_2}f_{2, j_1 - \widetilde P(r), j_2 + \widetilde P(r)}(x-\widetilde P_{[r,VV']}(z),y+\widetilde P_{[r,VV']}(z))\\
    \geq \exp(-\log(1/\delta)^{O(1)}) (N/(VV'))^2.
\end{multline*}
Changing variables $(x,y) \mapsto (x+\widetilde P_{[r,VV']}(z), y-\widetilde P_{[r,VV']}(z))$ simplifies this expression to
\begin{multline*}
    \sum_{x,y}\E_{z \in [(N/W_d)^{1/d}(VV')^{-1}]} b(x+y) \Delta_{\uh_1 \be_1,\uh_2 \be_2}f_{0, j_1 - \widetilde P(r), j_2}(x,y-\widetilde P_{[r,VV']}(z))\\
    \Delta_{\uh_1 \be_1,\uh_2 \be_2}f_{2, j_1 - \widetilde P(r), j_2 + \widetilde P(r)}(x,y)\\
    \geq \exp(-\log(1/\delta)^{O(1)}) (N/(VV'))^2.
\end{multline*}
After an application of Corollary~\ref{C: Sarkozy}, with
\begin{align*}
 b(x+y)\Delta_{\uh_1 \be_1,\uh_2 \be_2}f_{2, j_1 - \widetilde P(r), j_2 + \widetilde P(r)}(x,y)\quad \textrm{and}\quad \Delta_{\uh_1 \be_1,\uh_2 \be_2}f_{0, j_1 - \widetilde P(r), j_2}(x,y) 
\end{align*}
in place of $f_0(x,y), f_1(x,y)$ and with $\bv = \be_2$), maneuvers similar to those in the proof of Theorem \ref{thm:degree lowering-lambda} give  
    $$\|f_{0}\|_{(\be_1\tilde{V} \cdot [N/\tilde{V}])^{k-1},\; (\be_2\tilde{V} \cdot [N/\tilde{V}])^{\ell}}^{2^{k+\ell-1}} \geq \exp(-\log(1/\delta)^{O(1)}) N^2$$
    with $\tilde{V} = VV'.$
Exchanging the roles of $x$ and $y$ gives the last assertion of the proposition.

\section{Reducing Proposition \ref{prop:popular-a-value} to a nilsequence problem}\label{S: lambda'}
The next two sections are devoted to the proof of Proposition \ref{prop:popular-a-value}; this will complete the degree-lowering argument from the previous section. Recall the counting operator
$$\Lambda'(a,\widetilde{f_1},\widetilde{f_2})=\sum_{x,y} \E_{z \in [(N/W_d)^{1/d}V^{-1}]} e(a(y)x) \widetilde{f_1}(x+\widetilde P_{[r, V]}(z),y) \widetilde{f_2}(x,y+\widetilde P_{[r, V]}(z));$$
our goal is to show that if $\Lambda'(a, \widetilde{f_1}, \widetilde{f_2})$ is large for some $1$-bounded functions $\widetilde{f_1}, \widetilde{f_2}$, then there are many arithmetic progressions of difference $V' = V^{O_P(1)}$ on each of which the phase $a$ is approximately constant. In this section we will use higher-order Fourier analysis manipulations to replace the functions $\widetilde f_1,\widetilde f_2$ with nilsequences; in the next section we will use sophisticated machinery from the theory of nilsequences to complete the proof of Proposition \ref{prop:popular-a-value}. Our starting point is the following norm-control.

\begin{proposition}\label{prop:lambda_1}
For every $C_1>0$ there exists $C_2=C_2(C_1,P)>0$ such that the following holds. Let $\delta\in(0,1/10)$, let $V \leq W^{C_1}$ be a nonnegative integer power of $W$, and let $r\in[V]$. Let $\widetilde{f_1},\widetilde{f_2}: \Z^2 \to \mathbb{C}$ be $1$-bounded functions supported on $[N/V]^2$, for some integer $N\ge C_2V^{C_2}\delta^{-C_2}$, and let $a: \mathbb{Z} \to \mathbb{R}/\mathbb{Z}$ be a phase function.  If
$$|\Lambda'(a,\widetilde{f_1}, \widetilde{f_2})| \geq \delta (N/V)^{2},$$
then we have the norm-control
\begin{align*}
&\|\widetilde{f_2}\|_{(\be_2-\be_1)\cdot  [\pm N/V],\; \be_1\cdot [\pm N/V]}^4 \gg_{C_1, P} \delta^{O_{C_1,P}(1)}(N/V)^2.    
\end{align*}
\end{proposition}

Na\"ively, one would expect to have
$\{\be_2-\be_1,\; \be_2\}$-control on $\widetilde{f_2}$.  Our norm-control looks slightly different because, as will become clear in the proof, taking discrete derivatives in the $\be_1$-direction is very convenient for dealing with the $e(a(y)x)$ term in $\Lambda'$. The argument would give the same conclusion with $\|\widetilde{f_1}\|_{(\be_2-\be_1)\cdot [\pm N/V],\; (\be_1\cdot [\pm N/V])^2}$, but we do not need this.

\begin{proof}
     We fix $C_1>0$ and let all the constants depend on $C_1, P$. 
     We introduce an extra averaging over $[N/V]$ in the $x$-direction\footnote{If we did not want the range of the new variable $h$ to be restricted, then we could apply the Cauchy--Schwarz inequality directly.} and apply the Cauchy--Schwarz inequality in order to double the variable $x$; this gives
\begin{multline}\label{E: first CS}
    \sum_{x,y,h}\mu_{N/V}(h)\E_{z \in [(N/W_d)^{1/d}V^{-1}]}
    e(a(y)h)\\ \Delta_{h\be_1}\widetilde{f_1}(x + \widetilde P_{[r, V]}(z), y)\Delta_{h\be_1}\widetilde{f_2}(x, y+ \widetilde P_{[r, V]}(z)) \geq \delta^2 (N/V)^2.
\end{multline}
Shifting $x\mapsto x - \widetilde P_{[r, V]}(z)$ gives
\begin{multline*}
    \sum_{x,y,h}\mu_{N/V}(h)\E_{z \in [(N/W_d)^{1/d}V^{-1}]}
    e(a(y)h)\\ \Delta_{h\be_1}\widetilde{f_1}(x, y)\Delta_{h\be_1}\widetilde{f_2}(x - \widetilde P_{[r, V]}(z), y+ \widetilde P_{[r, V]}(z)) \geq \delta^2 (N/V)^2.
\end{multline*}
The crucial point is that now only $\widetilde f_2$ has $\widetilde P_{[r, V]}(z)$ in its argument and so we can apply the S\'ark\"ozy-type estimate (Corollary \ref{C: Sarkozy}) (after using the popularity principle to find many $h$'s for which the inner expression is large in absolute value). Hence there exist positive integers\footnote{A direct application of Corollary \ref{C: Sarkozy} gives $q_{h}$ that depends on $h$; however, $q$ can be chosen uniformly by the pigeonhole principle since there are only $O(\delta^{-O(1)})$ possibilities for $q_{h}$.} $q\ll \delta^{-O(1)}$ and $\delta^{O(1)} N\ll N'\leq N$ such that
\begin{align*}
    \sum_{x,y,h}\mu_{N/V}(h)\abs{\E_{z\in[N'/V]}\Delta_{h\be_1}\widetilde{f_2}(x - qz, y + qz)} \gg \delta^{O(1)} (N/V)^2.
\end{align*}
Applying the Cauchy--Schwarz inequality and then using Lemma \ref{L: properties of box norms}\eqref{i: enlarging} to replace $N'$ with $N$ and drop $q$, we get the desired bound on $\widetilde f_2$. 
\end{proof}

The next result shows that the largeness of $\Lambda'(a,\widetilde{f_1}, \widetilde{f_2})$ for some functions $\widetilde{f_1}, \widetilde{f_2}$ implies some correlation of the phase $a$ with nilsequences on arithmetic progressions. We refer the reader to Definition \ref{D: nilsequence} for definition of a nilsequence.
\begin{proposition}\label{P: local correlation with nilsequences}
For every $C_1>0$ there exists $C_2=C_2(C_1,P)>0$ such that the following holds. Let $\delta\in(0,1/10)$, let $V \le W^{C_1}$ be a positive integer power of $W$, and let $r\in[V]$.
Let $a:\Z\to\R/\Z$ be a phase function, and let $\widetilde{f_1},\widetilde{f_2}: \Z^2 \to \mathbb{C}$ be $1$-bounded functions supported on $[N/V]^2$, for some integer $N\ge C_2V^{C_2}\exp(\log(1/\delta)^{C_2})$.  If
    $$|\Lambda'(a,\widetilde{f_1}, \widetilde{f_2})| \geq \delta (N/V)^{2},$$
    then there exist a positive integer $k=O_P(1)$,
    a positive integer power $V' = V^{O_{C_1, P}(1)}$, and a function $\varphi:\Z\to\C$ such that
\begin{multline*}
    \E_{j, r'\in[V']}\sum_{y\in[N/(VV')]}\left|\E_{z \in [(N/W_d)^{1/d}(VV')^{-1}]} e(-a_{(j, V')}(y)V'\cdot\widetilde{P}_{[Vr'+r,VV']}(z))\right.\\
    \left.\varphi_{({j} + \widetilde{P}_{[r,V]}(r'), V')}(y + \widetilde{P}_{[Vr'+r,VV']}(z)) \vphantom{\E_{z \in [(N/W_d)^{1/d}(VV')^{-1}]}} \right| \geq \exp(-\log(1/\delta)^{O_{C_1, P}(1)}) N/(VV')
\end{multline*}
and for each $j\in[V']$, the restriction $\varphi_{(j,V')}$ is a $1$-bounded, degree-$k$ nilsequence of complexity $\exp(\log(1/\delta)^{O_{C_1, P}(1)})$ on a $\log(1/\delta)^{O_{C_1, P}(1)}$-dimensional nilmanifold. 

\end{proposition}
\begin{proof}

The proof consists of massaging the expression $\Lambda'(a,\widetilde{f_1}, \widetilde{f_2})$, with the help of the same techniques as in the proof of Theorem \ref{thm:degree lowering-lambda},
until the structure of $a$ along arithmetic progressions becomes apparent. We fix $C_1>0$ and let all constants depend on $C_1, P$ without notating the dependence explicitly.

\smallskip
\textbf{Step 1: Replacing $\widetilde{f_2}$ by a box norm obstruction.}
\smallskip

Our first objective is to use Proposition \ref{prop:lambda_1} in conjunction with relevant box norm inverse theorems and stashing to replace $\widetilde{f_2}$ by a ``structured'' function.
By Proposition \ref{prop:lambda_1}, we have the lower bound
$$\|\widetilde{f_2}\|_{(\be_2-\be_1)\cdot [\pm N/V],\; \be_1\cdot[\pm N/V]}^4 \gg \delta^{O(1)}(N/V)^2.$$
The stashing argument gives
$$\|\mathcal{D}'_2(a, \widetilde{f_1})\|_{(\be_2-\be_1)\cdot [\pm N/V],\; \be_1\cdot  [\pm N/V]}^4 \gg \delta^{O(1)}(N/V)^2,$$
for the dual function
\begin{multline*}
\mathcal{D}'_{2}(a, \widetilde{f_1})(x,y):=\E_{z \in [(N/W_d)^{1/d}V^{-1}]}e(a(y-\widetilde{P}_{[r, V]}(z))x)\\
\widetilde{f_1}(x+\widetilde{P}_{[r, V]}(z),y-\widetilde{P}_{[r, V]}(z))1_{[N/V]}(y),
\end{multline*}
where we were able to insert the term $1_{[N/V]}$ because $\widetilde{f_2}$ is supported on $[N/V]^2.$
By the inverse theorem for the $U^1\times U^1$-box norm (Lemma \ref{L: U1xU1 inverse}), there are $1$-bounded functions $b_1, b_2:\Z\to\C$ such that 
\begin{multline*}
\sum_{x, y} \E_{z \in [(N/W_d)^{1/d}V^{-1}]} b_1(y) b_2(x+y) e(a(y - \widetilde{P}_{[r,V]}(z))x)\\
\widetilde{f_1}(x + \widetilde{P}_{[r,V]}(z), y - \widetilde{P}_{[r,V]}(z))1_{[N/V]}(y) \gg \delta^{O(1)}(N/V)^2,    
\end{multline*}
and the presence of the $1_{[N/V]}(y)$ term lets us assume that $b_1$ is supported on $[N/V]$.
The change of variables 
$y \mapsto y + \widetilde{P}_{[r,V]}(z)$ gives
\begin{multline*}
\sum_{x, y}\E_{z \in [(N/W_d)^{1/d}V^{-1}]} e(a(y)x)
\widetilde{f_1}(x+ \widetilde{P}_{[r,V]}(z), y)\\
b_1(y+ \widetilde{P}_{[r,V]}(z)) b_2(x + y + \widetilde{P}_{[r,V]}(z)) \gg \delta^{O(1)}(N/V)^2.
\end{multline*}
We observe that $b_2(x + y + \widetilde{P}_{[r,V]}(z))$ is a function of $(x+\widetilde{P}_{[r,V]}(z), y)$. Shifting $x\mapsto x - \widetilde{P}_{[r,V]}(z)$ and setting $\widetilde{f_1}'(x,y):= e(a(y)x)\widetilde{f_1}(x,y)b_2(x + y)$, we get 
\begin{align*}
\sum_{x, y}\E_{z \in [(N/W_d)^{1/d}V^{-1}]} e(-a(y)\widetilde{P}_{[r,V]}(z))
\widetilde{f_1}'(x, y)
b_1(y+ \widetilde{P}_{[r,V]}(z))  \gg \delta^{O(1)}(N/V)^2.
\end{align*}
Recall that $\widetilde{f_1}'$ is 1-bounded and supported on $[N/V]^2$.  Using the triangle inequality in the $x$- and $y$-variables, we can eliminate $\widetilde {f_1}'$ entirely and obtain
\begin{align*}
\sum_{y\in[N/V]}\abs{\E_{z \in [(N/W_d)^{1/d}V^{-1}]} e(-a(y)\widetilde{P}_{[r,V]}(z))
b_1(y+ \widetilde{P}_{[r,V]}(z))}\gg \delta^{O(1)}N/V.
\end{align*}
We have thus reduced matters to analyzing the auxiliary operator
\begin{align}\label{E: Lambda''}
    \Lambda''(b_1):= \sum_{y\in[N/V]}\abs{\E_{z \in [(N/W_d)^{1/d}V^{-1}]} e(-a(y)\widetilde{P}_{[r,V]}(z))
b_1(y+ \widetilde{P}_{[r,V]}(z))}.
\end{align}

\smallskip
\textbf{Step 2: Replacing $b_1$ by a nilsequence.}
\smallskip

Our next goal is to show that $\Lambda''$ can be controlled by a Gowers norm of $b_1$; this will let us replace $b_1$ by a nilsequence (at least on arithmetic progressions of not-too-small difference). To obtain Gowers norm control, we observe that $\Lambda''$ is amenable to a standard PET argument.  Let $b_3:\Z\to\C$ be a $1$-bounded function supported on $[N/V]$ such that 
$$\Lambda''(b_1) = \sum_{y}b_3(y)\E_{z \in [(N/W_d)^{1/d}V^{-1}]} e(-a(y)\widetilde{P}_{[r,V]}(z))
b_1(y+ \widetilde{P}_{[r,V]}(z)) \gg \delta^{O(1)}N/V.$$
Shifting $y \mapsto y-\widetilde P_{[r,V]}(z)$, stashing, and shifting back $y \mapsto y+\widetilde P_{[r,V]}(z)$, we find that
$$\sum_{y}\overline{b_3(y)}\E_{z \in [(N/W_d)^{1/d}V^{-1}]} e(a(y)\widetilde{P}_{[r,V]}(z))
\widetilde b_1(y+ \widetilde{P}_{[r,V]}(z))\gg \delta^{O(1)}N/V,$$
where
$$\widetilde b_1(y):= \E_{z \in [(N/W_d)^{1/d}V^{-1}]} e(-a(y-\widetilde{P}_{[r,V]}(z))\widetilde{P}_{[r,V]}(z))
b_3(y-\widetilde{P}_{[r,V]}(z))1_{[N/V]}(y).$$ Note that the indicator function $1_{[N/V]}(y)$ can be included since $b_1$ is supported on this set. 
Now, using the Cauchy--Schwarz inequality to double the $z$-variable $d+1$ times (in the style of PET), we eliminate $b_3(y)$ and the exponential involving $a(y)$; we obtain
\begin{multline*}
    \sum_{y, h_1, \ldots, h_{d+1}}\mu_{(N/W_d)^{1/d}V^{-1}}(h_1, \ldots, h_{d+1})\\ \E_{z \in [(N/W_d)^{1/d}V^{-1}]} \prod_{\ueps\in\{0,1\}^{d+1}}\CC^{|\ueps|}\widetilde b_1(y + \widetilde{P}_{[r,V]}(z+\ueps\cdot\uh)) \gg \delta^{O(1)}N/V.
\end{multline*}
By \cite[Proposition 10.1]{KKL24a} and the formula $V_d = \beta_d \beta_1^{d-2}V^{d-1}W^{d-1}$, we have
\begin{align*}
    \norm{\widetilde b_1}_{U^{k+1}(V_d\cdot[\pm N/(VV_d)])}^{2^{k+1}}\gg \delta^{O(1)} N/V
\end{align*}
for some integer $k=O(1)$. Take $V'=V^{O(1)}$ to be the smallest positive integer power of $V$ such that $V_d|\beta_d \beta_1 V'$.
Since $\beta_1,\; \beta_d=O(1)$, we can apply parts \eqref{i: enlarging} and \eqref{i: passing to APs} of Lemma \ref{L: properties of box norms} to replace the boxes $V_d\cdot[\pm N/(VV_d)]$ by $V'\cdot[\pm N/(VV')]$, obtaining
\begin{align*}
    \norm{\widetilde b_1}_{U^{k+1}(V'\cdot[\pm N/(VV')])}^{2^{k+1}}\gg \delta^{O(1)} N/V.
\end{align*}

Splitting the domain of $\tilde{b}_1$ into arithmetic progressions of difference $V'$, we find using the popularity principle that
\begin{align*}
    \norm{(\tilde{b}_1)_{(j, V')}}_{U^{k+1}([\pm N/(VV')])}^{2^{k+1}}\gg \delta^{O(1)} N/(VV')
\end{align*}
for an $\Omega(\delta^{O(1)})$-proportion of $j\in[V']$.  The $U^{k+1}$-inverse theorem \cite[Theorem 1.2]{LSS24} tells us that for each such $j\in[V']$, there is a $1$-bounded\footnote{$1$-boundedness is not mentioned in the statement of \cite[Theorem 1.2]{LSS24}; however, the control on the Lipschitz norms of $\varphi_j$ ensures that the nilsequences coming from \cite[Theorem 1.2]{LSS24} can be rescaled to be $1$-bounded at an acceptable loss of the factor $\exp(-\log(1/\delta)^{O(1)})$ in the lower bound.} degree-$k$ nilsequence $\varphi_j$ of complexity $\exp(\log(1/\delta)^{O(1)})$ on an $O(\log(1/\delta)^{O(1)})$-dimensional nilmanifold (see Definition \ref{D: nilsequence} for the definition of a nilsequence and its complexity)
such that
\begin{align*}
    \sum_y (\tilde{b}_1)_{(j,V')}(y)\varphi_j(y) \geq \exp(-\log(1/\delta)^{O(1)})N/(VV').
\end{align*}
For each remaining value of $j\in[V']$, set $\varphi_j:= 0$.
Thus, $\varphi_j$ is a $1$-bounded degree-$k$ nilsequence of complexity $\exp(\log(1/\delta)^{O(1)})$ on an $O(\log(1/\delta)^{O(1)})$-dimensional nilmanifold for \textit{each} $j\in [V']$.

Define $\varphi:\Z\to\C$ by setting $ \varphi_{(j,V')}(y)=\varphi(V'y + j)$ to equal $\varphi_j(y)$ for each $j\in[V']$ and $y\in\Z$, so that
\begin{align*}
    \sum_y \tilde{b}_1(y)\varphi(y) \geq \exp(-\log(1/\delta)^{O(1)})N/V.
\end{align*}
Expanding the definition of $\widetilde b_1$, shifting $y \mapsto y+\widetilde P_{[r,V]}(z)$, and using the triangle inequality to eliminate $b_3$, we obtain
\begin{align*}
    \sum_{y\in[N/V]}\left|\E_{z \in [(N/W_d)^{1/d}V^{-1}]} e(-a(y)\widetilde{P}_{[r,V]}(z)) \varphi(y + \widetilde{P}_{[r,V]}(z))\right| \geq \exp(-\log(1/\delta)^{O(1)}) N/V.
\end{align*}
Since $\varphi$ is a nilsequence only when restricted to arithmetic progressions of common difference $V'$, we also split the ranges of $y$ and $z$ into arithmetic progressions of common difference $V'$.  This gives
\begin{multline*}
    \E_{j, r'\in[V']}\sum_{y\in[N/(VV')]}\left|\E_{z \in [(N/W_d)^{1/d}(VV')^{-1}]} e(-a_{(j, V')}(y)V'\cdot\widetilde{P}_{[Vr'+r,VV']}(z))\right.\\
    \left.\varphi_{({j} + \widetilde{P}_{[r,V]}(r'), V')}(y + \widetilde{P}_{[Vr'+r,VV']}(z)) \vphantom{\E_{z \in [(N/W_d)^{1/d}(VV')^{-1}]}} \right| \geq \exp(-\log(1/\delta)^{O(1)}) N/(VV'),
\end{multline*}
as desired.
\end{proof}
Notice that the last step of splitting into progressions introduced a factor of $V'$ in the first term; this arises essentially because we $V'$-tricked the exponential function.

\section{Nilsequence application}\label{S: nilsequence}
In the previous section, we arrived at expressions of the form
\begin{align}\label{E: section 9 expression}
    \sum_{y\in [N/\tilde{V}]}\left|\E_{z \in [(N/W_d)^{1/d}\tilde{V}^{-1}]} e(\tilde{a}(y)\cdot\widetilde{P}_{[\tilde{r},\tilde{V}]}(z))\tilde{\varphi}(y + \widetilde{P}_{[\tilde{r},\tilde{V}]}(z)) \vphantom{\E_{z \in [(N/W_d)^{1/d}\tilde{V}^{-1}]}} \right|
\end{align}
with
\begin{align*}
    \Tilde{V} = VV',\quad \tilde{r} =  Vr'+r,\quad \tilde{\varphi} = \varphi_{({j} + \widetilde{P}_{[r,V]}(r'), V')},\quad \textrm{and}\quad\tilde{a} = -a_{(j, V')}V',
\end{align*}
where $\Tilde{\varphi}$ is a nilsequence. This section is devoted to extracting information about the phase $a$ using the largeness of \eqref{E: section 9 expression}.  The main underlying machinery comes from the equidistribution results of \cite{Leng23b}.  We refer the reader to Appendix \ref{A: nilsequence theory} for definitions and notation related to nilsequences.

Our main technical input is the following result, which is the special case $\ell = 1$ of \cite[Theorem 3]{Leng23b}.

\begin{theorem}\label{multiparameter}
For every $k>0$ there is a constant $C = C(k)>0$ such that the following holds. Let $G/\Gamma$ be a nilmanifold of degree $k$, step $s$, complexity at most $M$, and dimension $m$, with associated filtration $G_\bullet$. Let $\delta\in (0,1/10)$, and let $N > (M/\delta)^{Cm^C}$ be an integer. Let $F:G/\Gamma \to \mathbb{C}$ be a $1$-bounded $M$-Lipschitz function with a vertical character $\xi$ satisfying $0<|\xi| \leq M/\delta$.

Suppose $g \in \mathrm{poly}(\mathbb{Z}, G)$ with
\[\abs{\E_{n \in [N]} F(g(n)\Gamma)} \ge \delta.\]
Then there exist $0< r\le \mathrm{dim}(G/[G,G])$ and horizontal characters $\eta_1, \dots, \eta_r$ of size at most $(M/\delta)^{O_k(m)^{O_k(1)}}$ satisfying
    \[\|\eta_i \circ g\|_{C^\infty[{N}]} \le (M/\delta)^{O_k(m)^{O_k(1)}},\]
    and $\xi$ annihilates all $(s-1)$-fold commutators of elements in the subgroup $G' := \bigcap_{i=1}^r\mathrm{ker}(\eta_i)$.
\end{theorem}

The use of this theorem is that it allows one to iteratively ``reduce'' the step of $g$.

In the following proposition, we let $\Tilde{V}_d$ denote the leading coefficient of $\widetilde{P}_{[\tilde{r},\tilde{V}]}$.

\begin{proposition}\label{prop:Wtricknilsequenceprop}
Let $s,k\in\N$ with $s\leq k$. For every $C_1>0$ there exist constants $C_2=C_2(C_1, P, k),\; C_3=C_3(C_1, P, k)>1$ such that the following holds. Let $G/\Gamma$ be a nilmanifold of degree $k$, step $s$, complexity at most $M \geq 2$, and dimension $m\in\N$, with associated filtration $G_\bullet$, and let $g\in\poly(\Z, G)$.  Let $\delta\in(0, 1/10)$, and let $F:G/\Gamma \to \mathbb{C}$ be a $1$-bounded $M$-Lipschitz function with a vertical character $\xi$ satisfying $0<|\xi| \leq M/\delta$.  
Let $\tilde{V} \leq  W^{C_1}$ be a nonnegative integer power of $W$, let $\tilde{r}\in[\tilde{V}]$, and let 
$N\geq \tilde{V}_d \tilde{V}^{C_2} (M/\delta)^{2C_2 dm^{C_2}}$ be an integer.  Lastly, let $A \subseteq [N/\tilde{V}]$ and $B \subseteq \Z$ be nonempty arithmetic progressions of common difference at most $M/\delta$ satisfying 
\begin{align}\label{E: bound on B} 
   |B| \geq C_2\cdot(N/(\tilde{V}\tilde{V}_d))^{1/d}(M/\delta)^{-C_2 m^{C_2}}.
\end{align}
 
If
\begin{equation}\label{E:Wtricknilsequencehypothesis}
    \sum_{y\in A}\left|\E_{z \in B} e(\tilde{a}(y)\cdot\widetilde{P}_{[\tilde{r},\tilde{V}]}(z))F(g(y + \widetilde{P}_{[\tilde{r},\tilde{V}]}(z))\Gamma) \vphantom{\E_{z \in [(N/W_d)^{1/d}\tilde{V}^{-1}]}} \right| \geq  \delta |A|,
\end{equation}
then 
the following holds with $T:=|B|^d \tilde{V}_d$.
\begin{itemize}
    \item
 When $s>1$: There exists a factorization $g(n) = \veps(n)g'(n)\gamma(n)$, where $\veps$ is $((M/\delta)^{C_3m^{C_3}}, T)$-smooth, $g'$ takes values in an $(M/\delta)^{C_3m^{C_3}}$-rational normal subgroup $G'$ of $G$ such that $G'/\mathrm{ker}(\xi)$ has step at most $s - 1$, and $\gamma$ is $(M/\delta)^{C_3m^{C_3}}$-rational.
    \item When $s=1$ and $k>1$: There is a polynomial $p(n)=\sum_{j=0}^k \alpha_j n^j$ such that $F(g(n)\Gamma)=e(p(n))$, and there is a positive integer $q\leq C_3\delta^{-C_3}$ such that $\norm{q\alpha_j}_{\mathbb{R}/\mathbb{Z}} \leq C_3\delta^{-C_3}/T^j$ for all $j \geq 2$.
    \item When $s=1$ and $k=1$: There are a phase $\alpha \in \mathbb{R}/\mathbb{Z}$ and a positive integer $q \leq C_3\delta^{-C_3}$ such that $\norm{q\tilde{a}(y)-\alpha}_{\mathbb{R}/\mathbb{Z}} \leq C_3\delta^{-C_3}/T$ for at least {$C_3\delta^{C_3}|A|$} values of $y \in A$.
\end{itemize}
\end{proposition}
  
We remark that in the third case, the pigeonhole principle trivially provides some $\alpha$ such that $\norm{\widetilde a(y)-\alpha}_{\mathbb{R}/\mathbb{Z}} \leq C_3\delta^{-C_3}/T$ for at least $2C_3\delta^{-C_3}|A|/T$ values of $y$; the proposition improves the quantity $2C_3\delta^{-C_3}|A|/T$ by a factor of around $\delta^{O(1)}T$, which is quite substantial since $|B|$ is large in applications.
\begin{proof}
We fix $C_1>0$ and let all constants depend on $C_1, P, k$ (and hence also on $s$ since $s\leq k$). We address the three bullet-pointed cases one-by-one.

\smallskip
{\textbf{Case 1}: $s > 1$.}
\smallskip

Suppose first that $s>1$. By the pigeonhole principle, there is some $y_0 \in A$ such that
    \begin{align}\label{E: pigeonholed nilsequence average}
        \abs{\E_{z \in B} e(\tilde{a}(y_0)\widetilde{P}_{[\tilde{r}, \tilde{V}]}(z))F(g(y_0 + \widetilde{P}_{[\tilde{r}, \tilde{V}]}(z))\Gamma)} \ge \delta.
    \end{align}
Since $\xi$ is nonzero, there is some $g_s \in G_{(s)}$ such that $\xi(g_s)=a(y_0)$.  Now define the polynomial sequence $\tilde g: \mathbb{Z} \to G$ (with respect to the same filtration $G_\bullet$) via
$$\tilde g(n):=g_s^n g(y_0+n).$$
Since $\mathrm{poly}(\mathbb{Z}, G)$ is closed under pointwise multiplication, we have $\tilde{g} \in \mathrm{poly}(\mathbb{Z}, G)$. Notice that
$$\pi_{\mathrm{horiz}}(\tilde g(n))=\pi_{\mathrm{horiz}}(g(y_0+n))$$
since $g_s^n \in G_{(s)}$ has image zero in the horizontal torus. The identity $$F(\tilde g(n)\Gamma)=e(\tilde{a}(y_0)n)F(g(y_0+n)\Gamma)$$ then lets us rewrite \eqref{E: pigeonholed nilsequence average} as
$$\abs{\E_{z \in B} F(\tilde g(\widetilde P_{[r,V]}(z))\Gamma)} \geq \delta.$$
Now, consider the filtration $(\mathbb{Z}_i)_{i=0}^\infty$ on $\mathbb{Z}$ given by
\[\mathbb{Z}_i = \begin{cases}
\mathbb{Z} & \text{if $i \le d$}, \\
0 & \text{if $i > d$};
\end{cases}\]
with respect to this filtration, we have $\widetilde{P}_{[\tilde{r}, \tilde{V}]} \in \poly(\mathbb{Z},\mathbb{Z})$.  By \cite[Theorem 1.6.9]{T12}, the composition $\tilde{g} \circ \widetilde{P}_{[\tilde{r}, \tilde{V}]}(z)$ is a polynomial sequence on $G$ with a slightly modified filtration (now of degree $dk$).

Since
\[
|B| \geq C_2\cdot(N/(\tilde{V}\tilde{V}_d))^{1/d}(M/\delta)^{-C_2m^{C_2}} \geq C(M/\delta)^{C_2 m^{C_2}}
\]
and $C_2>0$ can be chosen to be sufficiently large, Theorem~\ref{multiparameter} gives us horizontal characters $\eta_1, \ldots, \eta_r$ of size $O((M/\delta)^{O(m)^{O(1)}})$ (for some $0<r\leq\mathrm{dim}(G/[G,G])$) such that each $\eta_i$ satisfies
\begin{equation}\label{eq:horizontalcharacterapplication}
    \|\eta_i \circ \tilde{g}\circ \widetilde{P}_{[\tilde{r}, \tilde{V}]}\|_{C^\infty[|B|]} \ll (M/\delta)^{O(m)^{O(1)}},    
\end{equation}
and such that $\xi$ annihilates all $(s-1)$-fold commutators of elements of $G':= \bigcap_{i=1}^r\ker(\eta_i)$.
Since
\begin{align}\label{E: B and V}
    |B| \geq C_2\cdot(N/(\tilde{V}\tilde{V}_d))^{1/d}(M/\delta)^{-C_2 m^{C_2}} \geq C_2\tilde{V}^{(C_2-1)/d}    
\end{align}
and $C_2>0$ is sufficiently large, Lemma~\ref{lem:composing-polynomials-2} tells us that
$$\|\eta_i \circ \tilde{g}\|_{C^\infty[T]} \ll (M/\delta)^{O(m)^{O(1)}}$$
for each $\eta_i$ (recall that $T=|B|^d\tilde{V}_d$).  Since
\begin{align}\label{E: y_0}
    |y_0|\leq N/\tilde{V}\ll (M/\delta)^{O(m)^{O(1)}} T   
\end{align}
by
\eqref{E: bound on B},  Lemma~\ref{translationextrapolation} further implies that each
$$\|\eta_i \circ g\|_{C^\infty[T]} \ll (M/\delta)^{O(m)^{O(1)}}$$
(where the $O(1)$ in this inequality may differ from the $O(1)$ in the preceding inequality).  The conclusion of the proposition now follows from \cite[Lemma A.1]{Leng23b}.

\smallskip
{\textbf{Case 2: $s=1, k>1$.}}
\smallskip

Now suppose that $s=1$.  Since $s=1$, there is a polynomial $p(n)=\sum_{j=0}^k \alpha_j n^j$ such that $F(g(n)\Gamma)=e(p(n))$.  Substituting this into \eqref{E:Wtricknilsequencehypothesis} and pigeonholing in $y$ gives that
    $$\abs{\E_{z \in [B]} e\left(\tilde{a}(y_0)\widetilde{P}_{[\tilde{r}, \tilde{V}]}(z) + \sum_{j = 0}^k \alpha_j (y_0 + \widetilde{P}_{[\tilde{r}, \tilde{V}]}(z))^{j}\right)} \ge \delta$$
for some $y_0 \in A$.  Weyl's inequality (see, e.g., \cite[Lemma 4.4]{GT12}) tells us that there is a positive integer $q\ll \delta^{-O(1)}$
such that
\begin{align*}
    \norm{q\cdot R(y_0 + \widetilde{P}_{[\tilde{r}, \tilde{V}]}(z))}_{C^\infty[|B|]}\ll \delta^{-O(1)}.
\end{align*}
for $R(n):= p(n) + \tilde{a}(y_0)(n-y_0)$.
Applying Lemma~\ref{lem:composing-polynomials-2} with \eqref{E: y_0} and then Lemma ~\ref{translationextrapolation} with \eqref{E: B and V}, we find that 
$\norm{q\alpha_j}_{\mathbb{R}/\mathbb{Z}} \ll \delta^{-O(1)}/T^j$ for all $j \geq 2$.
(We do not make use of the conclusion that $\norm{q(\tilde{a}(y_0)+\alpha_1)}_{\mathbb{R}/\mathbb{Z}} \ll \delta^{-O(1)}/T$.)  This concludes the proof of the proposition in the second case.

\smallskip
{\textbf{Case 3: $s=k=1$.}}
\smallskip

Finally, suppose that $s=1$ and moreover $k=1$.  Then there are $\alpha_0,\alpha_1 \in \mathbb{R}/\mathbb{Z}$ such that $F(g(n)\Gamma)=e(\alpha_1 n+\alpha_0)$.  Now \eqref{E:Wtricknilsequencehypothesis} reads
$$\sum_{y \in A} \abs{\E_{z \in [B]} e(\tilde{a}(y)\widetilde{P}_{[\tilde{r}, \tilde{V}]}(z) + \alpha_1 (y + \widetilde{P}_{[\tilde{r}, \tilde{V}]}(z)))} \ge \delta |A|.$$
The popularity principle provides at least $\delta|A|/2$ values of $y\in |A|$ such that
\begin{align*}
    \abs{\E_{z \in [B]} e(R_y\circ\widetilde{P}_{[\tilde{r}, \tilde{V}]}(z))}\geq \delta/2,
\end{align*}
where $R_y(n) := (\tilde{a}(y)+\alpha_1)n$.  Hence, by Weyl's inequality, we have $$\norm{q_y R_y\circ \widetilde{P}_{[\tilde{r}, \tilde{V}]}}_{C^\infty(|B|)}\ll \delta^{-O(1)}$$ for some positive integer $q_y\ll \delta^{-O(1)}$. Lemma \ref{lem:composing-polynomials-2} then gives $\norm{q_y R_y}_{C^\infty(T)}\ll \delta^{-O(1)}$, which is tantamount to
$$\norm{q_y(\tilde{a}(y)+\alpha_1)}_{\mathbb{R}/\mathbb{Z}} \ll \delta^{-O(1)}/T.$$
Pigeonholing in $y$, we can find $q$ such that this bound holds for $\gg \delta^{O(1)}|A|$ values of $y\in A$. 
\end{proof}

\subsection{Completing the proof of Proposition \ref{prop:popular-a-value}}
We are now in a position to complete the proof of Proposition~\ref{prop:popular-a-value}. We let all constants depend on $P$.
By Proposition \ref{P: local correlation with nilsequences}, there exist a positive integer $k=O(1)$, 
    a positive integer power $V' = V^{O(1)}$, and a function $\varphi:\Z\to\C$ such that for each $j\in[V']$, the restriction $\varphi_{(j,V')}$ is a $1$-bounded, degree-$k$ nilsequence of complexity $\exp(\log(1/\delta)^{O(1)})$ on a $\log(1/\delta)^{O(1)}$-dimensional nilmanifold. Moreover, we have 
\begin{multline*}
    \E_{j, r'\in[V']}\sum_{y\in[N/(VV')]}\left|\E_{z \in [(N/W_d)^{1/d}(VV')^{-1}]} e(-a_{(j, V')}(y)V'\cdot\widetilde{P}_{[Vr'+r,VV']}(z))\right.\\
    \left.\varphi_{({j} + \widetilde{P}_{[r,V]}(r'), V')}(y + \widetilde{P}_{[Vr'+r,VV']}(z)) \vphantom{\E_{z \in [(N/W_d)^{1/d}(VV')^{-1}]}} \right| \geq \exp(-\log(1/\delta)^{O(1)}) N/(VV').
\end{multline*}
By the popularity principle, there exists $r'\in[V']$ and a set $\mathcal{V}\subseteq [V']$ of size $|\mathcal{V}|\geq \exp(-\log(1/\delta)^{O(1)})V'$ such that for each $j\in\mathcal{V}$, we have the bound
\begin{align}\label{E: lower bound for j in V}
    \sum_{y\in [N/\tilde{V}]}\left|\E_{z \in [(N/W_d)^{1/d}\tilde{V}^{-1}]} e(\tilde{a}_{(j,V')}(y)\cdot\widetilde{P}_{[\tilde{r},\tilde{V}]}(z))\tilde{\varphi}_{j}(y + \widetilde{P}_{[\tilde{r},\tilde{V}]}(z)) \vphantom{\E_{z \in [(N/W_d)^{1/d}\tilde{V}^{-1}]}} \right| \geq \eta N/\tilde{V},
\end{align}
 where $\eta=\exp(-\log(1/\delta)^{O(1)})$ and
\begin{align*}
    \tilde{V}=VV',\quad \tilde{r} = Vr'+r,\quad \tilde{a}(y) = -a(y)V',\quad\textrm{and}\quad   \Tilde{\varphi}_j(y)= \varphi_{({j} + \widetilde{P}_{[r,V]}(r'), V')}(y).
\end{align*}

Fix $j\in\mathcal{V}$. We will iteratively apply Proposition  \ref{prop:Wtricknilsequenceprop}.  This will let us first reduce to the case where \eqref{E: lower bound for j in V} holds with $\tilde{\varphi}_{j}$ replaced by a degree-$1$ nilsequence on a $1$-dimensional nilmanifold, and then deduce that the phase function $\tilde{a}_{(j, V')}$ is nearly constant for a positive proportion of $y$'s. Our proof proceeds by iteratively reducing the step of the nilmanifold underlying $\tilde{\varphi}_{j}$. At each stage, we will have parameters $\eta, A, B, g, G/\Gamma$, and $F$. Since the step is initially at most $\log (1/\delta)^{O(1)}$ and we decrease the step at each stage, there are at most $\log (1/\delta)^{O(1)}$ stages.  Keeping track of how our bounds degrade at each stage, we can guarantee that throughout the iterative process we always have the bounds
\begin{gather*}
    \eta \geq \exp(-\log(1/\delta)^{O(1)}),\quad N \gg \Tilde{V}_d \tilde{V}^{\Omega(1)}\eta^{-O(m)^{O(1)}}, \quad |A| \gg \eta^{O(m)^{O(1)}}\frac{N}{\tilde{V}},\\ |B| \gg \eta^{O(m)^{O(1)}}(N/(\tilde{V}\tilde{V}_d))^{1/d},\quad T =|B|^d \Tilde{V}_d\gg \eta^{O(m)^{O(1)}} \frac{N}{\tilde{V}}, \\ M\ll \eta^{-O(m)^{O(1)}},\quad \text{ and }\quad
    m \ll \log(1/\delta)^{O(1)}.
\end{gather*}
We remind the reader that $m,M$ stand for dimension and complexity of the underlying nilmanifold, and we observe that the frequently appearing quantity $\eta^{O(m)^{O(1)}}$ is of the shape $\exp(-\log(1/\delta)^{O(1)})$.
Our lower bound $N\gg V^{\Omega(1)}\exp(\log(1/\delta)^{\Omega(1)})$ guarantees that the iteration will not halt before the claimed result is proven.

We begin with $A = [N/\tilde{V}]$ and $B = [(N/(\tilde{V}\tilde{V}_d))^{1/d}]=[(N/W_d)^{1/d}\tilde{V}^{-1}]$. We also start with $\eta \asymp \exp(-\log(1/\delta)^{O(1)})$, for which the lower bound in \eqref{E: lower bound for j in V} holds. 

\textbf{Step 1: Fourier-expanding $F$.}
\smallskip

In order to apply Proposition~\ref{prop:Wtricknilsequenceprop}, we will first Fourier-expand the function $F$. Suppose that $F:G/\Gamma \to \mathbb{C}$ is an $M$-Lipschitz function on a nilmanifold $G/\Gamma$ of complexity $M$, and that
\begin{equation}\label{e: firsthypothesis}
\sum_{y \in A} \abs{\E_{z \in B}e(\tilde{a}_{(j, \tilde{V})}(y)\widetilde{P}_{[\tilde{r}, \tilde{V}]}(z))F(g(y + \widetilde{P}_{[\tilde{r}, \tilde{V}]}(z))\Gamma)} \ge \eta |A|.   
\end{equation}
By \cite[Lemma A.6]{Leng23b}, we can Fourier-expand $F$ along the $G'_{(s)}$-torus as
$$\tilde{F} = \sum_{|\xi| \le (M/\eta)^{O(m)^{O(1)}}} F_\xi + O((\eta/M)^{O(m)^{O(1)}}),$$
where each $F_\xi$ is an $(M/\eta)^{O(m)^{O(1)}}$-Lipschitz vertical character of frequency $\xi$. By replacing $\eta$ with $(\eta/M)^{O(m)^{O(1)}}$, we may assume that $F = F_\xi$ and that $F$ is $1$-Lipschitz.
Thus, we have
\begin{equation}
\sum_{y \in A} \abs{\E_{z \in B}e(\tilde{a}_{(j, \tilde{V})}(y)\widetilde{P}_{[\tilde{r}, \tilde{V}]}(z))F(g(y + \widetilde{P}_{[\tilde{r}, \tilde{V}]}(z))\Gamma)}\gg (\eta/M)^{O(m)^{O(1)}} |A|
\end{equation}
(The purpose of these manipulations is to replace the function $F$ in \eqref{e: firsthypothesis} by a vertical character.) Finally, we relabel $(\eta/M)^{O(m)^{O(1)}}$ as simply $\eta$ so our hypothesis becomes
\begin{equation}\label{e: secondhypothesis}
\sum_{y \in A} \abs{\E_{z \in B}e(\tilde{a}_{(j, \tilde{V})}(y)\widetilde{P}_{[\tilde{r}, \tilde{V}]}(z))F(g(y + \widetilde{P}_{[\tilde{r}, \tilde{V}]}(z))\Gamma)}\gg \eta |A|.
\end{equation}

\smallskip
\textbf{Step 2: Applying Proposition~\ref{prop:Wtricknilsequenceprop}.}
\smallskip

Proposition~\ref{prop:Wtricknilsequenceprop} applied to \eqref{e: secondhypothesis} supplies a factorization $g = \varepsilon g' \gamma$, where $\gamma$ is an $(M/\eta)^{O(m)^{O(1)}}$-rational sequence, $\veps$ is an $((M/\eta)^{O(m)^{O(1)}}, T)$-smooth sequence, and $g'$ is a sequence lying in an $(M/\eta)^{O(m)^{O(1)}}$-rational subnilmanifold $G'/\Gamma'$ of step $\le s - 1$ for some normal subgroup $G'\subseteq G$ with lattice $\Gamma' := G'\cap \Gamma$. Since $\varepsilon$ is $((M/\eta)^{O(m)^{O(1)}}, T)$-smooth, it is also $(\sigma^{-1}, |A|)$-smooth for some $\sigma = (\eta/M)^{O(m)^{O(1)}}$.

\smallskip
\textbf{Step 3: Pigeonholing in a subprogression.}
\smallskip

Let $R \ll (M/\eta)^{O(m)^{O(1)}}$ be the period\footnote{We use here \cite[Lemma B.14]{Leng23b}, which says that if $\gamma$ is $Q$-rational, then it is $O(Q^{O(1)})$-periodic.} of $\gamma$. We now pigeonhole in subprogressions $A'$ of $A$ and $B'$ of $B$, of common difference $R$ and lengths
$$(\eta/M)^{O(m)^{O(1)}}|A| \ll |A'| \leq \sigma |A|\eta^{2} \quad \text{and} \quad (\eta/M)^{O(m)^{O(1)}}|B| \ll |B'| \leq \sigma |B|\eta^{2},$$
to find some such $A',B'$ such that
$$\sum_{y \in A'} \abs{\E_{z \in B'}e(\tilde{a}_{(j, \tilde{V})}(y)\widetilde{P}_{[\tilde{r}, \tilde{V}]}(z))F(g(y + \widetilde{P}_{[\tilde{r}, \tilde{V}]}(z))\Gamma)} \geq (\eta/2) |A'|.$$
Recall that we have the factorization 
$$g(y + \widetilde{P}_{[\tilde{r}, \tilde{V}]}(z)) = \varepsilon(y + \widetilde{P}_{[\tilde{r}, \tilde{V}]}(z))g'(y + \widetilde{P}_{[\tilde{r}, \tilde{V}]}(z))\gamma(y + \widetilde{P}_{[\tilde{r}, \tilde{V}]}(z)).$$
The term  $\gamma(y + \widetilde{P}_{[\tilde{r}, \tilde{V}]}(z))\Gamma$ is constant for $y \in A'$ and $z \in B'$ because $A',B'$ have common difference $R$.  Thus, there is some $\gamma_0 \in G$ such that for such $y,z$, our factorization can be written as
$$\varepsilon(y + \widetilde{P}_{[\tilde{r}, \tilde{V}]}(z))g'(y + \widetilde{P}_{[\tilde{r} \tilde{V}]}(z))\gamma_0,$$
up to right-multiplication by $\Gamma$. By the smoothness properties of $\varepsilon$, we have
\begin{gather*}
d_{G/\Gamma}(\varepsilon(y + \widetilde{P}_{[\tilde{r}, \tilde{V}]}(z)), \varepsilon_0) \leq \left(\sigma^{-1}/|A|\right) \cdot |A'| \leq \eta^2 \leq \eta/4,\\
\textrm{and}\quad d_{G/\Gamma}(\varepsilon_0, \mathrm{id}_G) \leq \sigma^{-1}\ll (M/\eta)^{O(m)^{O(1)}}
\end{gather*}
for some $\varepsilon_0 \in G$. Then, by the triangle inequality (recall that $F$ has Lipschitz constant at most $1$), we have
$$\sum_{y \in A'} \abs{\E_{z \in B'}e(\tilde{a}_{(j, \tilde{V})}(y)\widetilde{P}_{[\tilde{r}, \tilde{V}]}(z))F(\varepsilon_0 g'(y + \widetilde{P}_{[\tilde{r}, \tilde{V}]}(z))\gamma_0\Gamma)} \geq (\eta/4) |A'|.$$
By \cite[Lemma B.6]{Leng23b}, we can write $\gamma_0 = \{\gamma_0\}[\gamma_0]$ where $[\gamma_0] \in \Gamma$ and $|\psi(\{\gamma_0\})| \le \frac{1}{2}$. Expanding
$$\varepsilon_0 g'(y + \widetilde{P}_{[\tilde{r}, \tilde{V}]}(z))\gamma_0 = \varepsilon_0 \{\gamma_0\} (\{\gamma_0\}^{-1}g'(y + \widetilde{P}_{[\tilde{r}, \tilde{V}]}(z))\{\gamma_0\})[\gamma_0]$$ and setting
$\tilde{F} := F|_{G'}(\varepsilon_0 \{\gamma_0\} \cdot)$ (this function is $O((M/\eta)^{O(m)^{O(1)}})$-Lipschitz by \cite[Lemma B.4]{Leng23b}), we deduce from the normality of $G'$ in $G$ that $$\{\gamma_0\}^{-1}g'(y + \widetilde{P}_{[\tilde{r}, \tilde{V}]}(z))\{\gamma_0\} \in \mathrm{poly}(\mathbb{Z}, G')$$ and that
$$\sum_{y \in A'} \abs{\E_{z \in B'}e(\tilde{a}_{(j, \tilde{V})}(y)\widetilde{P}_{[\tilde{r}, \tilde{V}]}(z))\tilde{F}(\{\gamma_0\}^{-1}g'(y + \widetilde{P}_{[\tilde{r}, \tilde{V}]}(z))\{\gamma_0\}\Gamma')}\gg (\eta/M)^{O(m)^{O(1)}} |A'|.$$

\smallskip
\textbf{Step 4: Completing the iteration.}
\smallskip

We continue the iteration by redefining 
\begin{gather*}
A \leadsto A',\quad B \leadsto B',\quad g(n) \leadsto \{\gamma_0\}^{-1}g'(n)\{\gamma_0\},\\ G/\Gamma \leadsto G'/\Gamma',\quad \eta \leadsto (\eta/M)^{O(m)^{O(1)}},\quad \text{ and }\quad F \leadsto \tilde{F}.    
\end{gather*}

\smallskip
\textbf{Step 5: The special case $s = 1$.}
\smallskip

When $s = 1$ and $k \ge 2$, we have $F(g(n)\Gamma) = e(p(n))$ for some $p(n) = \sum_{j = 0}^k \alpha_j n^j \in \mathbb{R}[n]$. Applying Proposition~\ref{prop:Wtricknilsequenceprop}, we see that there exists some $q \ll \eta^{-O(m)^{O(1)}}$ such that
$$\|q\alpha_j\|_{\mathbb{R}/\mathbb{Z}} \ll \frac{(M/\eta)^{O(m)^{O(1)}}}{T^j}\quad \textrm{for all}\quad 2\leq j \leq k.$$
Notice that when $n$ is restricted to an arithmetic progression of difference $q$, the function $\sum_{j=2}^k \alpha_j n^j$ is slowly-varying and $p(n)$ is very close to a constant plus $\alpha_1 n$ .  Combining this with the pigeonholing argument from Step 2, which lets us replace $A,B$ by arithmetic progressions of difference $q$, we reduce to the case where $k=1$.

Finally, when $s = 1$ and $k = 1$, the last case of Proposition \ref{prop:Wtricknilsequenceprop} gives for each $j\in\CV$ a positive integer $q_j\ll \eta^{-O(m)^{O(1)}}\leq \exp(\log(1/\delta)^{O(1)})$ and a phase $\tilde{\alpha}_j\in\R/\Z$ such that
\begin{align*}
    \|q_j(\tilde{a}_{(j, \tilde{V})}(y) + \tilde{\alpha}_j)\|_{\mathbb{R}/\mathbb{Z}} &= \|q_j({a}_{(j, \tilde{V})}(y)V' - \tilde{\alpha}_j)\|_{\mathbb{R}/\mathbb{Z}}\ll \frac{(M/\eta)^{O(m)^{O(1)}}}{T}\\
    &\ll \frac{(M/(\delta\eta))^{O(m)^{O(1)}}\tilde{V}}{N}\leq \exp(\log(1/\delta)^{O(1)})\frac{\Tilde{V}}{N} 
\end{align*}
for at least $\exp(-(\log(1/\delta))^{O(1)})N/\tilde{V}$ values of $y\in [N/\Tilde{V}]$. Pigeonholing in the residue classes modulo $q_j$, we can find some $\alpha_j \in \{\frac{\tilde{\alpha_j}+\kappa_j}{q_j}: \kappa_j \in [q_j]\}$ such that
\begin{align*}
    \|{a}_{(j, \tilde{V})}(y)V' - {\alpha}_j\|_{\mathbb{R}/\mathbb{Z}}\leq \exp(\log(1/\delta)^{O(1)})\frac{\Tilde{V}}{N} 
\end{align*}
holds for at least $\exp(-(\log(1/\delta))^{O(1)})N/\tilde{V}$ values of $y\in [N/\Tilde{V}]$.
For $j\notin\CV$, we define $\alpha_j$ arbitrarily. This completes the proof of Proposition \ref{prop:popular-a-value}.

\section{Finishing the argument}\label{S: finishing}

Recall that
$$\Lambda^{\mathrm{Model}}(f_0, f_1, f_2) = \sum_{x, y} \E_{z \in [N]} f_0(x, y)f_1(x + z, y) f_2(x, y + z)\nu(z)$$
where $\nu(z) = d^{-1} 1_{[N]}(z) (N/z)^{(d-1)/d}$.  Our goal in this section is to compare $\Lambda^W$ with $\Lambda^{\model}$. We already know from Theorem \ref{thm:degree lowering-lambda} that $\Lambda^W$ is controlled by suitable degree-$2$ box norms of $f_0, f_1, f_2$, and we now show that the same holds for $\Lambda^\model.$ 
\begin{proposition}\label{P: box norm control of model}
    Let $\delta\in(0, 1/10)$, let $N\in\N$, and let $f_0, f_1, f_2:\Z^2\to\C$ be $1$-bounded functions supported on $[N]^2$. If 
    \begin{align}\label{E: lower bound for model}
        \Lambda^{\model}(f_0, f_1, f_2)\geq \delta N^3,
    \end{align}
    then $\norm{f_0}^4_{\be_1\cdot[\pm N],\; \be_2\cdot[\pm N]}, \norm{f_1}^4_{\be_1\cdot[\pm N],\; (\be_2-\be_1)\cdot[\pm N]}, \norm{f_2}^4_{\be_2\cdot[\pm N],\; (\be_2-\be_1)\cdot[\pm N]}\gg_d \delta^{O_d(1)} N^2$.
\end{proposition}

\begin{proof}
We will replace $\nu$ by a smoother weight function that is easier to work with.  Let $\veps>0$ be a small parameter to be chosen later, and let $\varphi: \Z \to [0,1]$ be a $4\veps^{-1}$-Lipschitz function that equals $1$ on $[2\veps N, (1-\veps)N-1]$ and vanishes outside of $[\veps N, N-1]$.  Define the cut-off weight function
$$\nu_{\veps}(z):= d^{-1} (N/(z+1))^{(d-1)/d}\varphi(z).$$
Then $\nu_{\eps}$ equals $\nu$ away from the boundary of $[N]$, and we have the $\ell^1$-bound
$$\sum_{z \in \Z} \abs{\nu_{\veps}(z)-\nu(z)} \ll \veps^{1/d}N.$$
Our construction of $\nu_{\veps}$ ensures that it is $O\brac{\veps^{-3}N^{-1}}$-Lipschitz.  Define the function $g: \R/\Z \to [0,1]$ by setting $g(z):=\veps^{1-1/d}\nu_{\eps}(zN)$ for $0 \leq z<1$; it follows that $g$ is $O\brac{\veps^{-3}}$-Lipschitz. By \cite[Lemma 4.1]{OP23}, there is a trigonometric polynomial $h: \R/\Z \to [0,1]$ of the form
$$h(x)=\sum_{m \ll \veps^{-10}} a(m) e(mx),$$
with each $|a_m| \leq 1$, such that
$$\norm{g-h}_\infty \leq \veps.$$
Unwinding the various definitions, we see that
$$\abs{\nu_{\veps}(z)-\veps^{1/d-1}\sum_{m \ll \veps^{-10}} a(m) e(mz/N)} \leq \veps^{1/d}$$
for all $z \in [N]$.  This, combined with our earlier $\ell^1$-bound, implies that $\Lambda^{\model}(f_0, f_1, f_2)$ differs by at most $O\brac{\veps^{1/d}N^3}$ from
\begin{align*}
\sum_{x,y}\sum_{z \in [N]} f_0(x,y)f_1(x+z,y)f_2(x,y+z) \cdot \veps^{1/d-1}\sum_{m \ll \veps^{-10}} a(m) e(mz/N).
\end{align*}
In particular, if we set $\veps$ to be a sufficiently small constant times $\delta^d$, then the latter quantity has absolute value at least $\delta N^3/2$.  Applying the triangle inequality and distributing the phases $e(mz/N)$, we obtain
$$\sum_{m \ll \delta^{-10d}} \abs{\sum_{x,y}\sum_{z \in [N]} f_0(x,y)f_1^{(m)}(x+z,y)f_2^{(m)}(x,y+z)} \gg \delta^{2d} N^3,$$
where
$$f_1^{(m)}(x,y):=f_1(x,y)e(mx/N) \quad \text{and} \quad f_2^{(m)}(x,y):=f_2(x,y)e(-mx/N)$$
are still $1$-bounded functions supported on $[N]$.  Pigeonholing, we find some $m$ such that
$$\abs{\sum_{x,y}\sum_{z \in [N]} f_0(x,y)f_1^{(m)}(x+z,y)f_2^{(m)}(x,y+z)} \gg \delta^{12d} N^3,$$
and then the desired norm-control on $f_0$ follows from the usual Cauchy--Schwarz maneuvers.  The norm-control on $f_1,f_2$ follows analogously.
\end{proof}

We are finally ready to show that $\Lambda^W$ and $\Lambda^\model$ are always close together.

\begin{proposition}\label{P: Lambda*}
    There exists $C = C(P)>0$ such that the following holds. Let $f_0, f_1, f_2:\Z^2\to\C$ be $1$-bounded functions supported on $[N]^2$. If $N, W$ satisfy $N\geq C W^{C}$ and $W\geq \exp\exp(\log(1/\delta)^{C})$, then
    \begin{align*}
        |\Lambda^W(f_0,f_1,f_2)-\Lambda^{\mathrm{Model}}(f_0,f_1,f_2)| <\delta N^2.
    \end{align*}
\end{proposition}

\begin{proof}
We let all constants depend on $P$.
Define the ``difference'' counting operator
\begin{align*}
    \Lambda^*(f_0, f_1, f_2) := \Lambda^W(f_0,f_1,f_2)-\Lambda^{\mathrm{Model}}(f_0,f_1,f_2)
\end{align*}
and the dual function
\begin{align*}
    \CD_0^*(f_1, f_2)(x,y) &:= \E_{z\in[(N/W_d)^{1/d}]}f_1(x + \widetilde P(z), y) f_2(x, y + \widetilde P(z))\\
    &\qquad\qquad- \E_{z\in[N]}f_1(x + z, y) f_2(x, y + z)\nu(z).
\end{align*}
Assume for the sake of contradiction that
\begin{align*}
    |\Lambda^*(f_0, f_1, f_2)|\geq \delta N^2;
\end{align*}
then stashing gives
\begin{align*}
    |\Lambda^*(\CD_0^*(f_1,f_2), f_1, f_2)|\geq \delta^2 N^2,
\end{align*}
and we conclude by the triangle inequality that either
\begin{align*}
    |\Lambda^W(\CD_0^*(f_1,f_2), f_1, f_2)|\geq \delta^2 N^2/2 \quad\textrm{or}\quad |\Lambda^\model(\CD_0^*(f_1,f_2), f_1, f_2)|\geq \delta^2 N^2/2.
\end{align*}
According to which of these inequalities holds, we apply either
Theorem \ref{thm:degree lowering-lambda} or Proposition \ref{P: box norm control of model} and Lemma \ref{L: properties of box norms}\eqref{i: passing to APs} to find a positive integer power $V = W^{O(1)}$ for which 
\begin{align*}
    \norm{\CD_{0}^*(f_1,f_2)}_{\be_1 V \cdot[\pm N/V],\; \be_2 V \cdot[\pm N/V]}^4 \geq \exp(-\log(1/\delta)^{O(1)})N^2.
\end{align*}
As usual, we must pass to arithmetic progressions of difference $V$ before using the $U^1 \times U^2$-inverse theorem. Write $g_{0, j_1, j_2}:= (\CD_0^*(f_1, f_2))_{(j_1, j_2, V)}$, so that
\begin{align*} 
    \E_{j_1, j_2\in[V]}\norm{g_{0, j_1, j_2}}_{\be_1 \cdot[\pm N/V],\; \be_2 \cdot[\pm N/V]}^4 \geq \exp(-\log(1/\delta)^{O(1)})(N/V)^2.
\end{align*}

The inverse theorem for the $U^1 \times U^1$-box norm (Lemma \ref{L: U1xU1 inverse}) gives $1$-bounded functions $b_{1, j_1, j_2},b_{2, j_1, j_2}: \mathbb{Z} \to \mathbb{C}$, for each $j_1, j_2\in [V]$, such that
\begin{align*}
    \E_{j_1, j_2\in[V]}\sum_{x,y}g_{0, j_1, j_2}(x,y)b_{1, j_1, j_2}(x)b_{2, j_1, j_2}(y)\geq \exp(-\log(1/\delta)^{O(1)})(N/V)^2.
\end{align*}

Define the weights
\begin{align*}
    \tilde{\nu}(z):= \frac{N}{(N/W_d)^{1/d}}1_{\widetilde P([(N/W_d)^{1/d}])}(z)\quad\textrm{and}\quad \nu^*: = \Tilde{\nu}-\nu,
\end{align*}
so that
\begin{align*}
    \E_{z\in[(N/W_d)^{1/d}]}f_1(x + \widetilde P(z), y)f_2(x, y + \widetilde P(z)) = \E_{z\in[N]} f_1(x+z, y) f_2(x, y+z)\Tilde{\nu}(z)
\end{align*}
and
\begin{align*}
    \CD_0^*(f_1, f_2)(x,y) = \E_{z\in[N]} f_1(x+z, y) f_2(x, y+z)\nu^*(z).
\end{align*}
Then
\begin{multline*}
    \E_{j_1,j_2,r\in[V]}\sum_{x,y}\E_{z\in[N/V]}b_{1,j_1,j_2}(x)b_{2,j_1,j_2}(y)\\
    f_{1, j_1 + r, j_2}(x+z,y)f_{2,j_1,j_2+r}(x,y+z)\nu^*_r(z)\geq \exp(-\log(1/\delta)^{O(1)})(N/V)^2, 
\end{multline*}
where $f_{l,j_1,j_2} :=(f_l)_{(j_1,j_2, V)}$ and $\nu^*_r := \nu^*_{(r,V)}$.

We will now deduce the largeness of the norm $\norm{\nu^*}_{U^2(V\cdot[\pm N/V])}$, in contradiction with the Fourier-uniformity estimate from Proposition \ref{prop:nu-uniformity}.  After shifting $x\mapsto x-z$, we use the Cauchy--Schwarz inequality and Lemma~\ref{L: vdC} in order to double $z$, obtaining
\begin{multline*}
    \E_{j_1,j_2,r\in[V]}\sum_{x,y,h_1}\E_{z\in[N/V]} \mu_{N'/V}(h_1)\Delta_{-h_1}b_{1,j_1,j_2}(x-z)\\
    \Delta_{h_1(\be_2-\be_1)}f_{2,j_1+r,j_2}(x-z,y+z)\Delta_{h_1} \nu^*_r(z)\geq \exp(-\log(1/\delta)^{O(1)})(N/V)^2
\end{multline*}
for any positive integer $ N'\leq \exp(-\log(1/\delta)^{O(1)}) N$. Shifting $(x,y)\mapsto (x+z, y-z)$ and repeating this maneuver gives
\begin{align*}
    \E_{r\in[V]}\norm{\nu^*_r}_{U^2([\pm N'/V])}^4 = \E_{r\in[V]}\sum_{z,h_1,h_2}\mu_{N'/V}(h_1,h_2)\Delta_{h_1,h_2}\nu^*_r(z)\geq \exp(-\log(1/\delta)^{O(1)})N/V.
\end{align*}
Choosing $N'\asymp \exp(-\log(1/\delta)^{C}) N$ for a sufficiently large constant $C>0$, we can 
use Lemma \ref{L: properties of box norms}\eqref{i: enlarging} to replace $N'$ with $N$ and obtain the norm-control
\begin{align*}
    \norm{\nu^*}_{U^2(V\cdot[\pm N/V])}^4 = \sum_{r\in[V]}\norm{\nu^*_r}_{U^2([\pm N/V])}^4 \geq \exp(-\log(1/\delta)^{O(1)}) N.
\end{align*}
But Proposition~\ref{prop:nu-uniformity} and Lemma \ref{L: 1-dim U^2 inverse} tell us that
\begin{align*}
    \norm{\nu^*}_{U^2(V\cdot[\pm N/V])}^4 \ll N/w^{c}
\end{align*}
for some constant $c>0$ depending only on $P$. Hence $w \leq \exp(\log(1/\delta)^{O(1)})$ and so $W\leq \exp\exp(\log(1/\delta)^{O(1)})$,
which contradicts the assumption on $W$ as long as the $O(1)$ term is sufficiently large.
\end{proof}
We are finally ready to conclude the proof of Theorem \ref{maintheorem}.
\begin{proof}[Proof of Theorem~\ref{maintheorem}]
It suffices to prove the result with $N$ replaced by $WN$.
Let $A \subseteq [WN]^2$ be a set with density $\gamma$. By the pigeonhole principle, there exist $j_1, j_2 \in [\beta_1^2W]$ such that the set $$B := \{(x,y)\in[N]^2: \beta_1^2W(x,y) + (j_1, j_2) \in A\}$$ has size at least $\gamma N^2$. Since $ \nu(z) \geq d^{-1}$ pointwise for $z \in [N]$, the supersaturation result (Proposition~\ref{prop:supersat}) tells us that 
$$\Lambda^{\mathrm{Model}}(1_B,1_B,1_B) \geq d^{-1}\exp(- \exp (\gamma^{-c'}))N^2$$
for some absolute $c'>0$. Setting $\delta:=\exp(- \exp (\gamma^{-c'}))/(2d)$,
we obtain from Proposition \ref{P: Lambda*} that
$$\Lambda^W(1_B,1_B,1_B) \gg_P \exp(- \exp (\gamma^{-c'})) N^2$$
for any $W\geq \exp\exp(\log(1/\delta)^{\Omega_P(1)})$ satisfying
\begin{align*}
    N\gg_P W^{\Omega_P(1)}\geq \exp\exp(\log(1/\delta)^{\Omega_P(1)})\geq \exp\exp\exp(\Omega_P(\gamma^{-c'})).
\end{align*} 
In particular, the set $B$ contains nontrivial configurations $(x,y),\; (x+\widetilde P(z),y),\; (x,y+\widetilde P(z))$, and hence the set $A$ contains nontrivial configurations of $(x, y),\; (x + P(z), y),\; (x, y + P(z))$.  This completes the proof.
\end{proof}
We will now deduce Corollary~\ref{cor: PSS improvement}.
\begin{proof}[Proof of Corollary~\ref{cor: PSS improvement}]
It suffices to prove the result with $N$ replaced by $WN$.
Let $A \subseteq [WN]$ be a set of density $\gamma$. By the pigeonhole principle, there exists $j \in [\beta_1^2W]$ such that the set $$B := \{x\in[N]: \beta_1^2Wx + j \in A\}$$ has size at least $\gamma N$. Since $\nu(z) \geq d^{-1}$ pointwise for $z \in [N]$, a standard Varnavides-style supersaturation version of the result of Kelley and Meka \cite{KM23} tells us that 
$$\sum_{x} \mathbb{E}_{z \in [N]}1_B(x)1_B(x + z)1_B(x + 2z)  \geq d^{-1}\exp(-\log(1/\gamma)^{C'})N$$
for some absolute constant $C'>0$.

Consider the set
$$B':=\{(x,y) \in [2N]^2: y-x \in B\},$$
and note that
$$N\sum_{x} 1_B(x)1_B(x + z)1_B(x + 2z) \asymp \sum_{x,y} 1_{B'}(x,y)1_{B'}(x + z,y)1_{B'}(x,y+z)$$
for each $z \in [N]$.  Using this correspondence and applying Proposition \ref{P: Lambda*} with $\delta:=\exp(- \log(1/\gamma)^{C'})/(10d)$, we conclude that
$$\sum_{x} \mathbb{E}_{z \in [(N/W_d)^{1/d}]} 1_B(x)1_B(x + \widetilde P(z))1_B(x + 2\widetilde P(z)) \gg_P \exp(-\log(1/\gamma)^{O_P(1)}) N$$
as long as $W\geq \exp\exp(\log(1/\delta)^{\Omega_P(1)})$ satisfies
\begin{align*}
    N\gg_P W^{\Omega_P(1)}\geq \exp\exp(\log(1/\delta)^{\Omega_P(1)})\geq \exp\exp(\log(1/\gamma)^{\Omega_P(1)}).
\end{align*}
In particular, the set $B$ contains nontrivial configurations $x,\; x+\widetilde P(z),\; x+2\widetilde P(z)$, and hence the set $A$ contains nontrivial configurations $x,\; x+ P(z),\; x+2 P(z)$.  This completes the proof.
\end{proof}

\appendix

\section{Inverse theorems for box norms}\label{A: inverse theorems}
We require inverse theorems for several box norms and directional Gowers norms that arise in our arguments.  Although these results are more or less standard, we include statements and proofs here for ease of reference.

\subsection{$U^1$-inverse theorems}
Heuristically, a function has large degree-$1$ Gowers norm in the direction $\bv$ if it correlates with a function that is periodic in the $\bv$-direction. Since we are working with norms in which the differencing parameters $h$ are constrained to lie in intervals, making this heuristic precise is somewhat delicate. In particular, we need the ``length'' of the differencing parameter to be sufficiently large relative to the size of the support of the function in question, where ``sufficiently large'' depends on the quality of the lower bound on the Gowers norm.
This issue does not appear for higher-degree box norms since Lemma \ref{L: properties of box norms}\eqref{i: enlarging} allows us to manipulate the lengths of boxes at will. 

We start with the $U^1$-inverse theorem for functions on $\Z$.
\begin{lemma}[$1$-dimensional $U^1$-inverse theorem]\label{L: 1-dim U^1 inverse}
    Let $\delta\in(0, 1]$ and $N, N'\in\N$ with $N'\geq \delta^{-1/2} N\geq \delta^{-1}$.  If $f: \Z \to \mathbb{C}$ is a $1$-bounded function supported on $[N]$, then
    \begin{gather*}
        \norm{f}_{[\pm N']}^2\geq \delta N \quad \Longrightarrow\quad \abs{\sum_{x}f(x)}\gg\delta^{1/4} N.
    \end{gather*}
\end{lemma}
\begin{proof}
By definition, we have
\begin{align*}
    \norm{f}_{[\pm N']}^2 = \sum_{x,x'\in[N]}\mu_{N'}(x'-x)f(x)\overline{f(x')}.
\end{align*}
The idea is that since $N'$ is sufficiently large relative to $N$, we have $\mu_{N'}(x-x') \approx 1/(2N'-1)$ for all $x,x'\in[N]$.  More precisely, since $f$ is $1$-bounded and $N' \geq \delta^{-1/2}N \geq N$, we can make the crude comparison
\begin{align*}
\abs{(2N'-1)\norm{f}_{[N']}^2 -\sum_{x,x'}f(x)\overline{f(x')}} &\leq \sum_{x,x'\in[N]}\abs{(2N'-1)\mu_{N'}(x'-x)-1}\\
 &=\frac{1}{2N'-1} \sum_{x,x' \in [N]} |x-x'|\\
 & \leq \frac{N^2(N-1)}{2N'-1} \leq \delta^{1/2} N^2/2.
\end{align*}
Hence
\begin{align*}
\abs{\sum_x f(x)}^2=\sum_{x,x'} f(x) \overline{f(x')} &\geq (2N'-1) \delta N-\delta^{1/2} N^2/2\\
&\geq(3/2)\delta^{1/2}N - \delta N\gg \delta^{1/2}N^2
\end{align*}
since $N\geq\delta^{-1/2}$.
\end{proof}

\begin{lemma}[$2$-dimensional $U^1$-inverse theorem]\label{L: U^1 inverse}
    Let $\delta\in(0, 1]$ and $N, N'\in\N$ with $N'\geq (\delta/2)^{-1/2} N\geq 2\delta^{-1}$, and let $f: \Z^2 \to \mathbb{C}$ be a $1$-bounded function supported on $[N]^2$.  Then
    $$\norm{f}^2_{\be_1\cdot[\pm N']}\geq \delta N^2 \; \Longrightarrow \; \text{$\sum_{x,y}f(x,y)b(y)\gg \delta^{5/4} N^2$ for some $1$-bounded $b:\Z\to\C$};$$
     $$\norm{f}_{\be_2\cdot[\pm N']}^2\geq \delta N^2 \; \Longrightarrow \; \text{$\sum_{x,y}f(x,y)b(x)\gg \delta^{5/4} N^2$ for some $1$-bounded $b:\Z\to\C$};$$
and
 $$\norm{f}^2_{(\be_2-\be_1)\cdot[\pm N']}\geq \delta N^2 \; \Longrightarrow \; \text{$\sum_{x,y}f(x,y)b(x+y)\gg \delta^{5/4} N^2$ for some $1$-bounded $b:\Z\to\C$}.$$
\end{lemma}

\begin{proof}
We begin with the first statement; the second follows by symmetry.  For each $y \in \Z$, define the function $f_y:\Z \to \Z$ via $f_y(x):=f(x,y)$.  Then we have the identity
$$\norm{f}_{\be_1\cdot[\pm N']}^2=\sum_{y \in [N]} \norm{f_y}_{[\pm N']}^2.$$
The popularity principle supplies at least $\delta N/2$ values of $y \in [N]$ for which $\norm{f_y}_{[\pm N']}^2 \geq \delta N/2$.  Applying Lemma \ref{L: 1-dim U^1 inverse} for each such $y$ and using nonnegativity for the remaining $y$'s, we find that
$$\sum_{y \in [N]} \abs{\sum_{x \in \Z} f(x,y)} \gg \delta N \cdot \delta^{1/4} N.$$
Adding $1$-bounded fudge factors $b(y)$ to remove the absolute values, we obtain
$$\sum_{x,y} f(x,y)b(y) \gg \delta^{5/4}N^2,$$
as desired.

The proof of the third statement goes similarly.  This time, we foliate $[N]^2$ using $1$-dimensional arithmetic progressions with common difference $\be_2-\be_1$; each such arithmetic progression intersects $[N]^2$ in at most $N$ elements.  Applying the popularity principle and Lemma \ref{L: 1-dim U^1 inverse}, and then removing the absolute values as in the first statement, we arrive at the desired conclusion.
\end{proof}

\subsection{Unnormalized box norms}
To prove inverse theorems for higher-degree box norms, it is often easier to start with unnormalized box norms. Given $\bv_1, \ldots, \bv_s\in\Z^D$, we define the \emph{unnormalized box norm} of a compactly supported function $f:\Z^D\to\C$ along $\bv_1, \ldots, \bv_s$ to be
    \begin{align*}
                \norm{f}_{\bv_1,\ldots, \bv_s}&:= \brac{\sum_{\bx, h_1, \ldots, h_s}\Delta_{\bv_1 h_1, \ldots, \bv_s h_s} f(\bx)}^{1/2^s}.
    \end{align*}
    If $D=1$ and $\bv_1 = \cdots = \bv_s=1$, we also set $\norm{f}_{U^s} := \norm{f}_{\bv_1, \ldots, \bv_s}$.
    The following lemma relates unnormalized and normalized box norms. 
    \begin{lemma}[Relating normalized and unnormalized box norms]\label{L: normalized vs. unnormalized}
        Let $s\geq 2$, let $D, H_1, \ldots, H_s\in\N$, and let $\bv_1, \ldots, \bv_s\in\Z^D$.  For any compactly supported function $f:\Z^D\to\C$, we have
            \begin{align*}
                \norm{f}_{\bv_1\cdot[\pm H_1],\ldots, \bv_s\cdot[\pm H_s]}^{2^s}\leq \prod_{i=1}^s \frac{1}{2H_i-1} \cdot\norm{f}_{\bv_1, \ldots, \bv_s}^{2^{s}}.
            \end{align*}
    \end{lemma}
\begin{proof}
Using the inductive formula for box norms, the pointwise bound $\mu_{H}(h) \leq 1/(2H-1)$, and nonnegativity, we have
\begin{align*}
\norm{f}_{\bv_1\cdot[\pm H_1],\ldots, \bv_s\cdot[\pm H_s]}^{2^s} &= \sum_{h_2, \ldots, h_s}\mu_{H_2}(h_2) \cdots \mu_{H_s}(h_s)\norm{\Delta_{\bv_2 h_2, \ldots, \bv_s h_s} f}_{\bv_1\cdot[\pm H_1]}^2\\
&\leq \prod_{i=2}^s \frac{1}{2H_i-1} \cdot\sum_{h_2, \ldots, h_s}\norm{\Delta_{\bv_2 h_2, \ldots, \bv_s h_s} f}_{\bv_1\cdot[\pm H_1]}^{2}\\
 &=\prod_{i=2}^s \frac{1}{2H_i-1} \cdot \sum_{h_1} \mu_{H_1}(h_1)\norm{\Delta_{\bv_1 h_1}f}_{\bv_2, \ldots, \bv_s}^{2^{s-1}}\\
 & \leq \prod_{i=1}^s \frac{1}{2H_i-1} \cdot \norm{f}_{\bv_1, \ldots, \bv_s}^{2^{s}};
\end{align*}
note that the applications of nonnegativity crucially used the assumption $s \geq 2$.
\end{proof}

\subsection{$U^2$-inverse theorems}
The starting point for $U^2$-inverse theorems is the following well-known $1$-dimensional result (see, e.g., \cite[Lemma 2.4]{Pel20}).
\begin{lemma}[$1$-dimensional $U^2$-inverse theorem]\label{L: 1-dim U^2 inverse}
    Let $\delta\in(0, 1]$ and $H, N\in\N$ with $\delta N \leq H\leq N$.  If $f: \Z \to \mathbb{C}$ is a $1$-bounded function supported on $[N]$, then
    \begin{gather*}
        \norm{f}_{[\pm H], [\pm H]}^4\geq \delta N \quad \Longrightarrow\quad \text{$\sum_{x}f(x)e(a x + b)\gg\delta^{3/2} N$ for some $a,\; b \in \R/\Z$}.
    \end{gather*}
\end{lemma}
\begin{proof}
    By Lemma \ref{L: normalized vs. unnormalized} and the assumption $H\geq \delta N$, we have $\norm{f}_{U^2}^4\gg \delta^3 N^3$.  Fourier inversion and Plancherel's Identity give $$\norm{f}_{U^2}^4=\norm{\hat f}_4^4 \leq \norm{ \hat f}_2^2 \cdot\norm{\hat{f}}_\infty^2=\norm{f}_2^2 \cdot\norm{\hat{f}}_\infty^2\leq N \cdot\norm{\hat{f}}_\infty^2.$$ Hence $\norm{\hat{f}}_\infty\gg \delta^{3/2} N$, which is equivalent to the conclusion of the lemma. 
\end{proof}

\begin{lemma}[$2$-dimensional $U^2$-inverse theorem]\label{L: U^2 inverse}
    Let $\delta\in(0, 1]$ and $H, N\in\N$ with $\delta^{-O(1)}\ll\delta N\leq H\leq N$, and let $f: \Z^2 \to \mathbb{C}$ be a $1$-bounded function supported on $[N]^2$.  Then
\begin{align*}
\norm{f}_{(\be_1\cdot[\pm H])^2}^4\geq \delta N^2 \quad \Longrightarrow\quad &\sum_{x,y}f(x,y)e(a(y) x + b(y))\gg \delta^{3/2} N^2 \\ 
 &\text{for some $1$-bounded $a,\; b: \Z \to \R/\Z$};
\end{align*}
\begin{align*}
\norm{f}^4_{(\be_2\cdot[\pm H])^2}\geq \delta N^2 \quad \Longrightarrow\quad &\sum_{x,y}f(x,y)e(a(x) y + b(x))\gg \delta^{3/2} N^2 \\ 
 &\text{for some $1$-bounded $a,\; b: \Z \to \R/\Z$}.
\end{align*}
\end{lemma}
This lemma follows from Lemma \ref{L: 1-dim U^2 inverse} just as Lemma \ref{L: U^1 inverse} follows from Lemma \ref{L: 1-dim U^1 inverse}.

\subsection{$U^1 \times U^1$-inverse theorems}
Let $\bv_1, \bv_2\in\Z^2$ be vectors such that $\langle \bv_1, \bv_2\rangle_{\Z} = \Z^2$. Then the {unnormalized box norm}  along $\bv_1, \bv_2$ can be written as
    \begin{align*}
                \norm{f}_{\bv_1, \bv_2}
                &=\brac{\sum_{m,m',n,n'}f(\bv_1 m + \bv_2 n)\overline{f(\bv_1 m' + \bv_2 n)f(\bv_1 m + \bv_2 n')}f(\bv_1 m' + \bv_2 n')}^{1/4}
    \end{align*}
    and satisfies the following inverse theorem.
\begin{lemma}[Inverse theorem for the unnormalized $U^1\times U^1$-norm]\label{L: unnormalized U1xU1 inverse}
    Let $\bv_1,\bv_2 \in \Z^2$ be vectors of length $O(1)$ satisfying $\langle \bv_1,\bv_2 \rangle_{\Z}=\Z^2$, and let $N \in \N$.  If $f: \Z^2 \to \C$ is a $1$-bounded function supported on $[N]^2$, then there are $1$-bounded functions $b_1, b_2:\Z\to\C$ such that
    \begin{align*}
        \norm{f}_{\bv_1, \bv_2}^4 \ll N^2\cdot \sum_{m,n}f(\bv_1 m + \bv_2 n)b_1(m)b_2(n).
    \end{align*}
\end{lemma}
\begin{proof}
    The pigeonhole principle gives some $n' \in \Z$, necessarily of size $O(N)$, such that
    \begin{align*}
        \norm{f}_{\bv_1, \bv_2}^4 \ll N\abs{\sum_{m,m',n}f(\bv_1 m + \bv_2 n)\overline{f(\bv_1 m' + \bv_2 n)f(\bv_1 m + \bv_2 n')}f(\bv_1 m' + \bv_2 n')}.
    \end{align*}
    Let $S := \{m'\in \Z:\; \bv_1 m' + \bv_2 n'\in[N]^2\}$.
    Then $|S|\ll N$, and the result follows upon setting
    $b_1(m): = \overline{f(\bv_1 m + \bv_2 n')}$ and letting $b_2(n)$ be equal to $\E_{m'\in S} \overline{f(\bv_1 m' + \bv_2 n)}f(\bv_1 m' + \bv_2 n')$ times a suitable complex number of norm $1$ that makes the expression in the absolute value positive.
\end{proof}

Combining this lemma with Lemma \ref{L: normalized vs. unnormalized} immediately gives the following result for unnormalized box norms.
\begin{lemma}[Inverse theorem for the $U^1\times U^1$-norm]\label{L: U1xU1 inverse}
Let $\delta\in(0, 1]$ and $H, N\in\N$ with $\delta N\leq H\leq N$.  Let $f: \Z^2 \to \mathbb{C}$ be a $1$-bounded function supported on $[N]^2$.  Then
\begin{align*}
\norm{f}_{\be_1\cdot[\pm H],\; \be_2\cdot[\pm H]}^4\geq \delta N^{2} \quad \Longrightarrow\quad &\sum_{x,y}f(x,y)b_1(x)b_2(y)\gg \delta^3 N^2\\
 &\text{for some $1$-bounded $b_1,b_2: \Z \to \C$};
\end{align*}
\begin{align*}
\norm{f}_{\be_1\cdot[\pm H],\; (\be_2-\be_1)\cdot[\pm H]}^4\geq \delta N^{2} \quad \Longrightarrow\quad &\sum_{x,y}f(x,y)b_1(x+y)b_2(y)\gg \delta^3 N^2\\
 &\text{for some $1$-bounded $b_1,b_2: \Z \to \C$};
\end{align*}
\begin{align*}
\norm{f}_{\be_2\cdot[\pm H],\; (\be_2-\be_1)\cdot[\pm H]}^4\geq \delta N^{2} \quad \Longrightarrow\quad &\sum_{x,y}f(x,y)b_1(x)b_2(x+y)\gg \delta^3 N^2\\
 &\text{for some $1$-bounded $b_1,b_2: \Z \to \C$}.
\end{align*}
\end{lemma}

\subsection{$U^2 \times U^1$-inverse theorem}
Recall that our degree-lowering result Theorem \ref{thm:degree lowering-lambda} is based on replacing ``$U^2 \times U^1$-obstructions'' with ``$U^1 \times U^1$-obstructions''.  We are finally ready to give a precise characterization of the former obstructions.

\begin{lemma}\label{L: U^2 x U^1 inverse}
Let $\delta\in(0, 1]$ and $H, N\in\N$ with $\delta^{O(1)}\ll\delta N\leq H\leq N$, and let $f: \Z^2 \to \mathbb{C}$ be a $1$-bounded function supported on $[N]^2$.  Then
\begin{align*}
\|f\|_{\be_1\cdot[\pm H],\; \be_1\cdot[\pm H],\; \be_2\cdot[\pm H]}^8 \geq \delta N^2 \Longrightarrow \sum_{x,y}f(x,y)g(x)e(a(y)x+b(y)) \gg\delta^{O(1)} N^2\\
 \text{for some $1$-bounded $g: \Z \to \C$ and phase functions $a,b: \Z \to \R/\Z$};
\end{align*}
\begin{align*}
\|f\|_{\be_1\cdot[\pm H],\; \be_2\cdot[\pm H],\; \be_2\cdot[\pm H]}^8 \geq \delta N^2 \Longrightarrow \sum_{x,y}f(x,y)g(y)e(a(x)y+b(x)) \gg\delta^{O(1)} N^2\\
\text{for some $1$-bounded $g: \Z \to \C$ and phase functions $a,b: \Z \to \R/\Z$}.
\end{align*}
\end{lemma}

\begin{proof}
We only prove the first statement as the other can be proved by swapping the roles of $x$ and $y$.  The inductive formula for box norms gives
$$\delta^{8} N^2 \leq \|f\|_{\be_1\cdot[\pm H],\; \be_1\cdot[\pm H],\; \be_2\cdot[\pm H]}^8=\sum_h \mu_H(h) \|\Delta_{h\be_2}f\|_{\be_1\cdot[\pm H],\; \be_1\cdot[\pm H]}^4.$$
By Lemma \ref{L: U^2 inverse}, there are phases $\alpha(y,y+h),\; \beta(y,y+h) \in \mathbb{R}/\mathbb{Z}$ such that
$$\sum_{x,y,h} \mu_H(h)\Delta_{h\be_2} f(x,y)e(\alpha(y,y+h)x+\beta(y,y+h)) \gg \delta^{O(1)}N^2$$
and the summand is real and nonnegative for each fixed $h,y$.  In particular, we have
$$\sum_{x,y, h}\Delta_{h\be_2} f(x,y) e(\alpha(y,y+h)x+\beta(y,y+h)) \gg \delta^{O(1)}N^3.$$
Replacing the variable $h$ with the variable $z:=y+h$, we get
$$\sum_{x,y,z}f(x,y)\overline{f(x,z)} e(\alpha(y,z)x+\beta(y,z)) \gg \delta^{O(1)}N^3.$$
The pigeonhole principle gives an integer $|z_0| \ll N$ such that
$$\sum_{x,y}f(x,y)\overline{f(x,z_0)} e(\alpha(y,z_0)x+\beta(y,z_0)) \gg \delta^{O(1)}N^2.$$
The lemma now follows upon setting 
$$g(x):=\overline{f(x,z_0)}, \quad a(y):=\alpha(y,z_0), \quad \text{and} \quad b(y):=\beta(y,z_0).$$
\end{proof}

\section{Nilsequence definitions}\label{A: nilsequence theory}

This section contains (mostly standard) definitions pertaining to nilsequences.

\begin{definition}[$(t - 1)$-fold commutator]
    For elements $g_1, \dots, g_t$ of a group $G$, we define
    $$[g_1, g_2, \dots, g_t] := [[[g_1, g_2], g_3], \dots, g_t].$$
    We will call any such commutator a \emph{$(t - 1)$-fold commutator} of the elements $g_1, \dots, g_t$.
\end{definition}

We now introduce the lower central series.
\begin{definition}\label{d:lowercentraseries}
The \emph{lower central series} of a nilpotent Lie group $G$ is the sequence of nested subgroups $G=G_{(0)}=G_{(1)} \supseteq G_{(2)} \supseteq G_{(3)} \cdots $ where $G_{(i)} := [G, G_{(i - 1)}]$ for $i \geq 1$.  The \textit{step} of $G$ is the smallest integer $s$ such that $G_{(s + 1)} = \mathrm{Id}_G$.
\end{definition}

Recall that the \emph{height} of the rational number $a/b$ (for $a,b$ coprime integers) is $\max(|a|,|b|)$.
We will now define a nilmanifold.

\begin{definition}[Nilmanifold]
A \emph{nilmanifold} of degree $k$, step $s$, complexity at most $M$, and dimension $m$ consists of the following data:
\begin{enumerate}
    \item an $s$-step connected and simply connected nilpotent real Lie group $G$ and a discrete cocompact subgroup $\Gamma$ of $G$;
    \item a \emph{filtration} $G_\bullet = (G_i)_{i = 0}^\infty$ of nested subgroups $G = G_0=G_1 \supseteq G_2 \supseteq \cdots$ such that $G_{i}=\{\mathrm{Id}_G\}$ for all $i>k$ and $[G_i,G_j] \subseteq G_{i+j}$ for all $i,j$; and
    \item a \emph{Mal'cev basis}, i.e., a basis $\{X_1, \dots, X_m\}$ of the Lie algebra $\mathfrak{g} := \log(G)$.
\end{enumerate}
We furthermore require the data to satisfy the following conditions:
\begin{enumerate}
    \item For $1 \le i, j \le m$, we have
    $$[X_i, X_j] = \sum_{l > \max(i, j)} a_{ijl} X_l$$
    for some rational numbers $a_{ijl}$ of height at most $M$.
    \item For $m_i := \text{dim}(G_i)$, we have $\log(G_i) = \text{span}\{X_j: m - m_i < j \le m\}$.
    \item Each element $g \in G$ can be expressed uniquely as $\exp(t_1X_1)\exp(t_2X_2) \cdots \exp(t_mX_m)$ with $t_1, \dots, t_m \in \mathbb{R}$; we let $\psi: g \mapsto (t_1, \ldots, t_m)$ denote the associated \emph{coordinate map} from $G$ to $\mathbb{R}^m$.
    \item $\Gamma$ consists of precisely the elements $g$ for which $\psi(g) \in \mathbb{Z}^m$.
\end{enumerate}
A Mal'cev basis that satisfies the last three of these conditions is said to be \textit{adapted} to the filtration $G_\bullet$.
\end{definition}

Note that condition (ii) implies that each $G_i$ is closed, connected, and simply connected.

The lower central series (as defined above) gives the so-called \textit{standard filtration} on a nilpotent Lie group $G$, and we emphasize that there are filtrations other than the standard filtration.  The standard filtration on $G$ is ``minimal'' in the sense that any other filtration $(G_i)$ satisfies $G_i \supseteq G_{(i)}$ for each $i\in\N$ \cite[Exercise 1.6.2]{T12}.

A nilmanifold can be endowed with a metric as follows.
\begin{definition}[Metric on a nilmanifold]\label{d:metricnilmanifold}
Given a nilmanifold $G/\Gamma$ with coordinate map $\psi$, we define a metric $d_G$ on $G$ to be the right-invariant metric
$$d_{G}(x, y) := \inf\left\{\sum_{i = 1}^n \min\{|\psi(x_ix_{i - 1}^{-1})|, |\psi(x_{i - 1}x_i^{-1})|\}: x_0 = x, x_n = y\right\}.$$
Setting $d(x\Gamma, y\Gamma) := \inf\{d(x', y'): x'\Gamma = x\Gamma, y'\Gamma = y\Gamma\}$, we obtain a metric on $G/\Gamma$.
\end{definition}
This metric makes $G/\Gamma$ into a compact metric space and hence can be used to define the Lipschitz parameter of a function on $G/\Gamma$.
\begin{definition}[Lipschitz norm]
    Let $(X, d_X)$ be a metric space.  The \emph{Lipschitz constant} of a function $F\colon X \to \mathbb{C}$ is the quantity $$\sup_{x \neq y \in X} \frac{|F(x) - F(y)|}{d_X(x, y)},$$ and the \emph{Lipschitz norm} of $F$ is
$$\|F\|_{\mathrm{Lip}(X)} := \|F\|_{L^\infty(X)} + \sup_{x \neq y \in X} \frac{|F(x) - F(y)|}{d_X(x, y)}.$$
If $\|F\|_{\mathrm{Lip}(X)} \le M$, then we say that $F$ is an \emph{$M$-Lipschitz function}.
\end{definition}
An important property of Lipschitz functions on a nilmanifold is that they can be approximated by finite linear combinations of vertical characters (as defined below). 

We now turn to polynomial sequences.
\begin{definition}[Polynomial sequences]
Let $G$ be an $s$-step connected and simply connected nilpotent real Lie group with a degree-$k$ filtration $(G_i)_{i = 0}^\infty$. 
The set $\mathrm{poly}(\mathbb{Z}, G)$ of \emph{polynomial sequences} to $G$ consists of the maps $g\colon \mathbb{Z} \to G$ of the form
$$g({n}) = \prod_{{i}=0}^k g_{{i}}^{\binom{{n}}{{i}}}$$
for some $g_i \in G_i$. 
It is known (see, e.g., \cite[Chapter 14]{HK18}) that $\mathrm{poly}(\mathbb{Z}, G)$ forms a group under pointwise multiplication.
\end{definition}

Our next definition is a smoothness norm on real polynomials that captures the extent to which their coefficients are close to integers. 
\begin{definition}[Smoothness norm]\label{def:smoothness}
Let $p(x)=\sum_{j=0}^D \alpha_j x^j \in \mathbb{R}[x]$ be a polynomial of degree $D$.  Its \emph{smoothness norm} at scale $T>0$ is defined to be
$$\norm{p}_{C^\infty[T]}:=\max_{1 \leq j \leq D}T^{j} \norm{\alpha_j}_{\mathbb{R}/\mathbb{Z}}.$$
\end{definition}
Note that this definition of smoothness is written using the monomial basis, which differs slightly from \cite[Definition 2.2]{Leng23b}. By \cite[Lemma A.10]{Leng23b}, these notions are equivalent up to a multiplicative factor of $O_D(1)$.

Two important special classes of polynomial sequences on $G$ are those whose coefficients are ``rational'' and those which grow slowly and remain ``close'' to the identity.

\begin{definition}[Rational and smooth sequences] Given $Q\in\N$, we say an element $g$ of $G$ is \emph{rational with height $Q$} (or \emph{$Q$-rational}) if $g^Q \in \Gamma$. We say that a polynomial sequence $\gamma$ on $G$ is \emph{$Q$-rational} if $\gamma({n})$ is $Q$-rational for each ${n} \in\mathbb{Z}$.

We say that a polynomial sequence $\eps$ on $G$ is \emph{$(K, {T})$-smooth} if for each $n\in\Z$, we have
$$d_{G}(\eps(n +1), \eps(n)) \le K/T \quad \text{ and }\quad d_{G}(\eps(0), \mathrm{id}_G) \le K.$$
\end{definition}

We are finally ready to define a nilsequence.
\begin{definition}[Nilsequence]\label{D: nilsequence}
A \emph{nilsequence} of \textit{complexity} $M$ is a sequence $\varphi: \Z\to\C$ of the form $\varphi(n) = F(g(n)\Gamma)$, where $g$ is a polynomial sequence on $G$ and $F\colon G/\Gamma \to \mathbb{C}$ is an $M$-Lipschitz function on a nilmanifold $G/\Gamma$ of complexity $M$.
\end{definition}

We now shift gears and discuss horizontal and vertical characters on $G/\Gamma$.

\begin{definition}[Horizontal and vertical components]
The \emph{horizontal component} of $G$ is $G/G_{(2)} = G/[G, G]$, and the \emph{vertical component} of $G$ is $G_{(s)}$.\footnote{In \cite{GT12}, the vertical component is defined to be $G_\ell$ where $G_\ell$ is the smallest nontrivial element of the filtration of $G$.} We let $\pi_{\mathrm{horiz}}\colon G \to G/G_{(2)}$ denote the quotient map. We use  $d_{\mathrm{horiz}}$ and $d_{(s)}$ to denote the dimensions of $G/G_{(2)}$ and $G_{(s)}$, respectively.
\end{definition}

We remark that $G/[G, G]\Gamma$, often called the \emph{horizontal torus}, is isomorphic to the torus $(\R/\Z)^{d_{\mathrm{horiz}}}$; the Mal'cev basis adapted to $G/\Gamma$ provides an isomorphism.  Horizontal characters come from characters of the horizontal torus, as follows.

\begin{definition}[Horizontal character]
A \emph{horizontal character} on $G/\Gamma$ is a continuous homomorphism $\eta\colon G \to \mathbb{R}$ such that $\eta(\Gamma) \subseteq \mathbb{Z}$.
\end{definition}

Note that since $\mathbb{R}$ is abelian, every horizontal character $\eta$ annihilates $[G, G]$ and hence descends to a character on the horizontal torus.  Since the characters of $(\R/\Z)^{d_{\mathrm{horiz}}}$ correspond to elements of $\Z^{d_{\mathrm{horiz}}}$, our choice of Mal'cev basis provides a correspondence between horizontal characters of $G/\Gamma$ and elements of $\Z^{d_{\mathrm{horiz}}}$.  More explicitly, each $\eta$ is of the form $\eta(g) = k \cdot \psi(g)$ for some $k \in \Z^{d_{\mathrm{horiz}}}$.

We will also require the notion of a vertical character.
\begin{definition}[Vertical character]
Let $H$ be a connected, simply connected subgroup of the center $Z(G)$ which is \emph{rational} in the sense that $H \cap \Gamma$ is cocompact in $H$, and let $\eta: H \to \R$ be a continuous homomorphism such that $\eta(H \cap \Gamma) \subseteq \Z$.  We say that a function $F: G/\Gamma \to \C$ is an \emph{$H$-vertical character} with \emph{frequency} $\eta$ if
$$F(gx)=e(\eta(g))F(x)$$
for all $g \in H$ and $x \in G/\Gamma$. We refer to a $G_{(s)}$-vertical character as simply a \emph{vertical character}.
\end{definition}

Horizontal and vertical characters come with natural notions of size.
\begin{definition}[Size of a character]
If $\eta$ is a horizontal character on $G/\Gamma$ associated with the element $k \in (\R/\Z)^{d_{\mathrm{horiz}}}$, then the \emph{size} (or \emph{modulus}) of $\eta$, denoted $|k|$ or $\|k\|_\infty$, is the $L^\infty$ norm of the vector $k$. 

Let $H$ be a rational subgroup of the center $Z(G)$, and let $\xi\colon H/(H\cap\Gamma) \to \mathbb{R}/\mathbb{Z}$ be a homomorphism.  The \emph{size} (or \emph{modulus}) of $\xi$, denoted $|\xi|$, is the Lipschitz constant of the function $x \mapsto e(\xi(x))$ on $H/(H\cap\Gamma)$. Equivalently,
$$|\xi| := \sup_{x \neq y \in H/(H\cap\Gamma)} \frac{|e(\xi(x)) - e(\xi(y))|}{d_{G/\Gamma}(x, y)}.$$ 
\end{definition}

Finally, we will need quantitative measures of rationality for subspaces and subgroups.
\begin{definition}[Rational subspace]
Let $V$ be a vector space with a basis $\mathcal{B}$, and let $W$ be a subspace of $V$. We say that $W$ is \emph{at most $Q$-rational} if $W$ has a basis $\mathcal{B}'$ such that each $w_j\in \mathcal{B}'$ can be written as
\[w_j = \sum_{v_i\in \mathcal{B}} a_{ji} v_i\]
for some rationals $a_{ji}$ of height at most $Q$.
\end{definition}
\begin{definition}[Rational subgroup and subnilmanifold]
Let $G/\Gamma$ be a nilmanifold with Mal'cev basis $\mathcal{X}$.
A connected and simply connected subgroup $H\subseteq G$ is \textit{$Q$-rational} if $\mathfrak{h} := \log(H)$ is $Q$-rational with respect to $\mathfrak{g}$ with the basis $\mathcal{X}$.  Define $\Gamma':=H\cap\Gamma$.  We then say that the nilmanifold $H/\Gamma'$, equipped with the filtration $H_\bullet = (H_i)_{i=0}^\infty$ given by $H_i := G_i\cap H$, is a \textit{$Q$-rational subnilmanifold} of $G/\Gamma$.
\end{definition}
\begin{remark}
If $H$ is a $Q$-rational subgroup of $G$, then by \cite[Lemma B.12]{Leng23a} there exists a $Q^{O_k(m^{O(1)})}$-rational Mal'cev basis for $H/(\Gamma \cap H)$.
\end{remark}

\section{Fraction-comparison estimates}\label{A: fraction comparison}
In this section, we prove elementary ``fraction-comparison estimates'' in the style of \cite{PSS23} (see, e.g., the proof of Proposition 5.6 there).  These estimates are useful in the nilsequences part of our argument.  While the first two lemmas are straightforward, the third makes use of some important properties of $V$-tricked polynomials.  Recall the smoothness norms from Definition \ref{def:smoothness}.

\begin{lemma}\label{extrapolation}
    Let $D, Q, T\in\N$, let $R\in\R[z]$ be a polynomial of degree $D$, and let $a, b \in \Z$ with $|a| \le QT$ and $0<|b| \leq Q$ nonzero. Then the polynomial $\tilde{R}(z):= R(a + bz)$ satisfies 
    $$\|b^D R\|_{C^\infty[T]} \ll_D Q^{D-1}\|\tilde{R}\|_{C^\infty[T]}.$$
\end{lemma}
\begin{proof}
Assume that $\|\tilde{R}\|_{C^\infty[T]} \le M$; we will show that $\|b^D R\|_{C^\infty[T]} \ll_D Q^{D-1} M$.  Write $R(z) = \sum_{i = 0}^D \alpha_i z^i$. Then
    \begin{align*}
        \tilde R(z)=R(a + bz) = \sum_{i = 0}^D \alpha_i (a + bz)^i = \sum_{i = 0}^D \alpha_i \sum_{j = 0}^i {i \choose j} b^jz^{j} a^{i - j} = \sum_{j =0}^D b^jz^{j} \sum_{i=j}^D \alpha_i {i \choose j} a^{i - j}.    
    \end{align*}
     The assumption $\|\tilde{R}\|_{C^\infty[T]} \le M$ means that
     \begin{align}\label{E: R/Z equations}
         \norm{b^j\cdot \sum_{i=j}^D  \alpha_i {i \choose j} a^{i - j}}_{\mathbb{R}/\mathbb{Z}} \le M/T^{j}
     \end{align}
    for each $1 \leq j \leq D$.  In particular the $j=D$ instance of \eqref{E: R/Z equations} gives $$\norm{b^D \alpha_D}_{\R/\Z} \leq M/T^D.$$
    Next, the $j=D-1$ instance of \eqref{E: R/Z equations} gives $$\norm{b^{D-1}(\alpha_{D-1}+Da\alpha_D)}_{\R//Z} \leq M/T^{D-1};$$
    multiplying the expression in $\norm{\cdot}_{\R/\Z}$ by $b$ and using the triangle inequality (and the assumption on the sizes of $a,b$), we find that
    $$\norm{b^D \alpha_{D-1}}_{\R/\Z} \leq QM/T^{D-1}+D(QT)(M/T^D) \ll_D QM/T^{D-1}.$$
    Continuing inductively in this manner, we find that
    $$\norm{b^D \alpha_{j}}_{\R/\Z} \ll_D Q^{D-j}M/T^j \leq Q^{D-1}M/T^j$$
    for each $1 \leq j \leq D$, as desired.
\end{proof}

For ease of reference we record the following special case of this lemma.

\begin{lemma}\label{translationextrapolation}
Let $D, T\in\N$, and let $R\in\R[z]$ be a polynomial of degree $D$. If $|y_0| \le QT$ and
$$\|R(y_0 + \cdot)\|_{C^\infty[T]} \le Q,$$
then
$$\|R\|_{C^\infty[T]} \ll_D Q^{D}.$$
\end{lemma}

Recall that we are interested in the $V$-tricked polynomial
$$\widetilde P_{[r, V]}(z):=\frac{\widetilde P(Vz+r)-\widetilde P(r)}{V},$$
where $V=W^{O_P(1)}$ is a nonnegative integer power of $W$ and $r\in[V]$; as mentioned previously, $\widetilde P_{[r, V]}$ has integer coefficients, zero constant term and degree $d\geq 2$.
Let us denote the coefficient of $z^i$ in $\widetilde P_{[r, V]}(z)$ by $V_i \in \mathbb{Z}[W,V,r,\; \beta_1, \ldots,\; \beta_d]$.  We record that
$$V_d=\beta_d \beta_1^{d-2} W^{d-1}V^{d-1}$$
and that $V_1$ is equal to $1$ plus $W$ times some integer-coefficient polynomial in $W,V,r,\; \beta_1, \ldots,\; \beta_d$.  Thus, if every prime factor of $\beta_1\beta_d$ also divides $W$ (which is the case for $w$ larger than an absolute constant depending only on $P$), then $V_1$ is equivalent to $1$ modulo every prime dividing $V_d$ and we have $$\gcd(V_d,V_1)=1.$$  The following lemma tells us that precomposing a polynomial with $\widetilde P_{[r, V]}$ cannot significantly decrease its (properly normalized) smoothness norm.

\begin{lemma}\label{lem:composing-polynomials-2}
Let $D\in\N$. For every $C_1>0$ there exists a constant $C_2 = C_2(C_1, D,P)>0$ such that the following holds.  Let $V \leq W^{C_1}$ be a {nonnegative} integer power of $W$, let $r \in [V]$, and let $T, Q \in\N$. Let $R \in \mathbb{R}[z]$ be a polynomial of degree $D$. Assume that every prime factor of $\beta_1\beta_d$ also divides $W$. If
$$\|R \circ \widetilde{P}_{[r, V]}\|_{C^\infty[(T/V_d)^{1/d}]} \leq Q$$
and $T>C_{2} W^{C_{2}}$,
then
$$\|R\|_{C^\infty[T]} \leq C_{2}Q.$$
\end{lemma}

The scales of these smoothness norms come from the estimate $\widetilde P_{[r,V]}((T/V_d)^{1/d}) \asymp T$. 

\begin{proof}
We fix $C_1>0$ and let all the constants depend on $P$ and $D$. The proof proceeds by induction on the degree $D$, where the base case $D=0$ holds vacuously. 
For the induction step, write $R(z)=\sum_{i=0}^D \alpha_i z^i$.  Recall that $\widetilde P_{[r, V]}(z)=V_d z^d+\cdots+V_1 z$.  We compute
\begin{align*}
R(\widetilde P_{[r, V]}(z)) &=\sum_{i = 0}^D \alpha_i (V_d z^d+\cdots +V_1 z)^i  \\
 &=\sum_{i = 0}^D \alpha_i \sum_{j_1+\cdots+j_d=i}\binom{i}{j_1, \ldots, j_d}(V_1z)^{j_1}\cdots (V_d z^d)^{j_d}\\
 &=\sum_{k=0}^{dD} z^k \sum_{i = 0}^D \alpha_i  \sum_{{\underline{j}} \in \mathcal{P}(i,k)}\binom{i}{j_1, \ldots, j_d}V_1^{j_1}\cdots V_d^{j_d},
\end{align*}
where $\mathcal{P}(i,k)$ denotes the set of tuples $\underline{j} = (j_1, \ldots, j_d)$ of nonnegative integers satisfying $j_1+\cdots+j_d=i$ and $j_1+2j_2+
\cdots+dj_d=k$.  The assumption $\|R \circ \widetilde P_{[r, V]}\|_{C^\infty [(T/V_d)^{1/d}]} \le Q$
tells us that
\begin{equation}\label{eq:expand-compose-poly}
\norm{\sum_{i = 0}^D \alpha_i  \sum_{{\underline{j}} \in \mathcal{P}(i,k)}\binom{i}{j_1, \ldots, j_d}V_1^{j_1}\cdots V_d^{j_d}}_{\mathbb{R}/\mathbb{Z}} \leq Q (V_d/T)^{k/d}
\end{equation}
for each $1 \leq k \leq dD$.  Let $E_k$ denote the expression inside $\| \cdot \|_{\mathbb{R}/\mathbb{Z}}$ in \eqref{eq:expand-compose-poly}.

We will show that $\|\alpha_D\|_{\mathbb{R}/\mathbb{Z}} \leq Q/T^D$.  
Since $E_{dD}=\alpha_D V_d^D$, Equation \eqref{eq:expand-compose-poly} for $k=dD$ tells us that
$$\norm{\alpha_D V_d^D}_{\mathbb{R}/\mathbb{Z}} \leq Q(V_d/T)^D,$$
and hence there is some integer $n$ such that
\begin{equation}\label{eq:fraction-comparison-first-estimate}
|\alpha_D-n/V_d^D| \leq Q/T^D.
\end{equation}
We will aim to show that $n$ is divisible by $V_d^D$.  To accomplish this, we will find a second rational approximation for $\alpha_D$ using a denominator that is coprime to $V_d^D$.  To begin, we have $E_1=\alpha_1V_1$.  Next, $E_2=\alpha_2 V_1^2+\alpha_1 V_2$, and so
$$V_1E_2-V_2E_1=V_1^3 \alpha_2.$$
We continue with $E_3=\alpha_3 V_1^3+\alpha_2(2V_1V_2)+\alpha_1 V_3$ and
$$V_1^2 E_3-2V_2(V_1E_2-V_2E_1)-V_1 V_3 E_1=V_1^5 \alpha_3.$$
Continuing inductively in this manner, we obtain a positive integer $\gamma$ and polynomials $s_1, \ldots, s_d \in \mathbb{Z}[V_1, \ldots, V_d]$ such that
$$s_1E_1+\cdots+s_dE_d=V_1^{\gamma}\alpha_D$$
holds (as an identity in the formal variables $V_1, \ldots, V_d, \alpha_1, \ldots, \alpha_D$); note that the $s_i$'s depend only on the parameters $d,D$.  The crucial point for the induction is that $E_k$ (for $1 \leq k \leq D$) equals $\alpha_k V_1^k$ plus a linear combination of $\alpha_{k-1}, \ldots, \alpha_1$ with coefficients in $\mathbb{Z}[V_1, \ldots, V_d]$.

Recalling that $V_i=O(W^{d-1}V^{d-1})=O(W^{O(1)})$ 
and applying \eqref{eq:expand-compose-poly} for $1 \leq k \leq D$, we find that
$$\norm{V_1^{\gamma}\alpha_D}_{\mathbb{R}/\mathbb{Z}}=O(QW^{O(1)}/T^{1/d}).$$ 
Thus there is an integer $n'$ such that
$$\abs{\alpha_D-n'/V_1^\gamma}=O(QW^{O(1)}/T^{1/d}).$$
Comparing this with \eqref{eq:fraction-comparison-first-estimate} gives
$$\abs{n/V_d^D-n'/V_1^\gamma}=O(QW^{O(1)}/T^{1/d}).$$
Now we must have
$$n/V_d^D=n'/V_1^\gamma$$
since otherwise we would have $\abs{n/V_d^D-n'/V_1^\gamma} \geq 1/(V_d^D V_1^\gamma)$, which (due to the constraint on $T$) is not of size $O(QW^{O(1)}/T^{1/d})$.  Since $V_1,V_d$ are coprime, we conclude that $n$ is divisible by $V_d^D$, and \eqref{eq:fraction-comparison-first-estimate} tells us that
$$\norm{\alpha_D}_{\mathbb{R}/\mathbb{Z}} \leq Q/T^D,$$
as desired.

We claim that
$$\norm{\alpha_D (\widetilde P_{[r, V]})^D}_{C^\infty [(T/V_d)^{1/d}]}=O(Q).$$
Indeed, since each $V_i \ll (WV)^{d-1} \ll V_d$, we see that the $z^i$-coefficient of $(\widetilde P_{[r, V]})^D$ has size $\ll V_d^{D}$.  Then for each $0 \leq i \leq dD$,
the distance from the $z^i$-coefficient of $\alpha_D (\widetilde P_{[r, V]})^D$ to the nearest integer is at most
$$\ll (Q/T^D)\cdot V_d^{D} \leq (V_d/T)^{i/d}Q;$$
this establishes the claim.   Hence
$$\norm{(R\circ \widetilde P_{[r, V]})-\alpha_D (\widetilde P_{[r, V]})^D}_{C^\infty [(T/V_d)^{1/d}]}=O(Q),$$
and the lemma follows from the induction hypothesis.
\end{proof}

\section{S\'ark\"ozy-type results}\label{A: Sarkozy}
The goal of this section is to make precise the heuristic that
\begin{align}\label{E: Sarkozy heuristic}
    \sum_\bx \E_{z\in[(N/V_d)^{1/d}]}f_0(\bx) f_1(\bx + \bv \widetilde{P}_{[r, V]}(z)) \approx \sum_{\bx}\E_{z\in[N]}f_0(\bx)f_1(\bx + \bv z),
\end{align}
where $V=W^{O_P(1)}$ is a nonnegative integer power of $W$, $r\in[V]$ (in particular, if $V = 1$, then $r=0$ and hence $\widetilde{P}_{[r, V]} = \widetilde{P}$), and $V_i$'s are the coefficients of $\widetilde{P}_{[r, V]}$.
We use circle-method arguments to make this heuristic precise.  These arguments originated in work of S\'ark\"ozy \cite{Sa78a, Sa78b}, and we will more closely follow a method due to Green \cite{Gr02}.

 Throughout this section, define $K: = (N/V_d)^{1/d}$, so that $|\widetilde{P}_{[r, V]}(z)|\ll_P N$ {for $z \in [K]$}.
For $\xi\in\R/\Z$, we also let
\begin{align*}
    S(\xi) := \sum_{z\in[K]}e(\xi \widetilde{P}_{[r, V]}(z))
\end{align*}
denote the Fourier transform of the (multi)set $\widetilde{P}_{[r, V]}([K])$.

\subsection{Circle method preliminaries}

The first step towards \eqref{E: Sarkozy heuristic} is showing that the counting operator for a $2$-point polynomial progression is always controlled by the major arcs of its arguments.  Let us make a few more definitions.  Let $\veps>0$ be a constant to be determined later. For coprime $0\leq a < q \leq K^\veps$, we define the \emph{major arc}
\begin{align*}
    \mathfrak{M}_{a,q} := \left\{\xi\in\R/\Z:\; \abs{\xi - \frac{a}{q}}\leq K^{-(d-\veps)}\right\}.
\end{align*}
As long as $\veps<1/3$, the intervals $\mathfrak{M}_{a,q}$ are disjoint.
So we partition $\R/\Z$ into the \emph{major arcs} $\mathfrak M$ and the \emph{minor arcs} $\mathfrak m$, where
\begin{align*}
    \mathfrak{M} := \bigcup_{\substack{1\leq a\leq q\leq K^\veps,\\ \gcd(a,q) = 1}}\mathfrak{M}_{a,q}
    \quad \textrm{and}\quad \mathfrak{m} := (\R/\Z)\setminus\mathfrak{M}.
\end{align*}
A key input for controlling the minor arcs is the following classical lemma that upper-bounds the even moments of $F$ (see, e.g., \cite[Corollary 14.2]{Wo19}). 
\begin{lemma}[Crude bound on even moments of $F$]\label{L: 4-th moment}
    There exists some positive integer $s_0(d) > d$ such that for all $s\geq s_0(d)$ and $\eta >0$, we have
    \begin{align*} 
        \int_{\R/\Z}\abs{S(\xi)}^{2s}\; d\xi  \ll_\eta K^{2s - d +\eta}.
    \end{align*}
\end{lemma}

The following lemma is a generalization of \cite[Lemma B.4]{PSS23}, and it follows from a nearly identical proof, which we omit.
\begin{lemma}[Minor arc control]\label{L: minor arcs}
    Let $s_0(d)$ be as in Lemma \ref{L: 4-th moment}, and let $s> s_0(d)$ be a positive integer. There exists an absolute $C=C(d)\geq 1$ such that for all $\veps\in(0,1)$
    and positive integers $K\gg_{d,\veps} 1$, if the coefficients of $\widetilde{P}_{[r, V]}$ have size at most $K^{\veps/4}$, then 
    we have the pointwise bound
    \begin{align*}
        \sup_{\xi\in\mathfrak{m}}|S(\xi)|
        \leq K^{1-\veps/C}
    \end{align*}
    as well as the $L^{2s}$-bound
    \begin{align*}
        \int_\mathfrak{m}\abs{S(\xi)}^{2s}
        \ll_{s, \veps}K^{2s-d-\veps/C}.
    \end{align*}

\end{lemma}

In order to handle the major arcs, we start with the following classical decomposition result for $F$  (see e.g. \cite[Corollary 5.2]{Wo23}).
\begin{lemma}\label{L: major arc expression}
    Let $0<\veps<1/3$ and $0\leq a < q \leq K^\veps$ with $a,q$ coprime.
    If $\xi\in\mathfrak{M}_{a,q}$, then 
    \begin{align*}
        S(\xi)
        = \E_{u\in[q]}e(a \widetilde{P}_{[r, V]}(u)/q)\int_0^K e\brac{\brac{\xi-\frac{a}{q}}\widetilde{P}_{[r, V]}(z)}dz + O(K^{2\veps}).
    \end{align*}
\end{lemma}
We recall that the assumption $\veps<1/3$ ensures the disjointness of the major arcs.

In what follows, let $$S(a,q):= \E_{u\in[q]}e(a \widetilde{P}_{[r, V]}(u)/q)=\E_{u\in\Z/q\Z}e(a \widetilde{P}_{[r, V]}(u)/q)$$ be the exponential sum appearing in Lemma \ref{L: major arc expression}. The following lemma gives bounds on $S(a,q)$ for various values of $a,q$.  It is inspired by \cite[Lemma 5.3]{CC23}; the proof is almost identical, so we omit it.
\begin{lemma}\label{L: S(a,q)}
    Let $1\leq a < q$ be coprime positive integers.  
    Then 
    \begin{align*}
        |S(a,q)|\ll_{\eta, P} q^{\eta - 1/d}\quad \textrm{for\; every}\quad \eta>0.
    \end{align*}
    Moreover, $S(a,q) = 0$ whenever $\gcd(q, W)>1$. 
\end{lemma}

Next, we upper-bound the integral appearing in Lemma \ref{L: major arc expression}. It essentially follows from the van der Corput Lemma (e.g., \cite[pp. 332]{Ste93}) and we omit the proof.
\begin{lemma}\label{L: integral bound}
    We have
    \begin{align*}
        \abs{\int_0^K e(\xi \widetilde{P}_{[r, V]}(z))\; dz}\ll_P 1 + |\xi|^{-1/d}V_d^{-1/d}.
    \end{align*}
\end{lemma}

Combining the previous three lemmas, we can finally upper-bound the contribution of the major arcs. 
\begin{lemma}[Major arc control]\label{L: major arcs}
    For $s>d$ and $\veps\in (0, 1/10]$, we have
    \begin{align*}
        \int_\mathfrak{M}\abs{S(\xi)}^{2s}\; d\xi \ll_{P, s} \frac{K^{2s-d}}{V_d}
    \end{align*}
    as long as $K\geq V_d^2$.
\end{lemma}
We recall that the parameter $\veps$ figures implicitly in the definition of the major arcs $\mathfrak{M}$; since the left-hand side of the inequality in the lemma is non-decreasing as $\veps$ increases, the important point is that we can take $\veps$ bounded away from $0$.

\begin{proof}
We let all constants depend on $P, s$.
Applying Lemma \ref{L: major arc expression}, integrating over all of the major arcs, and changing variables $\xi\mapsto \xi + a/q$ on each $\mathfrak{M}_{a,q}$, we have  
\begin{multline*}
    \int_\mathfrak{M}\abs{S(\xi)}^{2s}\; d\xi = \Bigbrac{1_{(a,q)=(0,1)} + \sum_{\substack{1\leq a < q\leq K^\veps,\\ \gcd(a,q) = 1}}
    |S(a,q)|^{2s}}\cdot\int_{-K^{-d+\veps}}^{K^{-d+\veps}} \abs{\int_0^K e\brac{{\xi}\widetilde{P}_{[r, V]}(z)} \,dz}^{2s}\, d\xi\\
    +O\Bigbrac{1_{(a,q)=(0,1)} + \sum_{\substack{1\leq a < q\leq K^\veps,\\ \gcd(a,q) = 1}}1}(K^{2s-1-d+3\veps}).
\end{multline*}

We bound the terms in the expression above one by one, starting with  the contribution of the sums $S(a,q)$. For every $\eta>0$, we can bound
\begin{align*}
    1_{(a,q)=(0,1)} + \sum_{\substack{1\leq a < q\leq K^\veps,\\ \gcd(a,q) = 1}}|S(a,q)|^{2s}\ll_{\eta} 1 +  \sum_{\substack{1\leq a < q\leq K^\veps,\\ \gcd(a,q) = 1}}q^{\eta -2s/d}\ll_\eta 1 + K^{\veps(\eta + 2-2s/d)},
\end{align*}
where the first inequality follows from an application of Lemma \ref{L: S(a,q)} with $2\eta s$ in place of $\eta$, and the second bound follows from trivially estimating the number of $a$'s in $1\leq a < q$ by $q$ and then trivially bounding $q\leq K^\veps$. 
Importantly, this expression is $O(1)$ as long as $s>d$ and $\eta<2/d$. Since there are no further restrictions on the choice of $\eta$, we can take any $\eta$ in the range above. 

To bound the contribution of the integral, we split it over $\xi$ at $K^{-d}/V_d$, obtaining
\begin{multline*}
    \int_{-K^{-d+\veps}}^{K^{-d+\veps}} \abs{\int_0^K e\brac{\xi \widetilde{P}_{[r, V]}(z)}dz}^{2s} d\xi\\ \leq \int_{-K^{-d}/V_d}^{K^{-d}/V_d} K^{2s}\, d\xi + 2 \int_{K^{-d}/V_d}^1 \abs{\int_0^K e\brac{\xi \widetilde{P}_{[r, V]}(z)}dz}^{2s}\, d\xi.
\end{multline*}
The first term can be bounded trivially by $K^{2s-d}/V_d$. To bound the second term, we apply Lemma \ref{L: integral bound}, which gives
\begin{align*}
    \int_{K^{-d}/V_d}^1 \abs{\int_0^K e\brac{\xi \widetilde{P}_{[r, V]}(z)}dz}^{2s}\, d\xi
    &\ll \int_{K^{-d}/V_d}^1 |\xi|^{-2s/d}V_d^{-2 s/d} d\xi\ll  \frac{K^{2s-d}}{V_d},
\end{align*}
where we use Lemma \ref{L: integral bound} to bound the integrand and the assumption $s>d$ to ensure that the exponent is of the right order. 

Lastly, the error term is at most
\begin{align*}
    \ll\Bigbrac{1_{(a,q)=(0,1)} + \sum_{\substack{1\leq a < q\leq K^\veps,\\ \gcd(a,q) = 1}}1}K^{2s-1-d+3\veps}\ll K^{2s-1-d+5\veps},
\end{align*}
which can be bounded by $\ll K^{2s-d}/V_d$ since $K\geq V_d^2$ and $\veps\leq 1/10$. 
\end{proof}
We summarize the computations performed in this section in the result below.
\begin{proposition}[Even moment bound]\label{P: even moment bound}
    There exists $s_1(d)\in\N$ and $C=C(d)>0$ such that if $s\geq s_1(d)$ and $N\geq V_d^{C}$, then
        \begin{align*}
        \int_{\R/\Z}|S(\xi)|^{2s}\; d\xi \ll_{P, s} \frac{K^{2s}}{N}.
    \end{align*}
\end{proposition} 
\begin{proof}
    Let $s_1(d) = \max(s_0(d), d+1)$, where $s_0$ comes from Lemma \ref{L: minor arcs} and $d+1$ is the smallest possible $s$ allowed by Lemma \ref{L: major arcs}. We fix $s>s_1(d)$ and let all the implicit constants depend on $P, s$.
    We split $\R/\Z$ into major and minor arcs with  $\veps = 1/10$.
    Then we can use Lemma \ref{L: minor arcs} to bound the minor arc contribution as 
    \begin{align*}
        \int_{\mathfrak{m}}|S(\xi)|^{2s}\; d\xi \ll K^{2s-d-\veps/C_1} 
    \end{align*}
    where $C_1>0$ is the constant from Lemma \ref{L: minor arcs}. Since $N = K^d V_d$, this can be bounded by $\ll \frac{K^{2s}}{N}$ as long as $V_d\leq K^{\veps/C_1}$, which we can assume by taking the implicit constant in $N\geq V_d^{\Omega(1)}$ to be sufficiently large. We also implicitly use the assumption $N\geq V_d^{\Omega(1)}$ to ensure that $K^{\veps/4}\geq V_d$ so that the assumption on the coefficients of $\widetilde{P}_{[r, V]}$ in Lemma \ref{L: minor arcs} is satisfied. To bound the major arc contribution by the same final expression, we simply apply Lemma \ref{L: major arcs}.

\end{proof}

\subsection{Applications to S\'ark\"ozy-type estimates}

We are now ready to show that counting operators for S\'ark\"ozy-type configurations are controlled by the major arcs of their arguments. 

\begin{proposition}\label{P: major arc control}
    For every $C_1>0$ there exists $C_2 = C_2(C_1, P)>0$ such that the following holds. Let $\delta\in(0, 1/10)$, let $V\leq W^{C_1}$ be a nonnegative integer power of $W$, let $r\in[V]$, and let $N\geq C_2 V_d^{C_2} \delta^{-C_2}$ be an integer. Let $f_0,f_1:\Z\to\C$ be $1$-bounded functions supported on $[N]$. If
    \begin{align}\label{E: Sarkozy input}
        \abs{\sum_x \E_{z\in [K]}f_0(x)f_1(x+\widetilde{P}_{[r, V]}(z))}\geq \delta N,
    \end{align}
    then for every $i=0,1$ there exist phases $\xi_i\in\R/\Z$ and positive integers $q_i\ll_{C_1, P} \delta^{-O_{C_1, P}(1)}$ such that $\norm{q_i\xi_i}_{\R/\Z}\ll_{C_1, P}\delta^{-O_{C_1, P}(1)}/N$ and $|\widehat{f_i}(\xi_i)|\gg_{C_1, P} \delta^{O_{C_1, P}(1)}N$.
\end{proposition}
\begin{proof}
    We fix $C_1>0$ and let all constants depend on $P$.
    We prove the bound only for $i=0$, since one can then deduce the result for $i=1$ by shifting $x \mapsto x-\widetilde{P}_{[r, V]}(z)$ and exchanging the roles of $f_0,f_1$. Our first goal is to show that the largeness of \eqref{E: Sarkozy input} implies the largeness of some Fourier coefficient of $f_0$, and then we use the fraction comparison argument to deduce that this Fourier coefficient must be a major arc.
    
    Define the mutliset $E:= \{\widetilde{P}_{[r, V]}(z):\; z\in[K]\}$, so that our assumption can be expressed as
    \begin{align*}
        \abs{\sum_{x,z}f_0(x) f_1(x+z)1_E(z)}\geq \delta NK.
    \end{align*}
  Note that $F = \widehat{1_E}$.  Now, Fourier-expanding, extracting the largest Fourier coefficient, and applying the triple H\"older inequality (with exponents $\frac{2s}{s-1}, 2, 2s$) followed by Parseval's Identity,
    we get
    \begin{align*}
        \delta N K &\leq \norm{\widehat{f_0}}_\infty^{1/s}\cdot \int_{\R/\Z}|\widehat{f_0}(\xi)|^{1-1/s}\cdot|\widehat{f_1}(\xi)|\cdot|F(\xi)|\; d\xi\\
        &\leq \norm{\widehat{f_0}}_\infty^{1/s}\cdot \norm{f_0}^{1-1/s}_2 \cdot \norm{f_1}_2 \cdot \brac{\int_{\R/\Z}|F(\xi)|^{2s}\; d\xi}^{1/(2s)}
    \end{align*}
    for any $s\geq 2$.
    We can trivially bound $\norm{f_0}^{1-1/s}_2 \cdot \norm{f_1}_2\leq N^{(2s-1)/(2s)}$, and Proposition \ref{P: even moment bound} gives
    \begin{align*}
        \int_{\R/\Z}|S(\xi)|^{2s}\; d\xi \ll \frac{K^{2s}}{N}
    \end{align*}
    for any $s$ sufficiently large in terms of $d$.
    Combining all these bounds together, we get the bound $\norm{\widehat{f_i}}_\infty\gg \delta^{O(1)} N$, i.e. we have shown that the S\'ark\"ozy configuration is controlled by Fourier analysis. 

    Now, let $\Tilde{f}_0 = \E_{z\in{K}}\overline{f_1}(x+\widetilde{P}_{[r, V]}(z))$. Stashing, we can replace $f_0$ by $\tilde{f}_0$ in \eqref{E: Sarkozy input} at the cost of squaring $\delta$ in the lower bound. Together with the just proved Fourier control, this gives us 
    \begin{align*}
        \abs{\sum_{x}\E_{z\in[K]}\overline{f_1}(x+\widetilde{P}_{[r, V]}(z)) e(-\xi x)}\gg \delta^{O(1)}N
    \end{align*}
    for some $\xi\in\R/\Z$. Shifting $x\mapsto x-\widetilde{P}_{[r, V]}(z)$, splitting the sums over $x$ and $z$, and using $\norm{f_1}_1\leq N$, we get that
    \begin{align*}
        \abs{\sum_x f_1(x) e(\xi x)}\gg \delta^{O(1)}N\quad \textrm{and}\quad \abs{\E_{z\in[K]}e(-\xi \widetilde{P}_{[r, V]}(z))}\gg \delta^{O(1)}.
    \end{align*}
    It remains to show that $\xi$ is a major arc. By Weyl's inequality (e.g. \cite[Lemma 4.4]{GT12}), the second inequality can only hold if there exists a positive integer $q\ll \delta^{-O(1)}$ for which $\norm{q\xi \widetilde{P}_{[r, V]}}_{C^\infty[K]}\ll \delta^{-O(1)}$. The result then follows from Lemma \ref{lem:composing-polynomials-2}.
\end{proof}

The following two corollaries make precise the heuristic \eqref{E: Sarkozy heuristic}.
\begin{corollary}\label{C: 1-dim Sarkozy}
        For every $C_1>0$ there exists $C_2 = C_2(C_1, P)>0$ such that the following holds. Let $\delta\in(0, 1/10)$, let $V\leq W^{C_1}$ be a nonnegative integer power of $W$, let $r\in[V]$, and let $N\geq C_2 V_d^{C_2}\delta^{-C_2}$ be an integer. If $f_0, f_1:\Z\to\C$ are $1$-bounded functions supported on $[N]$ satisfying
    \begin{align*}
        \abs{\sum_x \E_{z\in[K]}f_0(x) f_1(x+\widetilde{P}_{[r, V]}(z))}\geq \delta N,
    \end{align*}
    then for each $i\in\{0,1\}$ there exists a positive integer  $q_i\ll_{C_1, P} \delta^{-O_{C_1, P}(1)}$ such that 
    \begin{align*}
        \E_{u\in [\pm N]}\abs{\E_{z\in[N']} f_i(u + q_i z)}\gg_{C_1, P} \delta^{O_{C_1, P}(1)} \quad\textrm{for\; all}\quad 1\leq N'\ll_{C_1, P} \delta^{O_{C_1, P}(1)} N.
    \end{align*}
\end{corollary}
\begin{proof}
    We fix $C_1>0$ and let all constants depend on $P$. We just proof the claim for $i=1$, as the other case follows by symmetry. By Proposition \ref{P: major arc control}, there exist a phase $\xi\in\R/\Z$ and a positive integer $q\ll \delta^{-O(1)}$ such that $\norm{q\xi}_{\R/\Z}\ll\delta^{-O(1)}/N$ 
    and 
    \begin{align*}
        \abs{\sum_x f_1(x)e(\xi x)}\gg \delta^{O(1)}N.
    \end{align*}
    Set $N'_0 \asymp \delta^C N$ for some $C>1$ to be chosen later, and let $N'\in[N'_0]$ be arbitrary. Introducing an extra averaging over arithmetic progressions of length $N'$ and common difference $q$ gives
    \begin{align*}
        \abs{\sum_{u}\E_{z\in[N']} f_1(u+qz)e(\xi (u+qz))}\gg \delta^{O(1)}N.
    \end{align*} 
    The triangle inequality gives
    $$\sum_u \abs{\E_{z\in[N']} f_1(u+qz)e(\xi qz)}\gg_d \delta^{O_d(1)}N.$$
    We have
    $$|e(\xi qz)-1| \leq |z| \cdot \norm{\xi q}_{\R/\Z}\ll \delta^{-O(1)}N'/N$$
    for all $z \in [N']$, so, with $C$ sufficiently large, we obtain
    $$\sum_u \abs{\E_{z\in[N']} f_1(u+qz)}\gg \delta^{O(1)}N.$$
    Finally, since $f_1$ is supported on $[N]$, there is no contribution from $u$'s outside of $[\pm N]$, and we get

    \begin{align*}
        \E_{u\in [\pm N]} \abs{\E_{z\in[N']} f_1(u+qz)e(\xi u)}\gg \delta^{O(1)},
    \end{align*}      
    as desired.
\end{proof}

We now deduce a multidimensional version of this corollary.

\begin{corollary}\label{C: Sarkozy}
        For every $C_1>0$ there exists $C_2 = C_2(C_1, P)>0$ such that the following holds. Let $\delta\in(0, 1/10)$, let $V\leq W^{C_1}$ be a nonnegative integer power of $W$, let $r\in[V]$, and let $N\geq C_2 V_d^{C_2}\delta^{-C_2}$ be an integer. Let $D\in\N$, and let $\bv\in\Z^D$ have coordinates of size $O(1)$. If $f_0, f_1:\Z^D\to\C$ are $1$-bounded functions supported on $[N]^D$ satisfying 
    \begin{align*}
        \abs{\sum_\bx \E_{z\in[K]}f_0(\bx) f_1(\bx+\bv \widetilde{P}_{[r, V]}(z))}\geq \delta N^D,
    \end{align*}
    then for each $i\in\{0,1\}$ there exists a positive integer
    $q_i\ll_{C_1, D,P} \delta^{-O_{C_1, D,P}(1)}$ such that 
    \begin{align*}
        \sum_\bx\abs{\E_{z\in[N']} f_i(\bx + q_i\bv z)}\gg_{C_1, D, P} \delta^{O_{C_1, D, P}(1)}N^{D} \quad\textrm{for\; all}\quad 1\leq N'\leq \delta^{O_{C_1, D,P}(1)} N.
    \end{align*}
\end{corollary}
\begin{proof} We fix $C_1>0$ and let all constants depend on $D$ and $P$. We just proof the claim for $i=1$, as the other case follows by symmetry.
    Introducing an extra averaging over $\bv\cdot[N]$ gives
    \begin{align*}
        \sum_\bx\abs{\sum_{y\in[N]} \E_{z\in[K]}f_0(\bx+\bv y) f_1(\bx+\bv(y+ \widetilde{P}_{[r, V]}(z)))}\geq \delta N^{D+1}.
    \end{align*}    
    Since the coordinates of $\bv$ have size $O(1)$, the functions $\bx\mapsto f_0(\bx + \bv y)$ vanish whenever $|\bx|\gg N$ as $y$ ranges over $[N]$. Pigeonholing in $\bx$, we can therefore find $\Omega(\delta N^D)$ values $\bx\in\Z^D$ for which
    \begin{align}\label{E: pigeonholing in bx}
        \abs{\sum_{y\in[N]} \E_{z\in[K]}f_0(\bx+\bv y) f_1(\bx+\bv(y+ \widetilde{P}_{[r, V]}(z)))}\gg \delta N,
    \end{align}
    and so the previous corollary gives us positive integers $q_\bx\ll\delta^{-O(1)}$ and $\delta^{O(1)}N\ll N'_{0,\bx}\leq N$
    such that 
    \begin{align*}
        \E_{u\in [\pm N]}\abs{\E_{z\in[N'_\bx]} f_1(\bx + \bv(u+ q_\bx z))}\gg \delta^{O(1)}
    \end{align*}
    for all $N'_\bx\in[N'_{0,\bx}]$. We set $N'_0 := \min_\bx N'_{0,\bx}$ to eliminate the dependence on $\bx$, where the minimum is taken over the  $\Omega(\delta N^D)$ good values $\bx\in\Z^D$. Since the lower bounds on $N'_{0,\bx}$ are independent on $\bx$, we deduce that $N'_0\gg \delta^{O(1)}N$.
    Similarly, since there are only $O(\delta^{-O(1)})$ choices of $q_\bx$, the pigeonhole principle produces some $q\ll \delta^{-O(1)}$ for which
    \begin{align*}
        \sum_\bx\E_{u\in [\pm N]}\abs{\E_{z\in[N']} f_1(\bx + \bv(u+ q z))}\gg \delta^{O(1)}N^{D}
    \end{align*}
    for all $N' \leq N'_0$.  The result follows upon shifting $\bx \mapsto \bx - \bv u$.
\end{proof}

\section{Fourier uniformity of difference of weights}\label{A: Fourier uniformity}
The purpose of this section is to show that the weight $\nu^*$ from the proof of Proposition \ref{P: Lambda*} is highly Fourier-uniform on arithmetic progressions of difference $V$.  Our argument is a slight modification of the proof of \cite[Lemma B.8]{PSS23}.  Recall that the leading term of $\widetilde P(z)$ is $W_d z^d$, where $W_d=\beta_d \beta_1^{d-2} W^{d-1}$, and our weights are
\begin{gather*}
    \tilde{\nu}(z)= \frac{N}{(N/W_d)^{1/d}}1_{\widetilde P([(N/W_d)^{1/d}])}(z) = N^{(d-1)/d}W_d^{1/d} \sum_{z' \in [(N/W_d)^{1/d}]}1_{\widetilde P(z')}(z),\\
 \nu(z)=d^{-1} 1_{[N]}(z) (N/(z+1))^{(d-1)/d}, \quad \text{and} \quad  \nu^*(z) = \tilde{\nu}(z) - \nu(z).
\end{gather*}

The following proposition captures the main estimate.

\begin{proposition}\label{prop:nu-uniformity}
For every $C_1>0$ there exist $C_2 = C_2(C_1, P),\; c=c(C_1, P)>0$ such that the following holds. If $V\leq W^{C_1}$ is a nonnegative integer power of $W$ and $N\geq W^{C_2}$, then 
$$\max_{r\in[V]}\max_{\theta \in \mathbb{R}/\mathbb{Z}}\abs{\sum_{z \in \mathbb{Z}}e(\theta z)\nu^*(Vz+r)} \ll_{C_1,P} \frac{N}{V w^c}.$$

\end{proposition}
Notice that the trivial bound (from the triangle inequality) is $\ll N/V$; our gain is in the (relatively modest) $w^{-c}$ term.

\begin{proof}
We fix $C_1>0$ and let all the constants depend on $C_1$ and $P$.
Hensel's Lemma and the identity $\widetilde P(z) \equiv z \!\!\pmod{W}$ ensure that for each $r\in[V]$ there is a unique $r'\in[V]$ such that $\widetilde P(r') \equiv r \!\!\pmod{V}$. It follows that $\widetilde P(z') \equiv r \!\!\pmod{V}$ if and only if $z'\equiv r' \!\!\pmod{V}$, so
\begin{align*}
    \sum_{z}e(\theta z) \Tilde{\nu}(Vz+r) &=  N^{(d-1)/d}W_d^{1/d} \sum_z e(\theta z)\sum_{z' \in [(N/W_d)^{1/d}]}1_{\widetilde P(z')}(Vz+r)\\
    &= N^{(d-1)/d}W_d^{1/d} \sum_{z' \in [(N/W_d)^{1/d}V^{-1}]}e\brac{\theta \frac{\widetilde P(Vz'+r')-r}{V}}.
\end{align*}
Substituting $z$ for $z'$ and $V\theta$ for $\theta$, and then multiplying both $\Tilde{\nu}$ and $\nu$ by $e(\theta r)$, we find that it suffices to prove the inequality
\begin{multline*}
 \left|N^{(d-1)/d}W_d^{1/d} \sum_{z \in [(N/W_d)^{1/d}V^{-1}]}e(\theta \widetilde P(Vz+r'))\right.\\
 \left.-d^{-1}\sum_{z \in [N/V]} (N/(Vz+r))^{(d-1)/d}e(\theta (V z+r+1))\right|\ll \frac{N}{V w^c}   
\end{multline*}
for all $r \in [V]$ and $\theta\in\R/\Z$.

We will require different arguments when $V\theta$ is in a minor arc, when it is in the major arc around $0$, and when it is in a major arc not around $0$.\footnote{These major and minor arcs are different from the major and minor arcs in Appendix \ref{A: Sarkozy}.}

We first analyze when the contribution of $\nu$ can be large. By summation by parts, we have that
$$d^{-1}\abs{\sum_{z \in [N/V]} (N/(Vz+r+1))^{(d-1)/d}e(\theta(Vz+r))}$$
equals
\begin{align*}
 &d^{-1}\abs{\sum_{z \in [N/V]} (N/(Vz+r+1))^{(d-1)/d}e(\theta Vz)}\\
 &\qquad\qquad\ll \sum_{t \in [N/V-1]} \left(\left(\frac{N}{Vt+r+1}\right)^{(d-1)/d}-\left(\frac{N}{V(t+1)+r+1}\right)^{(d-1)/d}\right) \abs{\sum_{z \in [t]}e(V\theta z)}\\
 &\qquad\qquad\qquad\qquad\qquad\qquad\qquad\qquad\qquad\qquad\qquad\qquad\qquad\qquad\qquad+\abs{\sum_{z \in [N/V]}e(V\theta z)}\\
 &\qquad\qquad\ll \sum_{t \in [N/V-1]} (N/V)^{(d-1)/d} t^{(1-2d)/d} \min \left\{ t, \frac{1}{\norm{V\theta}_{\mathbb{R}/\mathbb{Z}}} \right\}+\min \left\{N/V,\frac{1}{\norm{V\theta}_{\mathbb{R}/\mathbb{Z}}}\right\}\\
 &\qquad\qquad\ll \frac{(N/V)^{(d-1)/d}}{\norm{V\theta}_{\mathbb{R}/\mathbb{Z}}^{1/d}} + \min \left\{N/V,\frac{1}{\norm{V\theta}_{\mathbb{R}/\mathbb{Z}}}\right\},
\end{align*}
where in the second-to-last line we split the sum at $t = 1/\|V\theta\|_{\mathbb{R}/\mathbb{Z}}$.  This quantity is $\ll N/(V w^c)$ unless $\norm{V\theta}_{\mathbb{R}/\mathbb{Z}}<V w^{cd}/N$.

We now analyze when the contribution of $\tilde{\nu}$ can be large.  By \cite[Proposition 4.3]{GT12} and \cite[Lemma A.10]{Leng23b}, we have
$$\abs{N^{(d-1)/d}W_d^{1/d} \sum_{z \in [(N/W_d)^{1/d}V^{-1}]}e(\theta \widetilde P(Vz+r'))} \ll \frac{N}{Vw^c}$$
unless there is some positive integer $q \ll w^{O(c)}$ such that
\begin{align}\label{E: Fourier uniformity estimate}
    \norm{q\theta\widetilde P(V \cdot +r')}_{C^\infty [(N/W_d)^{1/d}V^{-1}]} \ll w^{O(c)}.
\end{align}
Let $R(z):= q\theta(Vz + \widetilde P(r'))$, so that $R\circ \widetilde P_{[r',V]}(z) = q\theta\widetilde P(Vz +r')$. Setting $T=N/V$ and recalling that $V_d = W_d V^{d-1}$, we can write the estimate \eqref{E: Fourier uniformity estimate} as
\begin{align*}
    \norm{R\circ \widetilde P_{[r',V]}}_{C^\infty [(T/V_d)^{1/d}]} \ll w^{O(c)}.
\end{align*}
Lemma \ref{lem:composing-polynomials-2} then gives
\begin{align*}
    \norm{qV\theta}_{\R/\Z}\ll \frac{V w^{O(c)}}{N}.
\end{align*}
Hence the $\tilde{\nu}$-contribution is $\ll N/(V w^c)$ unless there is some positive integer $q \ll w^{O(c)}$ such that $\norm{qV\theta}_{\mathbb{R}/\mathbb{Z}} \ll Vw^{O(c)}/N$.

Comparing the outputs of the previous two paragraphs, we see that the conclusion of the proposition holds in the ``minor arc'' case where there is no $q \ll w^{O(c)}$ such that $\norm{qV\theta}_{\mathbb{R}/\mathbb{Z}} \ll Vw^{O(c)}/N$.  Henceforth, we assume that we are in the ``major arc'' case where there is some $q \ll w^{O(c)}$ such that $\norm{qV\theta}_{\mathbb{R}/\mathbb{Z}} \leq Vw^{O(c)}/N$.  In this case, we write
$$V\theta=a/q+\theta^*,$$
where $a,q$ are coprime integers with $1\leq q \ll w^{O(c)}$, and 
$$|\theta^*| \ll Vw^{O(c)}/N\leq (W_d/N)^{1-1/(4d)}V^{d-1/4}$$ (say) since $N$ bounded from below be a sufficiently large power of $W$ that is allowed to depend on $d$.
Since the implicit constant in the term $w^{O(c)}$ depends only on $d$, choosing $c$ sufficiently small (depending on $d$) guarantees that $q<w$.

First, assume that $q>1$.  The $\nu$-contribution is $\ll N/(Vw^c)$.   To handle the $\tilde{\nu}$-contribution, we use Lemma \ref{L: major arc expression} with $Q(z) =\widetilde P_{[r',V]}(z)$, $K = (N/W_d)^{1/d}V^{-1}$, and $\veps = 1/4$ to express
\begin{multline*}
\sum_{z \in [(N/W_d)^{1/d}V^{-1}]} e(\theta\widetilde P(Vz+r'))\\
=e(\theta \widetilde P(r'))\E_{u\in [q]} e((a/q)\widetilde P_{[r',V]}(u)) \int_0^{(N/W_d)^{1/d}V^{-1}} e(\theta^* \widetilde P_{[r',V]}(x)) \,dx\\
+O((N/W_d)^{1/(2d)}V^{-1/2}).
\end{multline*}

Since $q$ is smaller than $w$, it shares a prime factor with $W$, and so Lemma \ref{L: S(a,q)} tells us that the average $\E_{u\in [q]} e((a/q)\widetilde P_{[r',V]}(u))$ vanishes.
Hence
$$N^{(d-1)/d}W_d^{1/d} \sum_{z \in [(N/W_d)^{1/d}V^{-1}]}e(\theta \widetilde P(Vz+r'))\ll N/(Vw^c)$$
(with room to spare) by the assumption that $N$ is larger than a suitable power of $V$.

Finally, assume that $q=1$, in which case $a=0$ and $\norm{V\theta}_{\mathbb{R}/\mathbb{Z}} \ll Vw^{O(c)}/N$.  We will break the sum in the $ \nu$-contribution into intervals and show that each one approximately cancels out with a single summand from the $\tilde\nu$-contribution.  Since
$$\widetilde P_{[r',V]}(z)-W_d V^{d-1} z^d \ll V^{O(1)} z^{d-1}$$
and
\begin{align*}
    \widetilde P(r') - r = Vu\quad \textrm{for some integer  } |u|\ll W_d V^{d-1},
\end{align*}
we have
$$\abs{\sum_{z \in [(N/W_d)^{1/d}V^{-1}]} (e(\theta\widetilde P(Vz+r'))-e(\theta (W_d V^{d} z^d + r)))} \ll V^{O(1)}.$$
This lets us simplify the $\tilde{\nu}$-contribution by removing lower-order terms, and it suffices to show that
\begin{multline*}
\left|N^{(d-1)/d}W_d^{1/d} e(\theta r)\sum_{z \in [(N/W_d)^{1/d}V^{-1}]} e(\theta W_d V^{d} z^d)\right.\\
\left.-d^{-1}\sum_{z \in [N/V]} (N/(Vz+r+1))^{(d-1)/d}e(\theta(Vz+r))\right|\ll \frac{N}{V w^c}.  
\end{multline*}
We break the second sum into intervals of the form $$I_z = [A_z, B_z]:=[W_d V^{d-1} z^d,W_d V^{d-1} (z+1)^d-1]$$
with length $|I_z|=W_d V^{d-1}dz^{d-1}(1+O(z^{-1}))$ for $z\in[(N/W_d)^{1/d}V^{-1}]$. For $t\in I_z$, we bound
\begin{align*}
    \abs{\brac{\frac{N}{Vt+r+1}}^{(d-1)/d} - \brac{\frac{N}{V A_z}}^{(d-1)/d}}&\ll \abs{\frac{N}{Vt+r+1} - \frac{N}{V A_z}} \cdot\brac{\frac{N}{V A_z}}^{-1/d}\\
    &\leq \abs{\frac{V A_z - Vt -r-1}{Vt+r+1}} \cdot\brac{\frac{N}{V A_z}}^{(d-1)/d}\\
    &\ll \frac{|I_z|}{A_z}\cdot\brac{\frac{N}{V A_z}}^{(d-1)/d} \ll z^{-1}\brac{\frac{N}{V A_z}}^{(d-1)/d}
\end{align*}
and
\begin{align}\label{E: comparison of sums}
    \abs{e(\theta \cdot Vt) - e(\theta\cdot VA_z)}\leq\norm{V\theta}_{\R/\Z} \cdot |I_z|\ll \frac{V^{O(1)}}{N^{1/d}}\ll N^{-1/(2d)}
\end{align}
as long as $N\gg V^{\Omega(1)}$.  Summing over $t \in I_z$, we obtain the key estimate
\begin{align*}
& d^{-1}\sum_{t \in I_z} (N/(Vt+r+1))^{(d-1)/d}e(\theta(Vt+r))\\
&\qquad \qquad =e(\theta r) d^{-1}\sum_{t \in I_z} (N/(VA_z))^{(d-1)/d}e(\theta V A_z )(1+O(N^{-1/2d}+z^{-1}))\\
&\qquad \qquad =e(\theta r) W_d V^{d-1} z^{d-1} (N/(W_d V^{d} z^d))^{(d-1)/d}e(\theta W_d V^{d} z^d))(1+O(N^{-1/2d}+z^{-1}))\\
&\qquad \qquad =e(\theta r) N^{(d-1)/d} W_d^{1/d}e(\theta W_d V^{d} z^d) (1+O(N^{-1/2d}+z^{-1})).
\end{align*}

The lemma follows by summing over $z \in [(N/W_d)^{1/d}V^{-1}]$ and noting that
$$\abs{\sum_{z \in [(N/W_d)^{1/d}V^{-1}]}N^{(d-1)/d} W_d^{1/d}(N^{-1/2d}+z^{-1})} \ll \frac{N}{V w^c},$$
again with room to spare since $W$ is sufficiently small compared to $N$.
\end{proof}

\section{Supersaturation}\label{A: supersaturation}
The purpose of this short section is to prove a supersaturation result for corners.  We apply Varnavides' now-standard trick from \cite{Var59} to Shkredov's result on corners \cite{Sh06b}.

\begin{theorem}[\cite{Sh06b}]\label{thm:shkredov}
There exists a constant $c>0$ such that the following holds: If $A \subseteq [N]^2$ has size $\delta N^2$ for some $\delta\in(0, 1/10)$ and $N \geq \exp \exp (\delta^{-c})$, then $A$ contains a nontrivial corner.
\end{theorem}

\begin{proposition}\label{prop:supersat}
There exists a constant $c'>0$ such that the following holds: If $A \subseteq [N]^2$ has size $\delta N^2$ for some $\delta\in(0, 1/10)$ and $N \geq \exp \exp (\delta^{-c'})$, then $A$ contains at least $(\exp(- \exp (\delta^{-c'})))N^3$ nontrivial corners.
\end{proposition}

\begin{proof}
Let $M:=\exp \exp ((\delta/2)^c)$, where $c$ is the constant from Theorem~\ref{thm:shkredov}.  For each $1 \leq d \leq N/(3M)$ and $\bu \in [-d(M-1),N-1]^2$, define the set
$$\lambda(d,\bu):= A \cap (\bu+d \cdot [M]^2).$$
For fixed $d$, each element of $A$ is contained in exactly $M^2$ of the sets $\lambda(d,\bu)$, so $$\sum_\bu |\lambda(d,\bu)|=M^2|A|=\delta M^2 N^2.$$
Say that the pair $(d,\bu)$ is \emph{good} if $|\lambda(d,\bu)| \geq (\delta/2)M^2$, and let $X(d)$ denote the number of $\bu$'s such that $(d,\bu)$ is good.  We can bound
$$\delta M^2 N^2=\sum_\bu |\lambda(d,\bu)| \leq X(d)M^2+(N+d(M-1))^2 \cdot (\delta/2)M^2,$$
and so (for instance)
$$X(d) \geq \delta N^2-(\delta/2)(N+N/3)^2=(\delta/9)N^2.$$
Hence the total number of good pairs $(d,\bu)$ (for all $d$) is at least $(\delta/(27M))N^3$.  Theorem~\ref{thm:shkredov} tells us that if $(d,\bu)$ is good, then $\lambda(d,\bu)$ contains a nontrivial corner.  Finally, each nontrivial corner is contained in at most $M^3$ such sets $\lambda(d,\bu)$ (there are at most $M$ choice of $d$ and then at most $M^2$ choices of $\bu$), so we conclude that the number of nontrivial corners in $A$ is at least
$$\frac{\delta}{27M^4} N^3 \geq (\exp(- \exp (\delta^{-c'})))N^3$$
with a suitable choice of $c'>0$.
\end{proof}

  \bibliography{library}
\bibliographystyle{plain}

\end{document}